\documentclass[a4paper,10pt]{article}

\RequirePackage[T1]{fontenc}
\RequirePackage{fix-cm}
\usepackage{amsmath,amsfonts,amsthm}
\usepackage{graphicx} 
\usepackage[unicode,final]{hyperref} 
\usepackage{cite}
\usepackage{caption}
\usepackage{mathtools}
\usepackage{newtxtext}
\usepackage[varvw]{newtxmath}
\usepackage{bm}
\usepackage[T1]{fontenc}
\usepackage[utf8]{inputenc}
\usepackage{color}
\newcommand{\fz}{\frac}
\newcommand{\prz}[2]{ \frac{\partial{#1}}{\partial{#2}} }
\newcommand{\pz}{\partial}
\newcommand{\lA}{\langle}
\newcommand{\rA}{\rangle}
\newcommand{\rarrow}{\rightarrow}

\newtheorem{Hyp}{Hypothesis}
\newtheorem{Def}{Definition}
\newtheorem{Lem}{Lemma}
\newtheorem{Prop}{Proposition}
\newtheorem{Thm}{Theorem}
\newtheorem{Cor}{Corollary}
\newtheorem{Rmk}{Remark}
\newtheorem{Ex}{Example}
\allowdisplaybreaks[4]
\newcommand{\R}{\mathbb{R}}
\newcommand{\N}{\mathbb{N}}

\newcommand{\defeq}{\coloneqq}
\newcommand{\eqdef}{\eqqcolon}
\newcommand{\cK}{\color{black}}
\newcommand{\cB}{{}}
\newcommand{\cR}{{}}

\makeatletter
\renewenvironment{proof}[1][\proofname]{\par
  \normalfont
  \topsep6\p@\@plus6\p@ \trivlist
  \item[\hskip\labelsep{\itshape #1}\@addpunct{\itshape.}]\ignorespaces
  }{\qed\endtrivlist}
\renewcommand{\proofname}{Proof}
\makeatother
\usepackage{soul}

\usepackage{booktabs}
\usepackage{mathtools}
\providecommand{\keywords}[1]{\textit{Keywords:} #1}

\begin{document}
\title{A mass-preserving two-step Lagrange--Galerkin scheme for convection-diffusion problems}
%
%
%
%
\author{
  Kouta~Futai$^{1}$,
  Niklas~Kolbe$^{2}$\footnote{Corresponding author}\ ,
  Hirofumi~Notsu$^{3}$
  and
  Tasuku~Suzuki$^{4}$}
%
\date{
  \small
  $^1$Panasonic System Networks R\&D Lab. Co., Ltd. \smallskip\\
  $^2$Institute of Geometry and Practical Mathematics, RWTH Aachen University \smallskip\\
  $^3$Faculty of Mathematics and Physics, Kanazawa University \smallskip\\
  $^4$Division of Mathematical and Physical Sciences, Kanazawa University
  \smallskip\\
  {\tt futai.k.1275@gmail.com},\ \
  {\tt kolbe@igpm.rwth-aachen.de},\ \
  {\tt notsu@se.kanazawa-u.ac.jp},\ \
  {\tt suzu-2504@stu.kanazawa-u.ac.jp}
}
\maketitle

\begin{abstract}
A mass-{\cR preserving two-step} Lagrange--Galerkin scheme of second order in time for convection-diffusion problems is presented, and convergence with optimal error estimates is proved in the framework of $L^2$-theory.
The introduced scheme maintains the advantages of the Lagrange--Galerkin method, i.e., CFL-free robustness for convection-dominated problems and a symmetric and positive coefficient matrix resulting from the discretization.
In addition, the scheme conserves the mass on the discrete level {\cR if the involved integrals are computed exactly}.
Unconditional stability and error estimates of second order in time are proved by employing two new key lemmas on the truncation error of the material derivative in conservative form and on a discrete Gronwall inequality for multistep methods.
The mass-{\cR preserving} property is achieved by the Jacobian multiplication technique introduced by Rui and Tabata in 2010, and the accuracy of second order in time is obtained based on the idea of the multistep Galerkin method along characteristics originally introduced by Ewing and Russel in 1981.
For the first time step, the mass-{\cR preserving} scheme of first order in time by Rui and Tabata in 2010 is employed, which is efficient and does not cause any loss of convergence order in the $\ell^\infty(L^2)$- and $\ell^2(H^1_0)$-norms.
For the time increment~$\Delta t$, the mesh size~$h$ and a conforming finite element space of polynomial degree~$k \in \N$, the convergence order is of $O(\Delta t^2 + h^k)$ in the $\ell^\infty(L^2)\cap \ell^2(H^1_0)$-norm and of $O(\Delta t^2 + h^{k+1})$ in the $\ell^\infty(L^2)$-norm if the duality argument can be employed.
Error estimates of $O(\Delta t^{3/2}+h^k)$ in discrete versions of the $L^\infty(H^1_0)$- and $H^1(L^2)$-norm are additionally proved.
Numerical results confirm the theoretical convergence orders in one, two and three dimensions.
\par
\par
\end{abstract}
%
\section{Introduction}\label{sec:intro}
%
The convection-diffusion equation is one of the important equations in flow problems, as it is considered a simplification of the Navier--Stokes equations.
To deal with the equation especially in convection-dominant cases, nowadays, many finite element schemes have been proposed and analyzed, e.g., 
upwind methods~\cite{BabTab-1981,BraBurJohLub-2007,BroHug-1982,HugFraMal-1987,Joh-1987,Tab-1977}, characteristics(-based) methods~\cite{BenBer-2012_part1,BenBer-2012_part2,BermejoSaavedra2012,ChrWal-2008,ColeraCarpioBermejo2020,ColeraCarpioBermejo2021,DouRus-1982,EwiRus-1981,EwiRusWhe-1983,RuiTab-2002,RuiTab-2010,Pir-1982,PirTab-2010,Pir-1989,TabUch-2016-CD} and so on.
The Lagrange--Galerkin method (also called characteristic(-curve) finite element method or Galerkin-characteristics method) belongs to the latter group and is a finite element method based on the method of characteristics, where the idea is to consider the trajectory of a fluid particle and discretize the material derivative along this trajectory.
It is known that the Lagrange--Galerkin method has many advantages including robustness for convection-dominated problems without needing any stabilization parameters, symmetry of the resulting coefficient matrix, and no requirement of the so-called CFL condition, which enables the use of large time increments.
Hence, the Lagrange--Galerkin method has also been applied to other equations, e.g., the Oseen/Navier--Stokes/viscoelastic/natural convection equations, cf.~\cite{BenBer-2011,BermejoGalanSaavedra2012,BMMR-1997,N-2008-JSCES,NT-2009-JSC,NT-2008-JSIAM,NT-2015-JSC,NT-2016-M2AN,LMNT-Peterlin_Oseen_Part_I,LMNT-Peterlin_Oseen_Part_II,Sul-1988} and references therein.
\par
Some Lagrange--Galerkin schemes of second order in time for convection-diffusion problems have already been proposed, including single step methods~\cite{BenBer-2012_part1,BenBer-2012_part2,RuiTab-2002} and multistep methods~\cite{EwiRus-1981,BermejoSaavedra2012}.
However, in general, the mass-{\cR preserving} property is often not satisfied by Lagrange--Galerkin methods.
Recently mass-{\cR preserving} Lagrange--Galerkin schemes for {\cR convection-diffusion problems in conservative form and hyperbolic conservation laws, i.e., pure convection problems in conservative form, with arbitrary orders in time and space have been proposed by Colera et al.~\cite{ColeraCarpioBermejo2020,ColeraCarpioBermejo2021}} but error estimates are not yet given.
About a decade ago, Rui and Tabata~\cite{PirTab-2010} has proposed a mass-{\cR preserving} Lagrange--Galerkin scheme of first order in time for convection-diffusion problems by a Jacobian multiplication technique and proved error estimates of first order in time.
To the best of our knowledge, however, there are no Lagrange--Galerkin schemes of second order in time having both, a mass-{\cR preserving} property and error estimates.
\par
In this paper, we propose a Lagrange--Galerkin scheme of second order in time for convection-diffusion problems and prove its mass-{\cR preserving} property and error estimates.
Stability and convergence with optimal error estimates are proved in the framework of $L^2$-theory.
We devise the scheme based on  two ideas; one is the multistep (two-step) Galerkin method along characteristics by Ewing and Russel~\cite{EwiRus-1981}, and the other one is the Jacobian multiplication technique by Rui and Tabata~\cite{RuiTab-2010}.
To find the numerical solution at time step~$n$, we employ two Jacobians for the time steps~$n-1$ and~$n-2$.
The Jacobians are of the forms, $1 - \Delta t (\nabla\cdot u^n) + O(\Delta t^2)$ and $1 - 2\Delta t (\nabla\cdot u^n) + O(\Delta t^2)$, respectively, where $\Delta t$ is a time increment and $u^n$ is the velocity at time step~$n$. For this reason it is not obvious that our scheme is of second order in time and that the mass-{\cR preserving} property is satisfied.
We, therefore, prove these properties in this paper.
As two-step methods require solutions at two prior time steps, we propose to employ the mass-{\cR preserving
} Lagrange--Galerkin scheme of first order in time by Rui and Tabata~\cite{RuiTab-2010} for the first time step.
This construction is efficient and does not cause any loss of convergence order in the $\ell^\infty(L^2)$- and $\ell^2(H^1_0)$-norms.
\par
The main results for our scheme including the construction of the solution at the first time step are as follows.
(i)~The mass-{\cR preserving} property is proved, cf.~Theorem~\ref{thm:mass-conservation}.
(ii)~Stability in $\ell^\infty(L^2)\cap\ell^2(H^1_0)$ and $\ell^\infty(H^1_0)$ is proved, cf.~Theorem~\ref{thm:stability}.
(iii)~An error estimate of~$O(\Delta t^2 + h^k)$ in the $\ell^\infty(L^2)\cap \ell^2(H^1_0)$-norm is proved, where $h$ is the mesh size in space and $k \in \N$ is the polynomial degree of a conforming finite element space for the numerical solution, cf.~Theorem~\ref{thm:error_estimates}-(i).
(iv)~An error estimate of~$O(\Delta t^2 + h^{k+1})$ in the $\ell^\infty(L^2)$-norm is proved under the assumption that the duality argument can be employed, cf.~Theorem~\ref{thm:error_estimates}-(ii).
Furthermore, in Theorem~\ref{thm:error_estimates}-(i), we prove an error estimate of~$O(\Delta t^{3/2} + h^k)$ in a discrete version of the $L^\infty(H^1_0)\cap H^1(L^2)$-norm.
Although the convergence order in~the $L^\infty(H^1_0)\cap H^1(L^2)$-norm is slightly reduced to~$\Delta t^{3/2}$ due to the construction of the solution at the first time step, it is still higher than first order.
When we consider an application of the scheme to the Navier--Stokes equations, the further analysis will be useful for the estimate of the pressure.
\cR
\par
Here, we make two further remarks.
(i)~In real computations our scheme is only approximately mass conservative, since numerical integration is in general required to compute the integrals occuring in the scheme. This introduces an approximation error in the total mass of the discrete solution. In this paper, in place of \emph{mass-conservative}, which we only use if no mass is lost (in the discrete case up to machine precison), we employ the term \emph{mass-preserving} to refer to schemes that are mass-conservative if the involved integrals are computed exactly.
(ii)~While there are $\ell^\infty(L^2)$-error estimates for single-step Lagrange--Galerkin methods (including space-time versions) for convection-diffusion problems that are independent of the viscosity constant, cf., e.g.,~\cite{ChrWal-2008,RuiTab-2010,TabUch-2016-CD}, the error estimates in this paper are dependent on the viscosity constant.
This is caused by an estimate of the discrete material derivative using the two-step backward differentiation formula in combination with the discrete Gronwall’s inequality for the two-step method and to the best of our knowledge no viscosity-independent error estimates for multi-step Lagrange--Galerkin methods exist. Furthermore, in applications to the Navier--Stokes equations, viscosity-dependent error estimates are usually obtained even for single-step Lagrange--Galerkin methods due to the nonlinearity.
\cK
\par
This paper is organized as follows.
Our mass-{\cR preserving two-step} Lagrange--Galerkin scheme for convection-diffusion problems is presented in Section~\ref{sec:scheme}.
The main results on the mass-{\cR preserving} property, the stability, and the convergence with optimal error estimates are stated in Section~\ref{sec:thms}, and they are proved in Section~\ref{sec:proofs}.
The theoretical convergence orders are numerically confirmed by one-, two- and three-dimensional numerical experiments in Section~\ref{sec:numerics}.
The conclusions are given in Section~\ref{sec:conclusions}. 
In the Appendix three lemmas used in Section~\ref{sec:proofs} are proved.
%
%
%
%
\section{A Lagrange--Galerkin scheme}\label{sec:scheme}
The function spaces and the notations used throughout the paper are as follows.
Let $\Omega$ be a bounded domain in $\R^d$ for $d = 1, 2$ or $3$, $\Gamma\defeq\pz\Omega$ the boundary of $\Omega$, and $T$ a positive constant.
For $m \in \mathbb{N}\cup \{0\}$ and $p\in [1,\infty]$, we use the {\rm Sobolev} spaces $W^{m,p}(\Omega)$, $W^{1,\infty}_0(\Omega)$, $H^m(\Omega) \, (=W^{m,2}(\Omega))$ and $H^1_0(\Omega)$.
For any normed space~$S$ with norm~$\|\cdot\|_S$, we define function spaces $H^m(0,T; S)$ and $C([0,T]; S)$ consisting of $S$-valued functions in $H^m(0,T)$ and $C([0,T])$, respectively.
We use the same notation $(\cdot, \cdot)$ to represent the~$L^2(\Omega)$ inner product for scalar- and vector-valued functions.
The norm on $L^2(\Omega)$ is simply denoted by~$\|\cdot\|$, i.e., $\|\cdot\| \defeq \|\cdot\|_{L^2(\Omega)}$.
The dual pairing between $S$ and the dual space $S^\prime$ is denoted by~$\lA\cdot, \cdot\rA$.
The notation~$\|\cdot\|$ is employed not only for scalar-valued functions but also for vector-valued ones.
We also denote the norm on~$H^{-1}(\Omega)$ by~$\|\cdot\|_{H^{-1}(\Omega)}$.
For $t_0$ and $t_1\in\R$~$(t_0 < t_1)$, we introduce the function space
\begin{align*}
Z^m(t_0, t_1) & \defeq \bigl\{ \psi \in H^j(t_0, t_1; H^{m-j}(\Omega));~j=0,\ldots,m,\ \|\psi\|_{Z^m(t_0, t_1)} < \infty \bigr\}
\end{align*}
with the norm
\begin{align*}
\|\psi\|_{Z^m(t_0, t_1)} & \defeq \biggl[ \sum_{j=0}^m \|\psi\|_{H^j(t_0,t_1; H^{m-j}(\Omega))}^2 \biggr]^{1/2},
\end{align*}
and set $Z^m \defeq Z^m(0, T)$.
We often omit $[0,T]$, $\Omega$, and the superscript~$d$ if there is no confusion, e.g., we shall write $C(L^\infty)$ in place of $C([0,T]; L^\infty(\Omega)^d)$.
We denote by $c$ and $c(a_1, a_2, \ldots)$ a generic positive constant and a positive constant dependent on $a_1, a_2, \ldots$, respectively, and introduce the following constants, for $i=0,1$,
\begin{align*}
c_\nu & = c (1/\nu), &
c_0 & = c \bigl( \|u\|_{C(L^{\infty})} \bigr), &
c_1 & = c \bigl( \|u\|_{C(W^{1,\infty})} \bigr), \\
c_{i,\nu} & = c ( c_i, 1/\nu ), &
c_{i,T} & = c ( c_i, T ), &
c_{i,\nu,T} & = c ( c_i, 1/\nu, T ).
\end{align*}
\par
We consider a convection-diffusion problem; find $\phi:\Omega\times (0, T)\rarrow \R$ such that
\begin{subequations}
\label{prob:strong}
\begin{align}
&&&& \prz{\phi}{t} + \nabla \cdot (u \phi) - \nu \Delta \phi &= f && {\rm in}\ \Omega \times (0, T), &&&& \\
&&&& \nu\prz{\phi}{n} - \phi u\cdot n &= g && {\rm on}\ \Gamma \times (0, T), &&&& \label{eq:bc} \\
&&&& \phi&=\phi^0 && {\rm in}\ \Omega,\ {\rm at}\ t=0, &&&&
\end{align}
\end{subequations}
where 
$u:\Omega\times(0, T)\rarrow\R^d$, 
$f:\Omega\times(0, T)\rarrow\R$, 
$g: \Gamma \times (0, T) \rarrow \R$ and 
$\phi^0:\Omega\rarrow\R$ are given functions,
$n: \pz\Omega \to \R^d$ is the outward unit normal vector,
$\nu \in (0, \nu_0]$ is a viscosity constant, and
$\nu_0 (>0)$ is an upper bound of $\nu$.
Since we are interested in problems with a small~$\nu$, i.e., convection-dominated problems, we assume without loss of generality $\nu_0 = 1$ in this paper.
\par
Let $\Psi \defeq H^1(\Omega)$.
A weak formulation to problem~\eqref{prob:strong} is to find $\{\phi(t) = \phi(\cdot,t) \in \Psi;\ t \in (0, T)\}$ such that, for $t\in (0, T)$,
\begin{align}
&&&& \Bigl( \prz{\phi}{t}(t), \psi \Bigr) + a_0(\phi(t), \psi) + a_1(\phi(t), \psi; u(t)) & = \lA F(t), \psi \rA, & \forall \psi & \in \Psi &&&&
\label{prob:weak}
\end{align}
with $\phi(0) = \phi^0$, where $a_0(\cdot,\cdot)$ and $a_1(\cdot\,,\cdot) = a_1(\cdot\,,\cdot\,;u)$ are bilinear forms defined by
\begin{align*}
a_0(\phi, \psi) & \defeq \nu (\nabla \phi, \nabla\psi),
&
a_1(\phi, \psi; u) & \defeq -(\phi, u\cdot\nabla\psi),
\end{align*}
and $F(t)\in \Psi^\prime$, $t\in (0,T)$, is a functional defined by
\begin{align}
\lA F(t), \psi \rA & \defeq
( f(t), \psi ) + [g(t), \psi]_\Gamma,
&
[g(t), \psi]_\Gamma & \defeq \int_\Gamma g(t) \psi \, ds
\label{def:functional_F}
\end{align}
for $f(t) = f(\cdot,t) \in L^2(\Omega)$ and $g(t) = g(\cdot,t) \in L^2(\Gamma)$.
\par
Let us assume $f\in L^2(0,T;L^2(\Omega))$ and $g\in L^2(0,T;L^2(\Gamma))$.
Substituting $1\in\Psi$ into $\psi$ in~\eqref{prob:weak} and integrating over $(0, t)$, one can easily obtain the so-called mass-balance identity, i.e., for~$t\in (0, T)$,
\begin{align}
\int_\Omega \phi(x, t)\,dx = \int_\Omega \phi^0(x)\,dx + \int_0^td\tau \int_\Omega f(x, \tau)\,d\tau + \int_0^td\tau \int_\Gamma g(x, \tau) ds,
\label{eq:mass-balance}
\end{align}
which is an important property of problem~\eqref{prob:strong}.
This property is, therefore, desired to hold also on the discrete level.
It is known that conventional Galerkin, streamline diffusion~(SD)~\cite{HanJoh-1991,Joh-1987}, streamline upwind/Petrov--Galerkin~(SUPG), and least square schemes~\cite{BroHug-1982,HugFraHul-1989} satisfy a discrete version of~\eqref{eq:mass-balance}.
In~\cite{RuiTab-2010}, a characteristic finite element (Lagrange--Galerkin) scheme of first order in time satisfying a discrete version of~\eqref{eq:mass-balance} has been proposed and analyzed.
%
%
\par
Let $\Delta t >0$ be a time increment, $t^n \defeq n\Delta t~(n\in\mathbb{Z})$, and $N_T \defeq \lfloor T/\Delta t\rfloor$.
For a function~$\rho$ defined in $\Omega \times (0, T)$, $\rho(\cdot, t^n)$ is simply denoted by $\rho^n$.
Let $\mathcal{T}_h$  be a triangulation of~$\Omega$, and $\Omega_h \defeq {\rm int} ( \bigcup_{K\in\mathcal{T}_h} K )$ the approximate domain, where $h$ is the maximum mesh size of $\mathcal{T}_h$, i.e., $h \defeq \max \{ h_K;\ K\in\mathcal{T}_h \}$ for $h_K\defeq \mathrm{diam} (K)$ $(K\in\mathcal{T}_h)$.
For the sake of simplicity, we assume that $\Omega_h = \Omega$ throughout this paper.
Let $\Psi_h$ be a finite element space defined by
\begin{align}
\Psi_h \defeq \Bigl\{ \psi_h \in C(\bar\Omega);\ \psi_{h|K} \in P_k(K),\ \forall K\in\mathcal{T}_h \Bigr\},
\label{def:Psi_h}
\end{align}
where $P_k(K)$ is the space of polynomial functions of degree~$k \in \N$ on $K \in \mathcal{T}_h$.
For a velocity $v: \Omega \to \R^d$, let $X_1 (v,\Delta t): \Omega \to \R^d$ be the mapping defined by
\begin{align}
X_1 (v,\Delta t) (x) \defeq x - v(x)\Delta t,
\label{def:X1}
\end{align}
which is called the upwind point of~$x$ with respect to the velocity~$v$ and the time increment~$\Delta t$.
We define mappings $X_1^n, \tilde{X}_h^n: \Omega \to \R^d$ and their Jacobians~$\gamma^n, \tilde{\gamma}^n: \Omega \to \R$ by
\begin{align*}
X_1^n (x) & \defeq X_1(u^n,\Delta t)(x) = x - u^n(x) \Delta t, &
\tilde{X}_1^n(x) & \defeq X_1(u^n, 2\Delta t)(x) = x - 2 u^n(x) \Delta t, \\
\gamma^n (x) & \defeq \mathrm{det} \biggl( \prz{X_1^n}{x} (x)\biggr), &
\tilde{\gamma}^n (x) & \defeq \mathrm{det} \biggl( \prz{\tilde{X}_1^n}{x} (x) \biggr).
\end{align*}
\par
The scheme proposed in~\cite{RuiTab-2010} is to find at each time step $\phi_h^n\in\Psi_h$ such that
\begin{align}
&&&& \Biggl( \frac{\phi_h^n-\phi_h^{n-1} \circ X_1^n\gamma^n}{\Delta t}, \psi_h \Biggr) + a_0(\phi_h^n, \psi_h) & = \lA F^n, \psi_h \rA, & \forall \psi_h & \in \Psi_h. &&&&
\label{scheme:RuiTab-2010}
\end{align}
By multiplication with the Jacobian~$\gamma^n$ the mass of $\phi_h^{n-1}$ is conserved after taking the composite with the mapping $X_i^n$ and we call this ``the {\rm Jacobian} multiplication technique.''
That is substituting~$1\in\Psi_h$ into~$\psi$ in~\eqref{scheme:RuiTab-2010} and using the identity
\begin{align*}
\int_\Omega \phi_h^{n-1} \circ X_1^n\gamma^n\,dx & = \int_\Omega \phi_h^{n-1}\,dx,
\end{align*}
we obtain a discrete mass-balance identity, cf.~\cite{RuiTab-2010} for detail.
\par
Moreover, a multistep (two-step) Galerkin method along characteristics of second order in time~\cite{EwiRus-1981} is well known; at each time step $n \in \{ 2, \ldots, N_T\}$, find $\phi_h^n\in\Psi_h$ such that
\begin{align}
&& \Biggl( \frac{3\phi_h^n-4\phi_h^{n-1} \circ X_1^n+\phi_h^{n-2} \circ \tilde{X}_1^n}{2\Delta t}, \psi_h \Biggr) + a_0(\phi_h^n, \psi_h) & = \lA F^n, \psi_h \rA, & \forall \psi_h & \in \Psi_h. &&
\label{scheme:EwiRus-1981}
\end{align}
Scheme~\eqref{scheme:EwiRus-1981} is of second order in time but does not satisfy the mass-balance identity in general.
\par
Combining the Jacobian multiplication technique~\eqref{scheme:RuiTab-2010} with the multistep (two-step) Galerkin method along characteristics~\eqref{scheme:EwiRus-1981}, we obtain the Lagrange--Galerkin scheme  proposed in this paper.
\par
Let $\phi_h^0 \in \Psi_h$ and $F \in H^1(0,T; \Psi^\prime)$ be given.
We propose a mass-{\cR preserving two-step} Lagrange--Galerkin scheme of second order in time;
find $\{\phi_h^n\in \Psi_h;\ n=1, \ldots, N_T\}$ such that, for $n=1, \ldots, N_T$,
\begin{subequations}\label{scheme}
\begin{align}
\Biggl( \frac{\phi_h^n-\phi_h^{n-1} \circ X_1^n\gamma^n}{\Delta t},  \psi_h \Biggr) + a_0( \phi_h^n, \psi_h ) =  \lA F^n, \psi_h \rA, \quad
\forall \psi_h \in \Psi_h,\ n=1,
\label{scheme:eq1}\\
\Biggl( \frac{3\phi_h^n-4\phi_h^{n-1} \circ X_1^n \gamma^n + \phi_h^{n-2} \circ \tilde{X}_1^n \tilde{\gamma}^n}{2\Delta t}, \psi_h \Biggl) + a_0( \phi_h^n, \psi_h ) = \lA F^n, \psi_h \rA, \quad \forall \psi_h \in \Psi_h,\ n \ge 2.
\label{scheme:eq2}
\end{align}
\end{subequations}
%
Since the Jacobians~$\gamma^n$ and~$\tilde{\gamma^n}$ are of the forms $1 - \Delta t (\nabla\cdot u^n) + O(\Delta t^2)$ and $1 - 2\Delta t (\nabla\cdot u^n) + O(\Delta t^2)$, respectively, it is not clear that the combined scheme is of second order in time and that the mass-balance identity is satisfied.
These properties are therefore proved in this paper.
In the following, we rewrite scheme~\eqref{scheme} simply as
\begin{align*}
( \mathcal{A}_{\Delta t} \phi_h^n, \psi_h ) + a_0( \phi_h^n, \psi_h ) = \lA F^n, \psi_h \rA, \quad
\forall \psi_h \in \Psi_h,
\end{align*}
for $n \in \{1,\ldots, N_T\}$, where, for a series~$\{\rho^n\}_{n=0}^{N_T} (\subset \Psi)$, the function $\mathcal{A}_{\Delta t} \rho^n: \Omega \to \R$ is defined by
\begin{align*}
\mathcal{A}_{\Delta t} \rho^n \defeq
\left\{
\begin{aligned}
& \mathcal{A}_{\Delta t}^{(1)} \rho^n \defeq \fz{1}{\Delta t} \Bigl( \rho^n-\rho^{n-1} \circ X_1^n\gamma^n \Bigr),
& n & = 1, \\
& \mathcal{A}_{\Delta t}^{(2)} \rho^n \defeq \fz{1}{2\Delta t} \Bigl( 3\rho^n-4\rho^{n-1} \circ X_1^n \gamma^n + \rho^{n-2} \circ \tilde{X}_1^n \tilde{\gamma}^n \Bigr),
& n & \ge 2.
\end{aligned}
\right.
\end{align*}
\begin{Rmk}
(i)~The first order scheme~\eqref{scheme:eq1} is employed in the first time step, since then the approximate solution $\phi_h^1$ needed in~\eqref{scheme:eq2} with $n=2$ is not yet available.
This construction of $\phi_h^1$ is efficient and has no adverse effect on the convergence order in the $\ell^\infty(L^2)$-norm, cf. Theorem~\ref{thm:error_estimates}.
\smallskip \\
(ii)~$F \in H^1(0,T; \Psi^\prime)$ implies that $F \in C([0,T]; \Psi^\prime)$ and $\{F^n\}_{n=1}^{N_T} \subset \Psi^\prime$.
\end{Rmk}
%
%
%
%
%
%
%
%
%
\section{Main results}\label{sec:thms}
We start this section, by setting hypotheses for the velocity $u$ and the time increment $\Delta t$, and reviewing previous results.
\begin{Hyp}\label{hyp:u}
The function $u$ satisfies $u \in C([0,T]; W^{1,\infty}_0(\Omega)^d)$.
\end{Hyp}
\begin{Hyp}\label{hyp:dt}
The time increment $\Delta t$ satisfies the condition~$\Delta t |u|_{C(W^{1,\infty})} \le 1/8$.
\end{Hyp}
\begin{Prop}[ \cite{RuiTab-2002,TabUch-2018-NS} ]
\label{prop:bijective_jacobian}
\quad\\
(i)~Under Hypothesis~\ref{hyp:u} and $\Delta t |u|_{C(W^{1,\infty})} < 1/2$, it holds that $X_1^n(\Omega)=\tilde{X}_1^n(\Omega) = \Omega$ for $n=0, \ldots, N_T$.\\
(ii)~Under Hypotheses~\ref{hyp:u} and~\ref{hyp:dt}, it holds that $1/2 \le \gamma^n, \tilde{\gamma}^n \le 3/2$ for $n=0, \ldots, N_T$.
\end{Prop}
\par
For $n=0, \ldots, N_T$, let $\mathcal{M}_h^n$ be an approximate value of mass at $t=t^n$ defined by
\begin{align*}
\mathcal{M}_h^n & \defeq
\left\{
\begin{aligned}
& \int_\Omega \phi_h^n dx, && n = 0, 1, \\
& \int_\Omega \Bigl( \fz{3}{2}\phi_h^n-\fz{1}{2}\phi_h^{n-1} \Bigr) dx, && n\ge 2.
\end{aligned}
\right.
\end{align*}
\begin{Rmk}
The value~$\mathcal{M}_h^n$ is an approximation of $\int_\Omega \phi^n dx$ due to the relation $\fz{3}{2}\phi^n-\fz{1}{2}\phi^{n-1} (= \phi^{n+1/2} + O(\Delta t^2)) = \phi^n + O(\Delta t)$ for any smooth function~$\phi$.
\end{Rmk}
\begin{Thm}[conservation of mass]\label{thm:mass-conservation}
Suppose that Hypotheses~\ref{hyp:u} and~\ref{hyp:dt} hold true.
Let $\phi_h = \{\phi_h^n\}_{n=1}^T$ be a solution to scheme~\eqref{scheme} for a given $\phi_h^0$.
Then, we have the following.
\smallskip\\
(i)~It holds that, for $n=0,\ldots, N_T$,
\begin{align}
\mathcal{M}_h^n = \mathcal{M}_h^0 + \Delta t \sum_{i=1}^n \Bigl( \int_\Omega f^i dx + \int_\Gamma g^i ds \Bigr).
\label{eq:discrete-mass-conservation}
\end{align}
(ii)~Assume $f=0$ and $g=0$ additionally.
Then, for the solution to scheme~\eqref{scheme}, it holds that, for $n=0,\ldots, N_T$,
\begin{align}
\int_\Omega \phi_h^n dx = \int_\Omega \phi_h^0 dx.
\label{eq:mass-conservation_f0_g0}
\end{align}
\end{Thm}
\begin{Rmk}
The identity~\eqref{eq:discrete-mass-conservation} is equivalent to
\begin{align*}
\int_\Omega \phi_h^n dx &= \int_\Omega \phi_h^0 dx + \Delta t \sum_{i=1}^n \Bigl( \int_\Omega f^i dx + \int_\Gamma g^i ds \Bigr), \quad n=0, 1, \\
\int_\Omega \Bigl(\fz{3}{2}\phi_h^n - \fz{1}{2}\phi_h^{n-1} \Bigr) dx &= \int_\Omega \phi_h^0 dx + \Delta t \sum_{i=1}^n \Bigl( \int_\Omega f^i dx + \int_\Gamma g^i ds \Bigr), \quad n\ge 2.
\end{align*}
\end{Rmk}
\par
For a sequence $\{\rho^n\}_{n = 0}^{N_T}$, let $\bar{D}_{\Delta t}$ be the backward quotient operator defined by
\[
\bar{D}_{\Delta t}\rho^n \defeq
\left\{
\begin{aligned}
& \bar{D}_{\Delta t}^{(1)}\rho^n, && n = 1, \\
& \bar{D}_{\Delta t}^{(2)}\rho^n, && n \ge 2,
\end{aligned}
\right.
\]
where $\bar{D}_{\Delta t}^{(1)}$ and $\bar{D}_{\Delta t}^{(2)}$ are the first- and second-order backward difference quotient operators,
\[
\bar{D}_{\Delta t}^{(1)} \rho^n \defeq \fz{\rho^n-\rho^{n-1}}{\Delta t},
\qquad
\bar{D}_{\Delta t}^{(2)} \rho^n \defeq \fz{3\rho^n-4\rho^{n-1}+\rho^{n-2}}{2\Delta t}.
\]
%
Let $m \in\{ 0,\ldots,N_T\}$ be an integer and $Y$ be a normed space.
When $\{\rho^n\}_{n = 0}^{N_T} \subset Y$, we define the norms~$\|\cdot\|_{\ell^\infty_m(Y)}$ and~$\|\cdot\|_{\ell^2_m(Y)}$ by
\begin{align*}
\|\rho\|_{\ell^\infty_m(Y)} & \defeq \max_{n=m,\ldots,N_T} \|\rho^n\|_Y,
&
\|\rho\|_{\ell^2_m (Y)} &\defeq \biggl\{ \Delta t \sum_{n=m}^{N_T} \|\rho^n\|_{Y}^2 \biggr\}^{1/2},
\end{align*}
and let $\|\rho\|_{\ell^\infty(Y)} \defeq \|\rho\|_{\ell^\infty_1(Y)}$ and $\|\rho\|_{\ell^2(Y)} \defeq \|\rho\|_{\ell^2_1(Y)}$.
When $Y = L^2(\Omega)$, we omit $\Omega$ from the norms, e.g., $\|\rho\|_{\ell^\infty(L^2)}$, and use the same notations $\|\cdot\|_{\ell^\infty_m (L^2)}$, $\|\cdot\|_{\ell^2_m (L^2)}$, $\|\cdot\|_{\ell^\infty (L^2)}$ and $\|\cdot\|_{\ell^2 (L^2)}$ also for a sequence of vector valued functions, e.g., $\|\nabla\rho\|_{\ell^\infty(L^2)} = \max_{n=1,\ldots,N_T} \|\nabla\rho^n\|_{L^2(\Omega)^d}$.
\begin{Prop}[stability for a given $\phi_h^1$]\label{prop:stability}
Suppose that Hypothesis~\ref{hyp:u} holds true.
Let $F \in H^1(0,T; \Psi^\prime)$ be given.
Suppose that Hypothesis~\ref{hyp:dt} holds true, and assume $\Delta t \in {\cR (0, 1)}$.
For given functions~$\phi_h^0, \phi_h^1 \in \Psi_h$, let $\{\phi_h^n\}_{n=2}^{N_T} \subset \Psi_h$ be the solution to scheme~\eqref{scheme:eq2}.
Then, we have the following: \smallskip\\
(i)
There exists a positive constant $c_\dagger = c_\dagger( \|u\|_{C(W^{1,\infty})}, T, 1/\nu )$ independent of~$h$ and~$\Delta t$ such that
\begin{align}
\|\phi_h\|_{\ell^\infty_2(L^2)} + \sqrt{\nu} \|\nabla\phi_h\|_{\ell^2_2(L^2)}
\le c_\dagger \left( \|\phi_h^0\| + \|\phi_h^1\| + \|F\|_{\ell^2_2(\Psi_h^\prime)} \right).
\label{ieq:stability_prop_L2}
\end{align}
(ii)
Assume~$F \in H^1(0,T; L^2(\Omega))$ additionally.
Then, there exists a positive constant $\bar{c}_\dagger = \bar{c}_\dagger (\|u\|_{C(W^{1,\infty})}, T, 1/\nu)$ independent of $h$ and $\Delta t$ such that
\begin{align}
\sqrt{\nu} \|\nabla\phi_h\|_{\ell^\infty_2(L^2)} 
+ \|\bar{D}_{\Delta t}\phi_h\|_{\ell^2_2(L^2)}
\le \bar{c}_\dagger (\|\phi_h^0\|_{H^1(\Omega)} + \|\phi_h^1\|_{H^1(\Omega)} + \|F\|_{\ell^2_2(L^2)}).
\label{ieq:stability_prop_H1}
\end{align}
\end{Prop}
\begin{Thm}[stability]\label{thm:stability}
Suppose that Hypothesis~\ref{hyp:u} holds true.
Let $F \in H^1(0,T; \Psi^\prime)$ be given.
Suppose that Hypothesis~\ref{hyp:dt} holds true, and assume $\Delta t \in {\cR (0, 1)}$.
For a given function~$\phi_h^0 \in \Psi_h$, let $\{\phi_h^n\}_{n=1}^{N_T} \subset \Psi_h$ be the solution to scheme~\eqref{scheme}.
Then, we have the following: \smallskip\\
(i)
There exists a positive constant $c_\ddagger=c_\ddagger(\|u\|_{C(W^{1,\infty})}, T, 1/\nu)$ independent of $h$ and $\Delta t$ such that
\begin{align}
\|\phi_h\|_{\ell^\infty(L^2)} + \sqrt{\nu} \|\nabla\phi_h\|_{\ell^2(L^2)}
\le c_\ddagger \left( \|\phi_h^0\| + \|F\|_{\ell^2(\Psi_h^\prime)} \right).
\label{ieq:stability_L2}
\end{align}
(ii)
Assume~$F \in H^1(0,T; L^2(\Omega))$ additionally.
Then, there exists a positive constant~$\bar{c}_\ddagger = \bar{c}_\ddagger(\|u\|_{C(W^{1,\infty})}, T, 1/\nu)$ independent of $h$ and $\Delta t$ such that
\begin{align}
\sqrt{\nu} \|\nabla\phi_h\|_{\ell^\infty(L^2)} 
+ \|\bar{D}_{\Delta t}\phi_h\|_{\ell^2(L^2)}
\le \bar{c}_\ddagger \bigl( \|\phi_h^0\|_{H^1(\Omega)} + \|F\|_{\ell^2(L^2)} \bigr).
\label{ieq:stability_H1}
\end{align}
\end{Thm}
\cR
\begin{Rmk}
The assumption~$F \in H^1(0,T; L^2(\Omega))$ in Theorem~\ref{thm:stability}-(ii) implies $g=0$, which is explicitly written in Corollary~\ref{cor:stability}-(ii) below.
\end{Rmk}
\cK
\begin{Cor}
\label{cor:stability}
(i)~Suppose that the functional~$F \in H^1(0,T; \Psi^\prime)$ is given by~\eqref{def:functional_F} with $f\in H^1(0,T;L^2(\Omega))$ and $g\in H^1(0,T;L^2(\Gamma))$, the stability estimate~\eqref{ieq:stability_L2} in Theorem~\ref{thm:stability}-(i) becomes
\begin{align*}
\|\phi_h\|_{\ell^\infty(L^2)} + \sqrt{\nu} \, \|\nabla\phi_h\|_{\ell^2(L^2)}
\le c_\ddagger \bigl( \|\phi_h^0\| + \|f\|_{\ell^2(L^2)} + \|g\|_{\ell^2(L^2(\Gamma))} \bigr).
\end{align*}
(ii)~Suppose that the functional~$F \in H^1(0,T; \Psi^\prime)$ is given by~\eqref{def:functional_F} with $f\in H^1(0,T;L^2(\Omega))$ and $g=0$, the stability estimate~\eqref{ieq:stability_H1} in Theorem~\ref{thm:stability}-(ii) becomes
\begin{align*}
\sqrt{\nu} \, \|\nabla\phi_h\|_{\ell^\infty(L^2)} 
+ \|\bar{D}_{\Delta t}\phi_h\|_{\ell^2(L^2)}
\le \bar{c}_\ddagger \bigl( \|\phi_h^0\|_{H^1(\Omega)} + \|f\|_{\ell^2(L^2)} \bigr).
\end{align*}
\end{Cor}
%
%
%
%
\par
We present the convergence result of second order in time after stating regularity {\cB hypotheses} for the solution to problem~\eqref{prob:weak} {\cB given} the polynomial degree~$k\in\N$ of the finite element space~$\Psi_h$ {\cB in Hypothesis~\ref{hyp:phi} and for the solution of the Poisson problem in Hypothesis~\ref{hyp:A-N}. Then we define the Poisson projection in Definition~\ref{def:Poisson-Proj}.}
\begin{Hyp}\label{hyp:phi}
The solution~$\phi$ to~\eqref{prob:weak} satisfies $\phi \in Z^3 \cap H^2(0,T; H^{k+1}(\Omega))$.
\end{Hyp}
\begin{Rmk}
We suppose~$H^2(0,T; H^{k+1}(\Omega))$, since the regularity~$H^1(0,T; H^{k+1}(\Omega))$ is not sufficient to get the convergence of second order in time, especially for the estimate of the solution at the first time step.
\end{Rmk}
\begin{Hyp}\label{hyp:A-N}
The {\rm Poisson} problem is regular on the domain $\Omega$, i.e.,
for any $\tilde{f}\in L^2(\Omega)$, there exists a unique solution to the {\rm Poisson} problem; find $\rho \in \Psi$ such that
\begin{align*}
a_0(\rho,\psi) + (\rho, \psi) = (\tilde{f},\psi), \quad \forall \psi \in \Psi,
\end{align*}
and there exists a positive constant $c_R$ independent of~$\tilde{f}$ and $\rho$ such that
\begin{align*}
\|\rho\|_{H^2(\Omega)} \le c_R \|\tilde{f}\|.
\end{align*}
\end{Hyp}
\cR
%
%
\begin{Def}\label{def:Poisson-Proj}
For $\phi \in \Psi$, we define the {\rm Poisson} projection $\hat{\phi}_h \in \Psi_h$ to~$\phi$ by
\begin{align}
a_0(\hat{\phi}_h, \psi_h) + (\hat{\phi}_h, \psi_h) =
a_0(\phi, \psi_h) + (\phi, \psi_h), \quad \forall \psi_h \in \Psi_h.
\label{eq:Poisson-Proj}
\end{align}
\end{Def}
\cK
\begin{Thm}[error estimates]\label{thm:error_estimates}
Suppose that Hypothesis~\ref{hyp:u} holds true.
For a given $F\in H^1(0,T; \Psi^\prime)$, let $\{\phi(t) = \phi(\cdot,t)\in \Psi;~t\in (0,T) \}$ be the solution to problem~\eqref{prob:weak}.
Suppose that Hypothesis~\ref{hyp:phi} holds true.
Let $\Delta t\in {\cR (0, 1)}$ be a time increment satisfying Hypothesis~\ref{hyp:dt} and $\{\phi_h^n\}_{n=1}^{N_T} \subset \Psi_h$ be the solution to scheme~\eqref{scheme} with the initial condition~$\phi_h^0 = \hat{\phi}_h^0 \in \Psi_h$.
Then, we have the following: \smallskip\\
(i)
There exist positive constants~$c_\ast$ and~$c_\ast^\prime$ independent of $h$ and $\Delta t$ such that
\begin{subequations}
\begin{align}
\| \phi_h - \phi \|_{\ell^\infty(L^2)} + \sqrt{\nu} \, \| \nabla (\phi_h - \phi) \|_{\ell^2(L^2)} & \le c_\ast (\Delta t^2 + h^k) \|\phi\|_{Z^3\cap H^2(H^{k+1})}, 
\label{ieq:thm_convergence_i_a} \\
%
%
\sqrt{\nu} \| \nabla (\phi_h - \phi) \|_{\ell^\infty(L^2)} + \Bigl\| \bar{D}_{\Delta t}\phi_h - \prz{\phi}{t} \Bigr\|_{\ell^2(L^2)} & \le c_\ast^\prime (\Delta t^{3/2} + h^k) \|\phi\|_{Z^3\cap H^2(H^{k+1})}.
\label{ieq:thm_convergence_i_b}
\end{align}
\end{subequations}
(ii)
Suppose that additionally Hypothesis~\ref{hyp:A-N} holds.
Then, there exists a positive constant~$\bar{c}_\ast$ independent of $h$ and $\Delta t$ such that
\begin{align}
\| \phi_h - \phi \|_{\ell^\infty(L^2)} \le \bar{c}_\ast (\Delta t^2 + h^{k+1}) \|\phi\|_{Z^3\cap H^2(H^{k+1})}.
\label{ieq:thm_convergence_ii}
\end{align}
%
\end{Thm}
%
%
%
%
%
%
%
%
\section{Proofs}\label{sec:proofs}
%
%
%
%
%
%
%
\subsection{Proof of Theorem~\ref{thm:mass-conservation}}
We first note that due to Proposition~\ref{prop:bijective_jacobian}-(i)
\begin{align}
  \int_{\Omega} \rho \circ X_1^n(x) \gamma^n(x) dx = \int_\Omega \rho (x)\, dx, \quad \int_{\Omega} \rho \circ \tilde X_1^n(x) \tilde \gamma^n(x) dx = \int_\Omega \rho(x)\, dx
  \label{eq:jac_multiplication}
  \end{align}
hold for any $\rho\in\Psi$ and $n=1,\dots, N_t$.
We substitute $1 \in \Psi_h$ into $\psi_h$ in scheme~\eqref{scheme} in the following.
\par
We prove (i) by induction.
\\
(I)~Initial steps ($n=0, 1$): \quad
Since \eqref{eq:discrete-mass-conservation} with $n=0$ is trivial, we prove it for $n = 1$.
We have
\begin{align}
\mathcal{M}_h^1 & = \int_\Omega \phi_h^1(x) dx \notag\\
& = \int_\Omega \phi_h^0\circ X_1^1(x)\gamma^1(x) dx + \Delta t \Bigl( \int_\Omega f^1(x) dx + \int_\Gamma g^1(x) ds \Bigr) \qquad \mbox{(by~\eqref{scheme:eq1})} \notag \\
& = \int_\Omega \phi_h^0 (y) dy + \Delta t \Bigl( \int_\Omega f^1(x) dx + \int_\Gamma g^1(x) ds \Bigr) \qquad \mbox{(by~\eqref{eq:jac_multiplication})} \notag \\
& = \mathcal{M}_h^0 + \Delta t \Bigl( \int_\Omega f^1(x) dx + \int_\Gamma g^1(x) ds \Bigr). \notag
\end{align}
Hence, \eqref{eq:discrete-mass-conservation} holds for $n=0, 1$. \\
(II)~General steps: \quad
Let $m\in\{2,\ldots, N_T\}$ and suppose that \eqref{eq:discrete-mass-conservation} holds true for $n=m-1$.
Then, we obtain~\eqref{eq:discrete-mass-conservation} for $n=m$ as follows:
\begin{align*}
\mathcal{M}_h^m
& = \int_\Omega \Bigl( \fz{3}{2} \phi_h^m - \fz{1}{2} \phi_h^{m-1} \Bigr) dx \notag\\
& = \int_\Omega \Bigl( \fz{3}{2} \phi_h^m - \fz{1}{2} \phi_h^{m-1} \circ X_1^m \gamma^m \Bigr) dx \qquad \mbox{(cf.~\eqref{eq:jac_multiplication})} \notag\\
& = \int_\Omega \Bigl( \fz{3}{2} \phi_h^{m-1} \circ X_1^m \gamma^m - \fz{1}{2} \phi_h^{m-2} \circ \tilde{X}_1^m \tilde{\gamma}^m \Bigr) dx + \Delta t \Bigl( \int_\Omega f^m(x) dx + \int_\Gamma g^m(x) ds \Bigr) \\
& \qquad\qquad\qquad\qquad\qquad\qquad\qquad\qquad\qquad\qquad\qquad\qquad\qquad\qquad \mbox{(by~\eqref{scheme:eq2})}\\
& = \int_\Omega \Bigl( \fz{3}{2} \phi_h^{m-1} - \fz{1}{2} \phi_h^{m-2} \Bigr) dx + \Delta t \Bigl( \int_\Omega f^m(x) dx + \int_\Gamma g^m(x) ds \Bigr) \\
& = \mathcal{M}_h^{m-1} + \Delta t \Bigl( \int_\Omega f^m(x) dx + \int_\Gamma g^m(x) ds \Bigr) \qquad \mbox{(cf.~\eqref{eq:jac_multiplication})} \notag\\
& = \mathcal{M}_h^0 + \Delta t \sum_{i=1}^m \Bigl( \int_\Omega f^i(x) dx + \int_\Gamma g^i(x) ds \Bigr) \\
& \qquad\qquad\qquad\qquad\qquad \mbox{(by the induction assumption, i.e., \eqref{eq:discrete-mass-conservation} with $n=m-1$)}.
\end{align*}
From (I) and (II) the proof of~(i) is completed.
\par
We prove (ii) by induction. \\
(I')~Initial steps ($n=0, 1$): \quad
The property~\eqref{eq:mass-conservation_f0_g0} is obvious for $n=0, 1$, cf. (I) in the proof of~(i).
\\
(II')~General steps: \quad
Let $m\in\{ 2, \ldots, N_T\}$ and assume that \eqref{eq:mass-conservation_f0_g0} holds true for $n = m-1$ and $m-2$, we prove that \eqref{eq:mass-conservation_f0_g0} also does for $n=m$.
From~\eqref{scheme:eq2} with $f=0$, $g=0$ and the induction assumption, we obtain~\eqref{eq:mass-conservation_f0_g0} with $n=m$ as follows:
\begin{align*}
\int_\Omega \phi_h^m dx
& = \int_\Omega \Bigl( \fz{4}{3} \phi_h^{m-1}\circ X_1^m\gamma^m - \fz{1}{3}\phi_h^{m-2}\circ \tilde{X}_1^m \tilde{\gamma}^m \Bigr) dx \\
& = \int_\Omega \Bigl( \fz{4}{3} \phi_h^{m-1} - \fz{1}{3} \phi_h^{m-2} \Bigr) dx 
= \int_\Omega \phi_h^0 dx.
\end{align*}
From (I') and (II') the proof of~(ii) is completed.
\qed
%
%
%
%
%
\subsection{Proofs of Proposition~\ref{prop:stability} and Theorem~\ref{thm:stability}}
The proofs are given after stating two lemmas on a discrete Gronwall's inequality and composite functions.
The proof of the next lemma is given in Appendix~\ref{subsec:proof_lem:gronwall}.
\begin{Lem}\label{lem:gronwall}
Let $a_i, i=0,1,2,$ be non-negative numbers with $a_1 \ge a_2$, and $\Delta t\in (0, 3/(4a_0)]$.
Let $\{x_n\}_{n\ge 0}$, $\{y_n\}_{n\ge 1}$, $\{z_n\}_{n\ge 2}$ and $\{b_n\}_{n\ge 2}$ be non-negative sequences.
Suppose that
\begin{align}
\fz{1}{\Delta t}\Bigl( \fz{3}{2} x_n - 2 x_{n-1} + \fz{1}{2} x_{n-2} + y_n - y_{n-1} \Bigr) + z_n \le a_0 x_n + a_1 x_{n-1} + a_2 x_{n-2} + b_n,\quad \forall n\ge 2
\label{ieq:gronwall_hyp}
\end{align}
holds.
Then, it holds that
\begin{align}
x_n + \fz{2}{3} y_n + \fz{2}{3}\Delta t\sum_{i=2}^nz_i & \le \Bigl( \exp( 2 a_\ast n\Delta t ) +1 \Bigr) \Biggl( x_0 + \fz{3}{2} x_1 + y_1+\Delta t\sum_{i=2}^n b_i \Biggr), \quad \forall n \ge 2,
\label{ieq:gronwall}
\end{align}
where $a_\ast \defeq a_0+a_1+a_2$.
\end{Lem}
\par
We recall some results concerning the evaluation of composite functions, which are mainly due to Lemma~4.5 in~\cite{AchGue-2000} and Lemma~1 in~\cite{DouRus-1982}.
 \begin{Lem}[ \cite{AchGue-2000,DouRus-1982,NT-2016-M2AN,RuiTab-2002} ]\label{lem:comp_funcs}
  Let $a$ be a  function in $W^{1,\infty}_0(\Omega)^d$ satisfying
$\Delta t \|a\|_{1,\infty} \le 1/4$
and consider the mapping $X_1(a,\Delta t)$  defined in~\eqref{def:X1}.
Then, the following inequalities hold.
\begin{subequations}
\begin{align}
\|\psi\circ X_1(a,\Delta t)\| & \le (1 + c_1 \Delta t) \|\psi\|,& \forall \psi & \in L^2(\Omega),
\label{ieq:vX_1}\\
\|\psi - \psi \circ X_1(a,\Delta t)\| & \le c_0 \Delta t \|\psi\|_{H^1(\Omega)},& \forall \psi & \in H^1(\Omega),
\label{ieq:v-vX_1}\\
\|\psi - \psi \circ X_1(a,\Delta t) \|_{H^{-1}(\Omega)} & \le c_1 \Delta t \|\psi\|,& \forall \psi & \in L^2(\Omega).
\label{ieq:v-vX_2}
\end{align}
\end{subequations}
\end{Lem}
\begin{proof}[Proof of Proposition~\ref{prop:stability}]
The equation~\eqref{scheme:eq2} can be written as
\begin{align}
\bigl( \bar{D}_{\Delta t}^{(2)}\phi_h^n, \psi_h \bigr) + a_0(\phi_h^n, \psi_h) 
& = \lA F^n, \psi_h \rA + \lA I_h^n, \psi_h \rA, \quad \forall \psi_h \in \Psi_h
\label{scheme:eq2m}
\end{align}
for $n \ge 2$, where $I_h^n \in \Psi_h^\prime$ with the norm $\|\cdot\|_{\Psi_h} \defeq \|\cdot\|_{\Psi} \ (= \|\cdot\|_{H^1(\Omega)})$ is defined for $n \in \{ 2, \ldots, N_T\}$ by
\begin{align*}
I_h^n & \defeq 
\fz{1}{2\Delta t} \Bigl[
- 4 \bigl( \phi_h^{n-1} - \phi_h^{n-1} \circ X_1^n \gamma^n \bigr) 
+ \bigl( \phi_h^{n-2} - \phi_h^{n-2} \circ \tilde{X}_1^n\tilde{\gamma}^n \bigr) \Bigr] \\
& = \fz{1}{2\Delta t} \Bigl[
- 4 \bigl( \phi_h^{n-1} - \phi_h^{n-1} \circ X_1^n \bigr) 
+ \bigl( \phi_h^{n-2} - \phi_h^{n-2} \circ \tilde{X}_1^n \bigr) \Bigr] \notag\\
& \quad + \fz{1}{2\Delta t} \Bigl[
- 4 \bigl( \phi_h^{n-1} \circ X_1^n - \phi_h^{n-1} \circ X_1^n \gamma^n \bigr) 
+ \bigl( \phi_h^{n-2} \circ \tilde{X}_1^n - \phi_h^{n-2} \circ \tilde{X}_1^n\tilde{\gamma}^n \bigr) \Bigr] \notag\\
& \eqdef I_{h1}^n + I_{h2}^n. \notag
\end{align*}
\par
We prove~$(i)$.
Substituting $\phi_h^n \in \Psi_h$ into $\psi_h$ in~\eqref{scheme:eq2m}, we have
\begin{align}
\fz{1}{\Delta t} \Bigl[ \fz{3}{4} \|\phi_h^n\|^2 - \|\phi_h^{n-1}\|^2 & + \fz{1}{4}\|\phi_h^{n-2}\|^2 + \fz{1}{2}( \|\phi_h^n-\phi_h^{n-1}\|^2-\|\phi_h^{n-1} - \phi_h^{n-2}\|^2) \Bigr] + \fz{\nu}{2} \|\nabla\phi_h^n\|^2 \notag \\
& \le \fz{3}{8} \|\phi_h^n\|^2 + c_{1,\nu} \Bigl( \fz{1}{2} \|\phi_h^{n-1}\|^2 + \fz{1}{2} \|\phi_h^{n-2}\|^2 \Bigr) + c_\nu \|F^n\|_{\Psi_h^\prime}^2
\label{ieq:proof_stability_i_1}
\end{align}
from the estimates, thanks to Young's inequality and an identity in~\cite{Rav-2012} for~$(\bar{D}_{\Delta t}^{(2)}\phi_h^n, \phi_h^n)$,
\begin{align}
\bigl( \bar{D}_{\Delta t}^{(2)}\phi_h^n, \phi_h^n \bigr) 
& = \fz{1}{\Delta t} \Bigl[ \fz{3}{4} \|\phi_h^n\|^2 - \|\phi_h^{n-1}\|^2 + \fz{1}{4} \|\phi_h^{n-2}\|^2 + \fz{1}{4} \|\phi_h^n-2\phi_h^{n-1}+\phi_h^{n-2}\|^2 \notag \\
& \qquad + \fz{1}{2} \Bigl( \|\phi_h^n-\phi_h^{n-1}\|^2 - \|\phi_h^{n-1}-\phi_h^{n-2}\|^2 \Bigr) \Bigr] \notag \\
& \ge \fz{1}{\Delta t} \Bigl[ \fz{3}{4} \|\phi_h^n\|^2 - \|\phi_h^{n-1}\|^2 + \fz{1}{4} \|\phi_h^{n-2}\|^2 \notag\\
& \qquad + \fz{1}{2} \Bigl( \|\phi_h^n-\phi_h^{n-1}\|^2 - \|\phi_h^{n-1}-\phi_h^{n-2}\|^2 \Bigr) \Bigr], \notag \\
a_0(\phi_h^n, \phi_h^n) & = \nu \|\nabla\phi_h^n\|^2, \notag \\
\lA F^n, \phi_h^n \rA & \le \|F^n\|_{\Psi_h^\prime} \|\phi_h^n\|_{H^1(\Omega)}
\le \|F^n\|_{\Psi_h^\prime} (\|\phi_h^n\| + \|\nabla\phi_h^n\|) \notag \\
& \le \fz{1}{8} \|\phi_h^n\|^2 + \fz{\nu}{4}\|\nabla\phi_h^n\|^2 + c_\nu \|F^n\|_{\Psi_h^\prime}^2 \quad (c_\nu = 2 + 1/\nu),
\label{ieq:proof_stability_i_F_h} \\
\|I_{h1}^n\|_{\Psi_h^\prime}
& \le c_1(\|\phi_h^{n-1}\|+\|\phi_h^{n-2}\|) \quad \mbox{(by Lem.~\ref{lem:comp_funcs}-\eqref{ieq:v-vX_2})}, \notag \\
\|I_{h2}^n\|
& \le \fz{c}{\Delta t} \Bigl( \| \phi_h^{n-1} \circ X_1^n ( 1 - \gamma^n ) \| + \| \phi_h^{n-2} \circ \tilde{X}_1^n ( 1 - \tilde{\gamma}^n ) \| \Bigr) \notag \\
& \le c_1 \bigl( \| \phi_h^{n-1} \| + \| \phi_h^{n-2} \| \bigr) \notag\\
& \qquad \mbox{(by $\|1-\gamma\|_{C(L^\infty)},\ \|1-\tilde{\gamma}\|_{C(L^\infty)} \le c_1 \Delta t$,\ \  Lem.~\ref{lem:comp_funcs}-\eqref{ieq:vX_1})}
\label{ieq:proof_stability_i_I_h2} \\
\lA I_h^n, \phi_h^n \rA 
& \le \|I_{h1}^n\|_{\Psi_h^\prime} \| \phi_h^n \|_{\Psi_h} + \|I_{h2}^n\|  \| \phi_h^n \| 
\le \|I_{h1}^n\|_{\Psi_h^\prime} (\| \phi_h^n \| + \| \nabla\phi_h^n \|) + \|I_{h2}^n\|  \| \phi_h^n \| \notag \\
& \le \Bigl( 2 + \fz{1}{\nu} \Bigr) \|I_{h1}^n\|_{\Psi_h^\prime}^2 + 2 \|I_{h2}^n\|^2 + \fz{1}{4} \| \phi_h^n \|^2 + \fz{\nu}{4} \| \nabla\phi_h^n \|^2 \notag \\
& \le \fz{1}{4}\|\phi_h^n\|^2 + \fz{\nu}{4} \| \nabla\phi_h^n \|^2 + c_{1,\nu} \Bigl( \fz{1}{2} \|\phi_h^{n-1}\|^2 + \fz{1}{2} \|\phi_h^{n-2}\|^2 \Bigr).
\label{ieq:proof_stability_i_I_h}
\end{align}
The inequality~\eqref{ieq:proof_stability_i_1} and Lemma~\ref{lem:gronwall} with
\begin{align*}
x_n & = \fz{1}{2}\|\phi_h^n\|^2, & 
y_n & = \fz{1}{2}\|\phi_h^n-\phi_h^{n-1}\|^2, & 
z_n & = \fz{\nu}{2} \|\nabla\phi_h^n\|^2, \\
a_0 & = \fz{3}{4}, & 
a_1 & = a_2 = c_{1,\nu}, & 
b_n & = c_\nu \|F^n\|_{\Psi_h^\prime}^2
\end{align*}
imply 
\begin{align*}
\max_{n=2,\ldots,N_T} \| \phi_h^n \|^2
+ \nu \Delta t\sum_{n=2}^{N_T} \| \nabla \phi_h^n \|^2 \le c_{1,\nu,T} \Bigl[ \| \phi_h^0 \|^2 + \| \phi_h^1 \|^2 + \Delta t \sum_{n=2}^{N_T} \|F^n\|_{\Psi_h^\prime}^2 \Bigr],
\end{align*}
which completes the proof of~$(i)$.
\par
Next we prove~$(ii)$.
Substituting $\bar{D}_{\Delta t}^{(2)}\phi_h^n \in \Psi_h$ into $\psi_h$ in~\eqref{scheme:eq2m}, we have
\begin{align}
\fz{\nu}{\Delta t} \biggl[
& \fz{3}{4} \|\nabla\phi_h^n\|^2 - \|\nabla\phi_h^{n-1}\|^2 + \fz{1}{4} \|\nabla\phi_h^{n-2}\|^2 + \fz{1}{2} \Bigl( \|\nabla(\phi_h^n-\phi_h^{n-1})\|^2 - \|\nabla(\phi_h^{n-1}-\phi_h^{n-2})\|^2 \Bigr) \biggr] \notag\\ 
& + \fz{1}{2} \| \bar{D}_{\Delta t}^{(2)}\phi_h^n \|^2
\le \fz{c_1}{\nu} \biggl( \fz{\nu}{2} \|\nabla\phi_h^{n-1}\|^2 + \fz{\nu}{2} \|\nabla\phi_h^{n-2}\|^2 \biggr) + \|F^n\|^2 + c_1^\prime \sum_{i=1}^2 \|\phi_h^{n-i}\|^2
\label{ieq:proof_stability_ii_1}
\end{align}
from the estimates
\begin{align}
\bigl( \bar{D}_{\Delta t}^{(2)}\phi_h^n, \bar{D}_{\Delta t}^{(2)}\phi_h^n \bigr) & = \| \bar{D}_{\Delta t}^{(2)}\phi_h^n \|^2, \notag\\
a_0(\phi_h^n, \bar{D}_{\Delta t}^{(2)}\phi_h^n)
& = \fz{\nu}{\Delta t} \Bigl[ \fz{3}{4} \|\nabla\phi_h^n\|^2 - \|\nabla\phi_h^{n-1}\|^2 + \fz{1}{4} \|\nabla\phi_h^{n-2}\|^2 \notag\\
& \quad + \fz{1}{4} \|\nabla(\phi_h^n-2\phi_h^{n-1}+\phi_h^{n-2})\|^2 \notag\\
& \quad + \fz{1}{2} \Bigl( \|\nabla(\phi_h^n-\phi_h^{n-1})\|^2 - \|\nabla(\phi_h^{n-1}-\phi_h^{n-2})\|^2 \Bigr) \Bigr] \mbox{(by an identity in~\cite{Rav-2012})} \notag\\
& \ge \fz{\nu}{\Delta t} \Bigl[ \fz{3}{4} \|\nabla\phi_h^n\|^2 - \|\nabla\phi_h^{n-1}\|^2 + \fz{1}{4} \|\nabla\phi_h^{n-2}\|^2 \notag\\
& \quad + \fz{1}{2} \Bigl( \|\nabla(\phi_h^n-\phi_h^{n-1})\|^2 - \|\nabla(\phi_h^{n-1}-\phi_h^{n-2})\|^2 \Bigr) \Bigr], \notag\\
\lA F^n, \bar{D}_{\Delta t}^{(2)}\phi_h^n \rA & = ( F^n, \bar{D}_{\Delta t}^{(2)}\phi_h^n ) \le \|F^n\|^2 + \fz{1}{4}\|\bar{D}_{\Delta t}^{(2)}\phi_h^n\|^2, \notag\\
\| I_{h1}^n \|
& \le c_0 ( \|\phi_h^{n-1}\|_1 + \|\phi_h^{n-2}\|_1 ) \quad \mbox{(by Lem.~\ref{lem:comp_funcs}-\eqref{ieq:v-vX_1})}, \notag\\
\|I_{h2}^n\|
& \le c_1 \bigl( \| \phi_h^{n-1} \| + \| \phi_h^{n-2} \| \bigr)
 \quad \mbox{(cf.~\eqref{ieq:proof_stability_i_I_h2}),} \notag\\
\bigl\lA I_h^n, \bar{D}_{\Delta t}^{(2)}\phi_h^n \bigr\rA & \le \|I_h^n\|^2 + \fz{1}{4}\|\bar{D}_{\Delta t}^{(2)}\phi_h^n\|^2 
\le c_1 ( \|\phi_h^{n-1}\|_1^2 + \|\phi_h^{n-2}\|_1^2 ) + \fz{1}{4}\|\bar{D}_{\Delta t}^{(2)}\phi_h^n\|^2 \notag\\
& = c_1 \Bigl( \|\nabla \phi_h^{n-1}\|^2 + \|\nabla \phi_h^{n-2}\|^2 \Bigr) + c_1^\prime \sum_{i=1}^2 \| \phi_h^{n-i}\|^2 + \fz{1}{4}\|\bar{D}_{\Delta t}^{(2)}\phi_h^n\|^2.
\label{ieq:proof_stability_ii_I_h}
\end{align}
From the inequality~\eqref{ieq:proof_stability_ii_1}, applying Lemma~\ref{lem:gronwall} with
\begin{align*}
x_n & = \fz{\nu}{2} \|\nabla\phi_h^n\|^2, & 
y_n & = \fz{\nu}{2}\|\nabla(\phi_h^n-\phi_h^{n-1})\|^2, & 
z_n & = \fz{1}{2} \|\bar{D}_{\Delta t}^{(2)}\phi_h^n\|^2,\\
a_0 & = 0, & 
a_1 & = a_2 = \fz{c_1}{\nu}, & 
b_n & = \|F^n\|^2 + c_1^\prime \sum_{i=1}^2 \|\phi_h^{n-i}\|^2,
\end{align*}
and using the result of~(i), we obtain
\begin{align*}
\max_{n=2,\ldots,N_T} \nu \| \nabla \phi_h^n \|^2
+ \Delta t\sum_{n=2}^{N_T} \| \bar{D}_{\Delta t}^{(2)} \phi_h^n \|^2 \le c_{1,\nu,T} \left( \| \phi_h^0 \|_{H^1(\Omega)}^2 + \| \phi_h^1 \|_{H^1(\Omega)}^2 + \Delta t \sum_{n=2}^{N_T} \| F^n \|^2 \right),
\end{align*}
which completes the proof of~$(ii)$.
\end{proof}
\begin{proof}[Proof of Theorem~\ref{thm:stability}]
We employ Proposition~\ref{prop:stability} for the proof.
For the first step, $n=1$, scheme~\eqref{scheme:eq1} can be written as
\begin{align}
\bigl( \bar{D}_{\Delta t}^{(1)}\phi_h^1, \psi_h \bigr) + a_0(\phi_h^1, \psi_h) 
& = \lA F^1, \psi_h \rA + \lA I_h^1, \psi_h \rA, \quad \forall \psi_h \in \Psi_h,
\label{scheme:eq1m}
\end{align}
where $I_h^1 \in \Psi_h^\prime$ is defined by
\begin{align*}
I_h^1
& \defeq 
- \fz{1}{\Delta t} \bigl( \phi_h^0 - \phi_h^0 \circ X_1^1 \gamma^1 \bigr).
\end{align*}
\par
We first prove~$(i)$.
Substituting $\phi_h^1 \in \Psi_h$ into $\psi_h$ in~\eqref{scheme:eq1m}, and noting that 
\begin{align*}
\bigl( \bar{D}_{\Delta t}^{(1)}\phi_h^1, \phi_h^1 \bigr) 
& = \fz{1}{\Delta t} \biggl( \fz{1}{2} \|\phi_h^1\|^2 - \fz{1}{2} \|\phi_h^0\|^2 + \fz{1}{2} \|\phi_h^1-\phi_h^0\|^2 \biggr) 
\ge \fz{1}{\Delta t} \biggl( \fz{1}{2} \|\phi_h^1\|^2 - \fz{1}{2} \|\phi_h^0\|^2 \biggr), \\
a_0(\phi_h^1, \phi_h^1) & = \nu \|\nabla\phi_h^1\|^2, \notag \\
\lA F^1, \phi_h^1 \rA 
& \le \fz{1}{8} \|\phi_h^1\|^2 + \fz{\nu}{4}\|\nabla\phi_h^1\|^2 + c_\nu \|F^1\|_{\Psi_h^\prime}^2 \quad \mbox{(cf.~\eqref{ieq:proof_stability_i_F_h}),} \notag \\
\lA I_h^1, \phi_h^1 \rA 
& \le \fz{1}{4}\|\phi_h^1\|^2 + \fz{\nu}{4} \| \nabla\phi_h^1 \|^2 + \fz{c_{1,\nu}}{2} \|\phi_h^0\|^2
\quad
\mbox{(cf.~\eqref{ieq:proof_stability_i_I_h}),}
\end{align*}
we have
\begin{align*}
\fz{1}{\Delta t} \biggl( \fz{1}{2} \|\phi_h^1\|^2 - \fz{1}{2} \|\phi_h^0\|^2 \biggr) + \fz{\nu}{2} \|\nabla\phi_h^1\|^2 
& \le \fz{1}{2} \|\phi_h^1\|^2 + \fz{c_{1,\nu}}{2} \|\phi_h^0\|^2 + c_\nu \|F^1\|_{\Psi_h^\prime}^2,
\end{align*}
which implies
\begin{align}
\|\phi_h^1\|^2 + \nu \Delta t \|\nabla\phi_h^1\|^2 
& \le c_{1,\nu} \Bigl( \|\phi_h^0\|^2 + \Delta t \|F^1\|_{\Psi_h^\prime}^2 \Bigr).
\label{ieq:proof_stability_i_3}
\end{align}
The result~\eqref{ieq:stability_L2} is obtained by combining~\eqref{ieq:proof_stability_i_3} with Proposition~\ref{prop:stability}-$(i)$. 
%
%
%
%
\par
We next prove~$(ii)$.
Substituting $\bar{D}_{\Delta t}^{(1)}\phi_h^n \in \Psi_h$ into $\psi_h$ in~\eqref{scheme:eq1m}, 
%
%
%
and noting that 
\begin{align*}
\bigl( \bar{D}_{\Delta t}^{(1)}\phi_h^1, \bar{D}_{\Delta t}^{(1)}\phi_h^1 \bigr) 
& = \| \bar{D}_{\Delta t}^{(1)}\phi_h^1 \|^2, \\
a_0 \bigl( \phi_h^1, \bar{D}_{\Delta t}^{(1)}\phi_h^1 \bigr) 
& \ge \fz{1}{\Delta t} \biggl( \fz{\nu}{2} \|\nabla\phi_h^1\|^2 - \fz{\nu}{2} \|\nabla\phi_h^0\|^2 \biggr), \\
\lA F^1, \bar{D}_{\Delta t}^{(1)}\phi_h^1 \rA & = (F^1, \bar{D}_{\Delta t}^{(1)}\phi_h^1) \le \|F^1\|^2 + \fz{1}{4}\|\bar{D}_{\Delta t}^{(1)}\phi_h^1\|^2, \\
\lA I_h^1, \bar{D}_{\Delta t}^{(1)}\phi_h^1 \rA 
& \le c_1 \bigl( \|\nabla \phi_h^0\|^2 + \|\phi_h^0\|^2 \bigr) + \fz{1}{4}\|\bar{D}_{\Delta t}^{(1)}\phi_h^1\|^2
\quad \mbox{(cf.~\eqref{ieq:proof_stability_ii_I_h}),}
\end{align*}
we have
\begin{align*}
\fz{1}{\Delta t} \biggl( \fz{\nu}{2} \|\nabla \phi_h^1\|^2 - \fz{\nu}{2} \|\nabla \phi_h^0\|^2 \biggr) + \fz{1}{2} \|\bar{D}_{\Delta t}^{(1)}\phi_h^1\|^2 
& \le \fz{c_1}{\nu} \Bigl( \fz{\nu}{2}\| \nabla \phi_h^0 \|^2 \Bigr) + \|F^1\|^2 + c_1^\prime \|\phi_h^0\|^2,
\end{align*}
which implies
\begin{align}
\nu \|\nabla\phi_h^1\|^2 + \Delta t \|\bar{D}_{\Delta t}^{(1)}\phi_h^1\|^2 
& \le c_{1,\nu} \Bigl( \|\phi_h^0\|_{H^1(\Omega)}^2 + \Delta t \|F^1\|^2 \Bigr), \notag
\intertext{and, by taking into account~\eqref{ieq:proof_stability_i_3} with~$g=0$,}
\|\phi_h^1\|_{H^1(\Omega)}^2 + \Delta t \|\bar{D}_{\Delta t}^{(1)}\phi_h^1\|^2 
& \le c_{1,\nu} \Bigl( \|\phi_h^0\|_{H^1(\Omega)}^2 + \Delta t \|F^1\|^2 \Bigr).
\label{ieq:proof_stability_ii_3}
\end{align}
The result~\eqref{ieq:stability_H1} is obtained by combining~\eqref{ieq:proof_stability_ii_3} with Proposition~\ref{prop:stability}-$(ii)$.
%
%
%
\end{proof}
\subsection{Proof of Theorem~\ref{thm:error_estimates}}
%
%
%
%
%
Error estimates for the Poisson projection are summarized in the following lemma.
\begin{Lem}[ \cite{Cia-1978} ]
\label{lem:projection}
Let $\Psi_h$ be the finite element space defined in~\eqref{def:Psi_h} with polynomial degree~$k\in\N$.
Then, we have the following. \smallskip \\
(i)
There exists a positive constant~$c$ independent of $h$ such that
\[
\| \hat{\psi}_h - \psi\|_{H^1(\Omega)} \le c h^k \| \psi \|_{H^{k+1}(\Omega)}, \quad \forall \psi \in H^{k+1}(\Omega).
\]
(ii)
Under Hypothesis~\ref{hyp:A-N}, there exists a positive constant~$c^\prime$ independent of $h$ such that
\[
\| \hat{\psi}_h - \psi\| \le c^\prime h^{k+1} \| \psi \|_{H^{k+1}(\Omega)}, \quad \forall \psi \in H^{k+1}(\Omega).
\]
\end{Lem}
\par
The next lemma shows the truncation error of second order in time for the time-discretization of $\pz\phi/\pz t + \nabla\cdot ( u\phi )$, and plays an important role in the proof of Theorem~\ref{thm:error_estimates}.
\begin{Lem}[truncation error]\label{lem:truncation}
Suppose that Hypothesis~\ref{hyp:u} holds true.
Assume $\phi \in Z^3$.
Suppose that Hypothesis~\ref{hyp:dt} holds true.
Then, there exists a positive constant $c=c_1$ independent of $\Delta t$ such that
%
%
%
\begin{align}
\Bigl\| \mathcal{A}_{\Delta t} \phi^n - \Bigl[ \prz{\phi}{t} +\nabla\cdot \bigl( u\phi \bigr) \Bigr] (\cdot,t^n) \Bigr\| 
\le c \Delta t^{3/2} \|\phi\|_{Z^3(t^{n-2},t^n)},
\quad n\in \{2, \ldots, N_T\}.
\label{ieq:truncation}
\end{align}
\end{Lem}
\begin{proof}
Let $n\in \{2, \ldots, N_T\}$ be fixed arbitrarily.
From a simple calculation, the two Jacobians, $\gamma^n$ and $\tilde{\gamma}^n$, are written as
\begin{subequations}\label{eqns:gamma}
\begin{align}
\gamma^n(x)  & = 1 - \Delta t \nabla\cdot u^n(x) + \Delta t^2 \delta_1^n(x) + \Delta t^3 \delta_2^n(x), \\
\tilde{\gamma}^n(x) & = 1 - 2\Delta t \nabla\cdot u^n(x) + (2\Delta t)^2 \delta_1^n(x) + (2\Delta t)^3 \delta_2^n(x),
\end{align}
\end{subequations}
where $\delta_i: \Omega \times (0,T) \to \R$, $i = 1, 2$, are defined by
\begin{align*}
\delta_1 & \defeq
\left\{
\begin{aligned}
& \cB \ 0, & \cB d & \cB = 1, \\
& \prz{u_1}{x_1}\prz{u_2}{x_2} - \prz{u_1}{x_2}\prz{u_2}{x_1}, & d & = 2, \\
& \prz{u_1}{x_1}\prz{u_2}{x_2} + \prz{u_2}{x_2}\prz{u_3}{x_3} + \prz{u_3}{x_3}\prz{u_1}{x_1} - \prz{u_1}{x_2}\prz{u_2}{x_1} - \prz{u_2}{x_3}\prz{u_3}{x_2} {\cR - \prz{u_3}{x_1}\prz{u_1}{x_3},} & d & = 3,
\end{aligned}
\right. \\
\delta_2 & \defeq
\left\{
\begin{aligned}
& \ 0, & d & = {\cB 1,\,} 2, \\
& - \prz{u_1}{x_1}\prz{u_2}{x_2}\prz{u_3}{x_3}
- \prz{u_1}{x_2}\prz{u_2}{x_3}\prz{u_3}{x_1} 
- \prz{u_1}{x_3}\prz{u_2}{x_1}\prz{u_3}{x_2} \\
& \qquad\qquad\qquad
+ \prz{u_1}{x_1}\prz{u_2}{x_3}\prz{u_3}{x_2}
+ \prz{u_1}{x_3}\prz{u_2}{x_2}\prz{u_3}{x_1}
+ \prz{u_1}{x_2}\prz{u_2}{x_1}\prz{u_3}{x_3},
& \quad\  d & = 3,
\end{aligned}
\right.
\end{align*}
with the estimates $\| \delta_i \|_{C(L^\infty)} \le c_1$, $i=1, 2$.
The relations~\eqref{eqns:gamma} imply the key identity
\begin{align}
& \fz{1}{2\Delta t} \Bigl( 3\phi^n - 4\phi^{n-1} \circ X_1^n \gamma^n + \phi^{n-2} \circ \tilde{X}_1^n \tilde\gamma^n \Bigr) - \Bigl[ \prz{\phi}{t} +\nabla\cdot (u\phi) \Bigr] (\cdot, t^n) \notag\\
& = \Bigl[ \fz{1}{2\Delta t} \Bigl( 3\phi^n - 4\phi^{n-1} \circ X_1^n + \phi^{n-2} \circ \tilde{X}_1^n \Bigr) - \Bigl( \prz{\phi^n}{t} + u^n\cdot \nabla \phi^n \Bigr) \Bigr] \notag\\
& \quad
+ (\nabla\cdot u^n) \bigl[ \bigl( 2 \phi^{n-1} \circ X_1^n -\phi^{n-2} \circ \tilde{X}_1^n \bigr) - \phi^n \bigr] 
\notag\\
& \quad
-2\Delta t \delta_1^n \bigl[ \phi^{n-1} \circ X_1^n -\phi^{n-2} \circ \tilde{X}_1^n \bigr] 
-2\Delta t^2 \delta_2^n \bigl[ \phi^{n-1} \circ X_1^n - {\cR 2} \phi^{n-2} \circ \tilde{X}_1^n \bigr] \notag\\
& \eqdef \sum_{i=1}^4 I_i^n.
\label{eq:sum_I}
\end{align}
Let us introduce the notations
\begin{align*}
y(x,s) & = y(x,s;n) \defeq X_1 \bigl( u^n, (1-s)\Delta t \bigr)(x) = x - u^n(x) (1-s)\Delta t, \\
t(s) & = t(s;n) \defeq t^{n-1}+s\Delta t.
\end{align*}
Applying the identities
\begin{align*}
\rho^\prime(1) - \Bigl[ \fz{3}{2}\rho(1)-2\rho(0)+\fz{1}{2}\rho(-1) \Bigr]
& = 2\int_0^1sds\int_{2s-1}^s\rho^{\prime\prime\prime}(s_1)ds_1,\\
\rho(1)-2\rho(0)+\rho(-1) & = \int_0^1ds\int_{s-1}^s \rho^{\prime\prime}(s_1)ds_1, \\
\rho(0)-\rho(-1) & = \int_{-1}^0 \rho^{\prime}(s)ds
\end{align*}
for ${\cR \rho} (s) = \phi (y(\cdot,s),t(s))$
we have the next expressions of $O(\Delta t^2)$,
%
%
\begin{align*}
I_1^n (x) & = -2\Delta t^2 \int_0^1 s ds \int_{2s-1}^s \Bigl[ \Bigl( \prz{}{t}+u^n(x)\cdot\nabla \Bigr)^3\phi \Bigr] \bigl( y(x,s_1), t(s_1) \bigr) \, ds_1, \\
%
I_2^n (x) & = -\Delta t^2 (\nabla\cdot u^n)(x) \int_0^1ds\int_{s-1}^{s} \Bigl[ \Bigl( \prz{}{t}+u^n(x)\cdot\nabla \Bigr)^2\phi \Bigr] \bigl( y(x,s_1, t(s_1) \bigr) \, ds_1, \\
%
I_3^n (x) & = -2\Delta t^2 \delta_1(x) \int_{-1}^0 \Bigl[ \Bigl( \prz{}{t}+u^n(x)\cdot\nabla \Bigr) \phi \Bigr] \bigl( y(x,s), t(s) \bigr) \, ds.
\end{align*}
%
We evaluate $\|I_i^n\|_{L^2(\Omega)}$, $i=1,\ldots,4$, as follows:
\begin{subequations}\label{ieqs:I1I2I3I4}
\begin{align}
\|I_1^n\| & = 2\Delta t^2 \biggl\| \int_0^1 s ds \int_{2s-1}^s
\Bigl[ \Bigl( \prz{}{t}+u^n(\cdot)\cdot\nabla \Bigr)^3\phi \Bigr] \bigl( y(\cdot, s_1), t(s_1) \bigr) \, ds_1 \biggr\| \notag\\
& \le c_0 \Delta t^2 \int_0^1 s ds \int_{2s-1}^s \Bigl\| \Bigl[ \Bigl( \prz{}{t} + 1\cdot\nabla \Bigr)^3\phi \Bigr] \bigl( y(\cdot, s_1), t(s_1) \bigr) \Bigr\| \, ds_1 \quad \mbox{($1\cdot\nabla = \sum_{i=1}^d \prz{}{x_i}$)} \notag\\
& \le c_1 \Delta t^2 \int_0^1 s ds\int_{2s-1}^s \Bigl\| \Bigl[ \Bigl( \prz{}{t}+ 1\cdot\nabla \Bigr)^3\phi \Bigr] \bigl(\, \cdot\, , t(s_1) \bigr) \Bigr\| \, ds_1 \quad \mbox{(by Prop.~\ref{prop:bijective_jacobian})} \notag\\
& \le c_1^\prime\Delta t \int_{t^{n-2}}^{t^n} \Bigl\| \Bigl[ \Bigl( \prz{}{t}+ 1\cdot\nabla \Bigr)^3\phi \Bigr] (\, \cdot\, , t) \Bigr\| \, dt \notag\\
& \le \sqrt{2} \, c_1^\prime \Delta t^{3/2} \Bigl\| \Bigl( \prz{}{t} + 1\cdot\nabla \Bigr)^3\phi \Bigr\|_{L^2(t^{n-2},t^n; L^2)} \notag\\
& \le c_1^{\prime\prime} \Delta t^{3/2} \| \phi \|_{Z^3(t^{n-2},t^n)}, \\
\|I_2^n\|
& \le c_1\Delta t^2 \int_0^1ds\int_{s-1}^s \Bigl\|\Bigl[ \Bigl( \prz{}{t} + 1\cdot\nabla \Bigr)^2\phi \Bigr] \bigl( y(\cdot,s_1), t(s_1) \bigr) \Bigr\| \, ds_1 \notag\\
& \le c_1^\prime \Delta t \int_{t^{n-2}}^{t^n} \Bigl\|\Bigl[ \Bigl( \prz{}{t} + 1\cdot\nabla \Bigr)^2\phi \Bigr] ( \, \cdot\, , t ) \Bigr\| ds_1 
\le c_1^{\prime\prime} \Delta t^{3/2} \| \phi \|_{Z^2(t^{n-2},t^n)}, \\
\|I_3^n\|
& \le c_1 \Delta t^{3/2} \| \phi \|_{Z^1(t^{n-2},t^n)}, \\
%
\|I_4^n\|
& \le c_1 \Delta t^2 ( \|\phi^{n-1}\| + \|\phi^{n-2}\| )
\le c_1^\prime \Delta t^{3/2} \|\phi\|_{H^1(t^{n-2},t^n;L^2)},
\end{align}
\end{subequations}
where for the last inequality in the estimate of~$\|I_4^n\|$, we have employed the inequality,
\begin{align*}
\|\phi^{n-1}\| + \|\phi^{n-2}\|
\le
c\Delta t^{-1/2} \|\phi\|_{H^1(t^{n-2},t^n;L^2)}.
\end{align*}
From the identity~\eqref{eq:sum_I} and estimates~\eqref{ieqs:I1I2I3I4}, we obtain
\begin{align*}
\mbox{LHS of~\eqref{ieq:truncation}}
& \le \sum_{i=1}^4 \|I_i^n\|_{L^2(\Omega)}
\le c_1 \Delta t^{3/2} \| \phi \|_{Z^3(t^{n-2},t^n)},
\end{align*}
which completes the proof.
\end{proof}
\begin{Rmk}[ \cite{RuiTab-2010} ]
\label{rmk:truncation_first_order}
For any $n\in \{1, \ldots, N_T\}$, there exists a positive constant~$c = c_1$ independent of~$\Delta t$ such that
%
\begin{align}
\biggl\| \mathcal{A}_{\Delta t}^{(1)} \phi^n - \Bigl[ \prz{\phi}{t} +\nabla\cdot \bigl( u\phi \bigr) \Bigr] (\cdot,t^n) \biggr\| \le c \Delta t^{1/2} \|\phi\|_{Z^2(t^{n-1},t^n)}
\ \Bigl( \le c^\prime \Delta t \|\phi\|_{Z^3} \Bigr).
\label{ieq:truncation_first_order}
\end{align}
%
%
\end{Rmk}
\begin{Rmk}
\label{rmk:error_estimates_BDF2}
Lemma~\ref{lem:truncation} and Remark~\ref{rmk:truncation_first_order} with $u = 0$ imply that
\begin{align*}
\Bigl\| \bar{D}_{\Delta t}\phi^n -\prz{\phi}{t} \Bigr\|
& \le \left\{
\begin{aligned}
& c \Delta t^{1/2} \|\phi\|_{H^2(t^0,t^1;L^2)} \le c^\prime \Delta t \|\phi\|_{H^3(0,T;L^2)} && (n = 1), \\
& c^{\prime\prime} \Delta t^{3/2} \|\phi\|_{H^3(t^{n-2},t^n;L^2)} && (n \ge 2).
\end{aligned}
\right.
\end{align*}
\end{Rmk}
\par
Before the proof of Theorem~\ref{thm:error_estimates}, we prepare notations, equations and two lemmas to be employed.
Let $\{\phi(t) = \phi(\cdot,t) \in \Psi;\ t\in [0,T]\}$ be the solution to problem~\eqref{prob:weak}, and for each~$t\in [0, T]$, let $\hat{\phi}_h(t) = \hat{\phi}_h (\cdot, t) \in \Psi_h$ be the Poisson projection to~$\phi(t)$, cf.~Definition~\ref{def:Poisson-Proj}.
Let $\{\phi_h^n\}_{n=1}^{N_T} \subset \Psi_h$ be the solution to scheme~\eqref{scheme} with~$\phi_h^0 = \hat{\phi}_h^0 \in \Psi_h$.
We introduce the two functions $e_h^n$ and $\eta(t)$ defined by
\begin{align*}
e_h^n & \defeq \phi_h^n - \hat\phi_h^n \in \Psi_h, &
\eta (t) & \defeq \phi(t) - \hat\phi_h (t) \in \Psi 
\end{align*}
for $n\in\{0, \ldots, N_T\}$ and $t\in [0, T]$.
Then, the series $\{e_h^n\}_{n=0}^{N_T} \subset \Psi_h$ satisfies
\begin{align}
\bigl( \mathcal{A}_{\Delta t} e_h^n, \psi_h \bigr) + a_0( e_h^n, \psi_h ) =  \lA R_h^n, \psi_h \rA, \quad \forall \psi_h \in \Psi_h
\label{eq:error}
\end{align}
for $n \in \{ 1, \ldots, N_T \}$,
%
%
where $R_h^n \in \Psi_h^\prime$ is defined by
\begin{align*}
R_h^n & \defeq \sum_{i=1}^3 R_{hi}^n, \\
R_{h1}^n & \defeq 
\left\{
\begin{aligned}
& \prz{\phi^n}{t} + \nabla\cdot (u^n\phi^n) - \fz{\phi_h^n-\phi_h^{n-1}\circ X_1^n \gamma^n}{\Delta t}, & n & = 1, \\
& \prz{\phi^n}{t} + \nabla\cdot (u^n\phi^n) - \fz{3\phi_h^n-4\phi_h^{n-1}\circ X_1^n \gamma^n + \phi_h^{n-2}\circ \tilde{X}_1^n \tilde{\gamma}^n}{2\Delta t}, & n & \ge 2,
\end{aligned}
\right. \\
R_{h2}^n & \defeq 
\left\{
\begin{aligned}
& \fz{\eta^n-\eta^{n-1}\circ X_1^n \gamma^n}{\Delta t}, & n & = 1, \\
& \fz{3\eta^n-4\eta^{n-1}\circ X_1^n \gamma^n + \eta^{n-2}\circ \tilde{X}_1^n \tilde{\gamma}^n}{2\Delta t}, & n & \ge 2,
\end{aligned}
\right. \\
R_{h3}^n & \defeq -\eta^n.
\end{align*}
\par
We summarize some estimates to be used in the proof of Theorem~\ref{thm:error_estimates} in the next two lemmas. Their proofs are given in Appendix~\ref{subsection:lemma_R} and~\ref{subsection:lemma_eh1}.
The first lemma provides estimates for $R_h^n$ and $\eta^n$ and the second lemma provides estimates for $e_h^1$.
\begin{Lem}
\label{lem:R}
Suppose that Hypotheses~\ref{hyp:u}, \ref{hyp:dt} and~\ref{hyp:phi} hold true.
Assume $\Delta t \in {\cR (0, 1)}$.
Then, we have the following.
\smallskip\\
(i)~It holds that
\begin{subequations}
\begin{align}
\|\eta(\cdot,t)\|
& \le \|\eta(\cdot,t)\|_{H^1(\Omega)}
\le c h^k \Delta t^{-1/2} \|\phi\|_{H^1(t^{n-1},t^n; H^{k+1})} 
\le {\cB c^\prime} h^k \|\phi\|_{H^2(H^{k+1})} \notag\\
& \qquad\qquad\qquad\qquad\qquad\qquad\qquad\quad {\cB (t\in [t^{n-1},t^n] \cap [0,T], n\in\N),}
\label{ieq:eta_t_L2} \\
\|\bar{D}_{\Delta t}\eta^n\|
& \le \left\{
\begin{aligned}
& c h^k \Delta t^{-1/2} \|\phi\|_{H^1(t^0,t^1;H^{k+1})} && (n = 1), \\
& c^\prime h^k \Delta t^{-1/2} \|\phi\|_{H^1(t^{n-2},t^n;H^{k+1})} && (n \ge 2),
\end{aligned}
\right.
\label{ieq:D_dt_eta} \\
\|R_{h1}^n\|_{\Psi_h^\prime}
& \le \|R_{h1}^n\|
\le \left\{
\begin{aligned}
& c_1 \Delta t^{1/2} \|\phi\|_{Z^2(t^0,t^1)} && (n = 1), \\
& c_1^\prime \Delta t^{3/2} \|\phi\|_{Z^3(t^{n-2},t^n)} && (n \ge 2),
\end{aligned}
\right. 
\label{ieq:R_h1_n} \\
\|R_{h2}^n\|_{\Psi_h^\prime}
& \le \left\{
\begin{aligned}
& c_1 h^k \Delta t^{-1/2} \|\phi\|_{H^1(t^0,t^1;H^{k+1})} && (n = 1), \\
& c_1^\prime h^k \Delta t^{-1/2} \|\phi\|_{H^1(t^{n-2},t^n;H^{k+1})} && (n \ge 2),
\end{aligned}
\right.
\label{ieq:R_h2_n} \\
\|R_{h3}^n\|_{\Psi_h^\prime}
& \le \|R_{h3}^n\|
\le c h^k \Delta t^{-1/2} \|\phi\|_{H^1(t^{n-1},t^n; H^{k+1})} \quad (n \ge 1),
\label{ieq:R_h3_n} \\
\|R_{h2}^n\|
& \le \left\{
\begin{aligned}
& c_1 h^k \Delta t^{-1/2} \|\phi\|_{H^1(t^0,t^1;H^{k+1})} && (n = 1), \\
& c_1^\prime h^k \Delta t^{-1/2} \|\phi\|_{H^1(t^{n-2},t^n;H^{k+1})} && (n \ge 2).
\end{aligned}
\right.
\label{ieq:R_h2_n_L2}
\end{align}
\end{subequations}
(ii)~Under Hypothesis~\ref{hyp:A-N}, the estimates of~$\|\eta(\cdot,t)\|$, $\|R_{h2}^n\|_{\Psi_h^\prime}$ and~$\|R_{h3}^n\|_{\Psi_h^\prime}$ are given as
\begin{subequations}
\begin{align}
\|\eta(\cdot,t)\|
& \le c h^{k+1} \Delta t^{-1/2} \|\phi\|_{H^1(t^{n-1},t^n; H^{k+1})} 
\le {\cB c^\prime} h^{k+1} \|\phi\|_{H^2(H^{k+1})} \notag\\
& \qquad\qquad\qquad\qquad\qquad 
{\cB (t\in [t^{n-1},t^n] \cap [0,T], n\in\N),}
\label{ieq:eta_t_L2_with_hyp4}\\
\|R_{h2}^n\|_{\Psi_h^\prime}
& \le \left\{
\begin{aligned}
& c_1 h^{k+1} \Delta t^{-1/2} \|\phi\|_{H^1(t^0,t^1;H^{k+1})} && (n = 1), \\
& c_1^\prime h^{k+1} \Delta t^{-1/2} \|\phi\|_{H^1(t^{n-2},t^n;H^{k+1})} && (n \ge 2),
\end{aligned}
\right. 
\label{ieq:R_h2_n_with_hyp4} \\
\|R_{h3}^n\|_{\Psi_h^\prime}
& \le \|R_{h3}^n\|
\le c h^{k+1} \|\phi\|_{H^1(t^{n-1},t^n; H^{k+1})} \quad (n \ge 1).
\label{ieq:R_h3_n_with_hyp4}
\end{align}
\end{subequations}
\end{Lem}
\begin{Rmk}
Hypotheses~\ref{hyp:u} and~\ref{hyp:dt} are not needed for the estimates of~\eqref{ieq:eta_t_L2}, \eqref{ieq:D_dt_eta}, \eqref{ieq:R_h3_n}, \eqref{ieq:eta_t_L2_with_hyp4} and~\eqref{ieq:R_h3_n_with_hyp4}.
\end{Rmk}
\begin{Lem}
\label{lem:eh1}
Suppose that Hypotheses~\ref{hyp:u}, \ref{hyp:dt} and~\ref{hyp:phi} hold true.
Then, we have the following.
\begin{subequations}
\begin{align}
\|e_h^1\|
\le \|e_h^1\| + \sqrt{\nu \Delta t} \, \|\nabla e_h^1\|
& \le c_1 (\Delta t^2 + h^{k+1}) \|\phi\|_{Z^3\cap H^2(H^{k+1})},
\label{ieq:eh1_L2} \\
\sqrt{\nu} \, \| \nabla e_h^1 \| + \sqrt{\Delta t} \, \| \bar{D}_{\Delta t}^{(1)} e_h^1 \|
& \le c_1 (\Delta t^{3/2} + h^k) \|\phi\|_{Z^3\cap H^2(H^{k+1})}.
\label{ieq:eh1_H1_D_dt_eh1_L2}
\end{align}
\end{subequations}
\end{Lem}
\par
Now, we give the proof of the error estimates.
\begin{proof}[Proof of Theorem~\ref{thm:error_estimates}]
Considering the equation~\eqref{eq:error} for~$e_h$, applying Proposition~\ref{prop:stability}-\textit{(i)} and~\textit{(ii)}, and taking into account the fact $e_h^0 = 0$, we have
\begin{align}
\|e_h\|_{\ell^\infty_2(L^2)} + \sqrt{\nu} \|\nabla e_h\|_{\ell^2_2(L^2)}
& \le c_\dagger \left( \|e_h^1\| + \|R_h\|_{\ell^2_2(\Psi_h^\prime)} \right),
\label{ieq:proof_convergence_L2}\\
\sqrt{\nu} \|\nabla e_h\|_{\ell^\infty_2(L^2)} 
+ \|\bar{D}_{\Delta t} e_h\|_{\ell^2_2(L^2)}
& \le \bar{c}_\dagger \left( \|e_h^1\|_{H^1(\Omega)} + \|R_h\|_{\ell^2_2(L^2)} \right).
\label{ieq:proof_convergence_H1}
\end{align}
\par
We prove~(i).
From Lemma~\ref{lem:R}-(i), it holds that:
\begin{align}
\|R_{h1}\|_{\ell^2_2(\Psi_h^\prime)}
& \le \|R_{h1}\|_{\ell^2_2(L^2)}
= \Bigl( \Delta t \sum_{n=2}^{N_T} \|R_{h1}^n\|^2 \Bigr)^{1/2}
\le c_1 \Bigl( \Delta t \sum_{n=2}^{N_T} \Delta t^3 \|\phi\|_{Z^3(t^{n-2},t^n)}^2 \Bigr)^{1/2} \notag\\
& \le c_1 \bigl( 2 \Delta t^4 \|\phi\|_{Z^3}^2 \bigr)^{1/2} 
= c_1^\prime \Delta t^2 \|\phi\|_{Z^3}
\quad \mbox{($c_1^\prime = \sqrt{2}c_1$),} \notag\\
\|R_{h2}\|_{\ell^2_2(\Psi_h^\prime)} & 
= \Bigl( \Delta t \sum_{n=2}^{N_T} \|R_{h2}^n\|_{\Psi_h^\prime}^2 \Bigr)^{1/2} \notag\\
& \le c_1 \Bigl[ \Delta t \sum_{n=2}^{N_T} \bigl( h^k \Delta t^{-1/2} \|\phi\|_{H^1(t^{n-2},t^n; H^{k+1})} 
 \bigr)^2 \Bigr]^{1/2} 
\notag\\
& \le c_{1,T} h^k \|\phi\|_{H^1(H^{k+1})}, \notag\\
\|R_{h2}\|_{\ell^2_2(L^2)} & 
= \Bigl( \Delta t \sum_{n=2}^{N_T} \|R_{h2}^n\|^2 \Bigr)^{1/2} 
\le c_{1,T} h^k \|\phi\|_{H^1(H^{k+1})}, \notag\\ 
\|R_{h3}\|_{\ell^2_2(\Psi_h^\prime)}
& \le \|R_{h3}\|_{\ell^2_2(L^2)}
= \Bigl( \Delta t \sum_{n=2}^{N_T} \|\eta^n\|^2 \Bigr)^{1/2} 
\le c \Bigl[ \Delta t \sum_{n=2}^{N_T} \bigl( h^k \|\phi\|_{H^1(0,T; H^{k+1}(\Omega))} \bigr)^2 \Bigr]^{1/2} \notag\\
& \le c_T h^k \|\phi\|_{H^1(H^{k+1})} 
\quad \mbox{($c_T = c T^{1/2}$)}, \notag\\
\|R_h\|_{\ell^2_2(\Psi_h^\prime)} 
& \le \sum_{i=1}^3 \|R_{hi}\|_{\ell^2_2(\Psi_h^\prime)} 
\le c_{1,T} \bigl( \Delta t^2 + h^k \bigr) \|\phi\|_{Z^3\cap H^1(H^{k+1})}, 
\label{ieq:proof_convergence_L2_Rh_l22_Psi_prime} \\
\|R_h\|_{\ell^2_2(L^2)} 
& \le \sum_{i=1}^3 \|R_{hi}\|_{\ell^2_2(L^2)} 
\le c_{1,T} \bigl( \Delta t^2 + h^k \bigr) \|\phi\|_{Z^3\cap H^1(H^{k+1})}. 
\label{ieq:proof_convergence_L2_Rh_l22_L2}
\end{align}
Combining~\eqref{ieq:eh1_L2} and~\eqref{ieq:proof_convergence_L2_Rh_l22_Psi_prime} with~\eqref{ieq:proof_convergence_L2}, we obtain
\begin{align*}
\|e_h\|_{\ell^\infty(L^2)} + \sqrt{\nu} \, \|\nabla e_h\|_{\ell^2(L^2)}
& \le \bigl( \|e_h^1\| + \sqrt{\nu \Delta t} \, \|\nabla e_h^1 \| \bigr) + \|e_h\|_{\ell^\infty_2(L^2)} + \sqrt{\nu} \, \|\nabla e_h\|_{\ell^2_2(L^2)} \\
& \le \bigl( \|e_h^1\| + \sqrt{\nu \Delta t} \, \|\nabla e_h^1 \| \bigr) + c_\dagger \Bigl( \|e_h^1\| + \|R_h\|_{\ell^2_2(\Psi_h^\prime)} \Bigr) \\
& \le c_{1,\nu,T} ( \Delta t^2 + h^k ) \|\phi\|_{Z^3\cap H^2(H^{k+1})},
\end{align*}
which implies the  error estimate~\eqref{ieq:thm_convergence_i_a} of~\textit{(i)}, as
\begin{align*}
& \| \phi_h - \phi \|_{\ell^\infty(L^2)} + \sqrt{\nu} \, \|\nabla (\phi_h - \phi)\|_{\ell^2(L^2)} \\
& \le \| e_h \|_{\ell^\infty(L^2)} + \| \eta \|_{\ell^\infty(L^2)} + \sqrt{\nu} (\|\nabla e_h\|_{\ell^2(L^2)} + \|\nabla \eta \|_{\ell^2(L^2)}) \\
& \le \| e_h \|_{\ell^\infty(L^2)} + \sqrt{\nu} \|\nabla e_h\|_{\ell^2(L^2)} + c_T h^k \|\phi\|_{H^1(H^{k+1})} \quad \mbox{(by~\eqref{ieq:eta_t_L2})}\notag\\
& \le c_{1,\nu,T} ( \Delta t^2 + h^k ) \|\phi\|_{Z^3\cap H^2(H^{k+1})}.
\end{align*}
\par
For the  error estimate~\eqref{ieq:thm_convergence_i_b}, we have
\begin{align}
& \sqrt{\nu} \, \| \nabla e_h \|_{\ell^\infty(L^2)} + \bigl\| \bar{D}_{\Delta t} e_h \bigr\|_{\ell^2(L^2)} \notag \\
& \le \bigl( \sqrt{\nu} \, \| \nabla e_h^1 \| + \sqrt{\Delta t} \, \bigl\| \bar{D}_{\Delta t}^{(1)}e_h^1 \bigr\| \bigr) + \sqrt{\nu} \, \| \nabla e_h \|_{\ell^\infty_2(L^2)} + \bigl\| \bar{D}_{\Delta t} e_h \bigr\|_{\ell^2_2(L^2)} \notag \\
& \le \bigl( \sqrt{\nu} \, \| \nabla e_h^1 \| + \sqrt{\Delta t} \, \bigl\| \bar{D}_{\Delta t}^{(1)}e_h^1 \bigr\| \bigr) + \bar{c}_\dagger \bigl( \| e_h^1 \|_{H^1(\Omega)} + \|R_h\|_{\ell^2_2(L^2)} \bigr) \quad \mbox{(by~\eqref{ieq:proof_convergence_H1})} \notag \\
& \le c_{1,\nu,T} (\Delta t^{3/2} + h^k) \|\phi\|_{Z^3\cap H^2(H^{k+1})} \quad \mbox{(by~\eqref{ieq:eh1_H1_D_dt_eh1_L2}, \eqref{ieq:eh1_L2} and~\eqref{ieq:proof_convergence_L2_Rh_l22_L2})}.
\label{ieq:eh_liH10_D_dt_eh_l2L2}
\end{align}
Noting the estimate
\begin{align}
\Bigl\| \bar{D}_{\Delta t} \phi - \prz{\phi}{t} \Bigr\|_{\ell^2(L^2)} 
& = \sqrt{ \Delta t \, \Bigl\| \bar{D}_{\Delta t}^{(1)} \phi^1 - \prz{\phi^1}{t} \Bigr\|^2 + \Delta t \sum_{n=2}^{N_T} \Bigl\| \bar{D}_{\Delta t}^{(2)} \phi^n - \prz{\phi^n}{t} \Bigr\|^2 } \notag\\
& \le \sqrt{ c (\Delta t^3 + \Delta t^4) \| \phi \|_{H^3(L^2)}^2 } \qquad \mbox{(cf. Rmk.~\ref{rmk:error_estimates_BDF2})} \notag\\
& \le c^\prime \Delta t^{3/2} \| \phi \|_{H^3(L^2)},
\label{ieq:truncation_D_dt_phi_l2L2}
\end{align}
we obtain the  estimate~\eqref{ieq:thm_convergence_i_b} of~\textit{(i)}, as 
\begin{align*}
& \sqrt{\nu} \, \| \nabla (\phi_h - \phi) \|_{\ell^\infty(L^2)} + \Bigl\| \bar{D}_{\Delta t} \phi_h - \prz{\phi}{t} \Bigr\|_{\ell^2(L^2)} \\
& \le \sqrt{\nu} \, \Bigl( \| \nabla e_h \|_{\ell^\infty(L^2)} + \| \nabla \eta \|_{\ell^\infty(L^2)} \Bigr) + \bigl\| \bar{D}_{\Delta t} e_h \bigr\|_{\ell^2(L^2)} + \bigl\| \bar{D}_{\Delta t} \eta \bigr\|_{\ell^2(L^2)} + \Bigl\| \bar{D}_{\Delta t} \phi - \prz{\phi}{t} \Bigr\|_{\ell^2(L^2)} \\
& \le c_{1,\nu, T} ( \Delta t^{3/2} + h^k ) \| \phi \|_{Z^3\cap H^2(H^{k+1})}  + \Bigl\| \bar{D}_{\Delta t} \phi - \prz{\phi}{t} \Bigr\|_{\ell^2(L^2)} 
\quad \mbox{(by~\eqref{ieq:eh_liH10_D_dt_eh_l2L2}, \eqref{ieq:eta_t_L2} and~\eqref{ieq:D_dt_eta})} \\
& \le c_{1,\nu, T}^\prime ( \Delta t^{3/2} + h^k ) \| \phi \|_{Z^3\cap H^2(H^{k+1})} \quad \mbox{(by~\eqref{ieq:truncation_D_dt_phi_l2L2})}.
\end{align*}
\par
We next prove~\textit{(ii)}.
Under Hypothesis~\ref{hyp:A-N}, we have, from Lemma~\ref{lem:R}-(ii),
\begin{align}
\|R_{h2}\|_{\ell^2_2(\Psi_h^\prime)}
& \le c_{1,T} h^{k+1} \|\phi\|_{H^1(H^{k+1})}, \notag\\
\|R_{h3}\|_{\ell^2_2(\Psi_h^\prime)}
& \le c_T h^{k+1} \|\phi\|_{H^1(H^{k+1})}, \notag\\
\|R_h\|_{\ell^2_2(\Psi_h^\prime)} 
& \le \sum_{i=1}^3 \|R_{hi}\|_{\ell^2_2(\Psi_h^\prime)} 
\le c_{1,T} \bigl( \Delta t^2 + h^{k+1} \bigr) \|\phi\|_{Z^3\cap H^1(H^{k+1})}. 
\label{ieq:proof_convergence_L2_Rh_l22_Psi_prime_with_hyp4}
\end{align}
Combining~\eqref{ieq:eh1_L2} and~\eqref{ieq:proof_convergence_L2_Rh_l22_Psi_prime_with_hyp4} with~\eqref{ieq:proof_convergence_L2} and taking into account Lemma~\ref{lem:projection}-(ii),
we obtain
\begin{align*}
\|\phi_h - \phi\|_{\ell^\infty(L^2)}
& \le \|e_h\|_{\ell^\infty(L^2)} + \|\eta\|_{\ell^\infty(L^2)} \\
& \le \max \bigl\{ \|e_h^1\|, \|e_h\|_{\ell^\infty_2(L^2)} \bigr\} + c h^{k+1} \|\phi\|_{H^1(H^{k+1})} \\
& \le c_{1,\nu,T} ( \Delta t^2 + h^{k+1} ) \|\phi\|_{Z^3\cap H^2(H^{k+1})},
\end{align*}
which completes the proof of~\textit{(ii)}.
\end{proof}
%
%
%
%
\section{Numerical results}\label{sec:numerics}
%
%
In this section we verify the theoretical orders of convergence from Theorem~\ref{thm:error_estimates} in numerical experiments. To this end we solved an example problem by scheme~\eqref{scheme} in a finite element space of polynomial order $k=1$. {\cR As initial data we set $\phi_h^0 = \Pi_h\phi^0$ using the Lagrange interpolation operator $\Pi_h\colon C(\bar\Omega) \to \Psi_h$, and note that this choice of~$\phi_h^0$ does not cause any loss of convergence order in Theorem~\ref{thm:error_estimates}.} For the computation of the integrals appearing in the scheme we employed numerical quadrature formulae of degree nine for $d=1$ (five points) and degree five for $d=2$ (seven points) and $d=3$ (fifteen points)~\cite{Str-1971}.
{\cR While higher order quadrature formulae can improve numerical results of Lagrange-Galerkin methods, cf., e.g.,~\cite{BermejoSaavedra2012,ColeraCarpioBermejo2021},  we do not consider them in this paper.}
The linear systems were solved using the conjugate gradient method and meshes were generated using FreeFem++~\cite{FreeFem}.
\begin{Ex}\label{ex:1}
In problem~\eqref{prob:strong}, for $d=1,2,3$, we set $\Omega = (-1,1)^d$, $T=0.5$, $f=0$, $g=0$, and
\begin{align*}
u(x,t) &= \sum_{i=1}^d (1+\sin(t-x_i))e_i,
\end{align*}
where $\{e_i\}_{i=1}^d \subset \R^d$ is the standard basis in $\R^d$.
The function~$\phi^0$ is given according to the exact solution
\[
\phi(x,t) = \prod_{i=1}^d \exp\left(-\frac{1-\cos(t-x_i)}{\nu}\right).
\]
\cR
The viscosity constant is set $\nu = 10^{-2}$ if not otherwise noted.
\end{Ex}
%
%
%
%
\par
We applied scheme~\eqref{scheme} to Example~\ref{ex:1} and computed the errors
\[
E_Y\defeq \frac{\| \phi_h - \Pi_h \phi\|_Y}{\| \Pi_h \phi \|_Y}
\]
for $Y=\ell^\infty(L^2)$, $\ell^2(H^1_0)$, $\ell^{\infty}(H^1_0)$, where $\| \phi \|_{\ell^2(H^1_0)} \defeq  \| \nabla \phi\|_{\ell^2(L^2)}$, $\| \phi \|_{\ell^\infty(H^1_0)} \defeq  \| \nabla \phi\|_{\ell^\infty(L^2)}$ and $\Pi_h:C(\bar \Omega) \rightarrow \Psi_h$ is the Lagrange interpolation operator.
Tables~\ref{tab:first}--\ref{tab:last} show the errors and the corresponding experimental orders of convergence (EOCs)\footnote{We used the formula $\text{EOC}=\log (E_2/E_1) / \log (\Delta t_2/\Delta t_1)$ for errors $E_1$, $E_2$ and time increments $\Delta t_1$, $\Delta t_2$ from two consecutive table rows.} after grid refinement.
The number $N$ in the tables denotes the division number of the domain in each space dimension determining the mesh, whose size is taken as  $h \defeq 2/N$.
We coupled time increment and mesh size by $\Delta t = c h^p$ and varied the constant $c$ and the exponent $p$ in the tables to see the theoretical convergence orders.
According to Theorem~\ref{thm:error_estimates} we expected to see experimental convergence orders 2 ($E_{\ell^\infty(L^2)}$), 1 ($E_{\ell^2(H^1_0)}$) and 1 ($E_{\ell^\infty(H^1_0)}$) for $p=1$, 2 ($E_{\ell^\infty(L^2)}$), 2 ($E_{\ell^2(H^1_0)}$) and 3/2 ($E_{\ell^\infty(H^1_0)}$) for $p=1/2$ and 2 ($E_{\ell^\infty(L^2)}$), 3/2 ($E_{\ell^2(H^1_0)}$) and 3/2 ($E_{\ell^\infty(H^1_0)}$) for $p=2/3$.
The EOCs in the tables either agree with or exceed our expectations and therefore support our theoretical results.
\cR
To see $\Delta t$-convergence for a fixed~$h~(=2/256)$ and $h$-convergence for a fixed~$\Delta t~(=0.01)$, we present Tables~\ref{tab:N_fixed} and~\ref{tab:dt_fixed}, respectively, which further support the convergence rates in Theorem~\ref{thm:error_estimates}.
\cK
The tables, i.e., Tables~\ref{tab:first}--\ref{tab:dt_fixed}, moreover show a low relative loss of mass,
\[
  E_{\text{mass}} \defeq \frac{\left|\int_\Omega\phi_h^{N_T}\, dx - \int_\Omega\Pi_h \phi^{N_T}\, dx\,\right| }{ \left| \int_\Omega\Pi_h \phi^{N_T}\, dx \, \right|},
\]
which decreases as the mesh is refined. \cR
Furthermore we computed the error formulas
\[
  E_{\text{mass}}^\prime \defeq \frac{\left| \int_\Omega \phi_h^{N_T} dx - \int_\Omega\phi_h^0\, dx\,\right| }{ \left| \int_\Omega \phi_h^0\, dx \, \right|},
  \quad
  E_{\text{mass}}^{\prime\prime} \defeq \frac{\Delta t \sum_{n=1}^{N_T} \, \left| \int_\Omega\phi_h^n\, dx - \int_\Omega \Pi_h \phi^n\, dx\,\right| }{ \Delta t \sum_{n=1}^{N_T} \, \left| \int_\Omega\Pi_h \phi^n \, dx \, \right|},
\]
for $\Delta t = 4h$ shown in Table~\ref{tab:2d_mass} providing additional information on the error of mass within the computation and throughout all time steps.
Both $E_{\text{mass}}^\prime$ and~$E_{\text{mass}}^{\prime\prime}$ also decrease as the mesh is refined.
\cK
{\cR These results indicate} that mass is lost only due to numerical integration {\cR and Lagrange interpolation of the exact solution} and thus support the mass-{\cR preserving} property of the scheme~(Theorem~\ref{thm:mass-conservation}).
\cR
When the viscosity~$\nu$ is decreased to $\nu = 10^{-3}$ or $10^{-4}$, we observe a reduction in the EOC in~$\ell^\infty(L^2)$ to orders smaller than~$2$ for some~$N$ but still larger than~$1$, and almost no effect in the EOCs in~$\ell^2(H^1_0)$ and~$\ell^\infty(H^1_0)$, as we show in~Tables~\ref{tab:nu_10m3} and~\ref{tab:nu_10m4}.
\cK
We further present numerical solutions for $d=2$ and~$3$ in~Fig.~\ref{fig:numres}.
%
%
%
%
%
%
%
\begin{figure}
   \newlength{\fwidth}\setlength{\fwidth}{.295\linewidth}
  \begin{tabular}{ccc}
  \includegraphics[width=\fwidth]{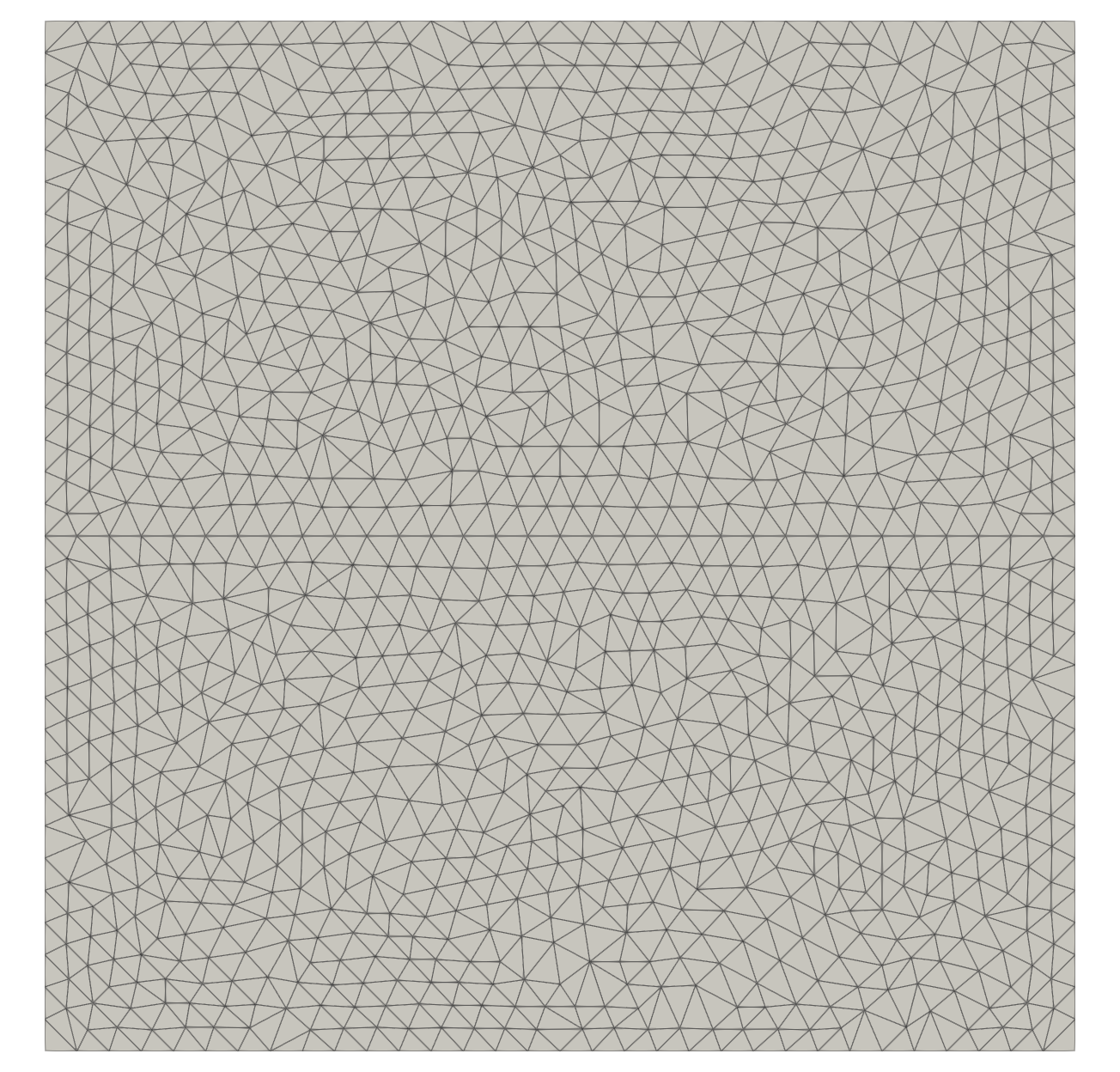} & \includegraphics[width=\fwidth]{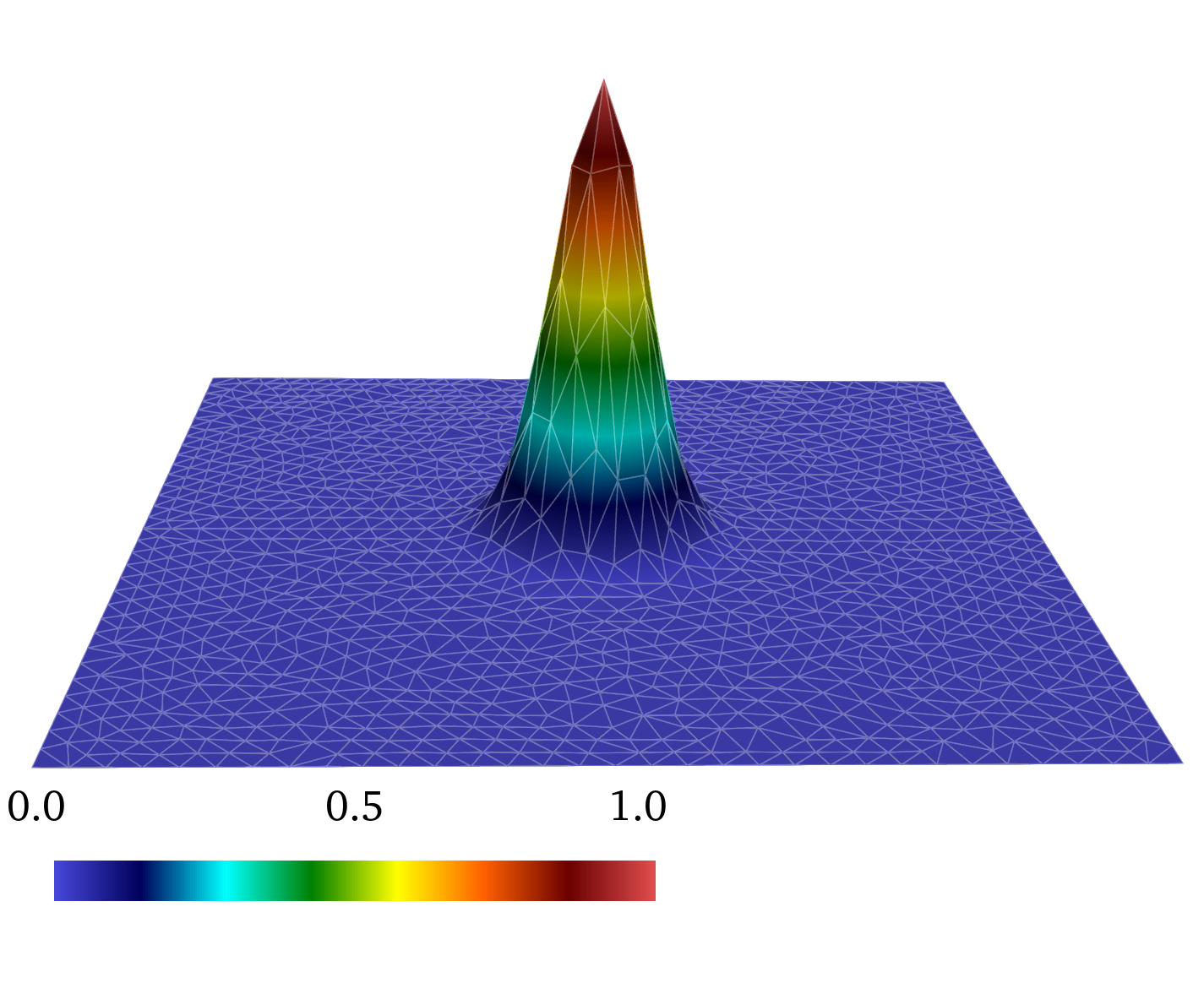} & \includegraphics[width=\fwidth]{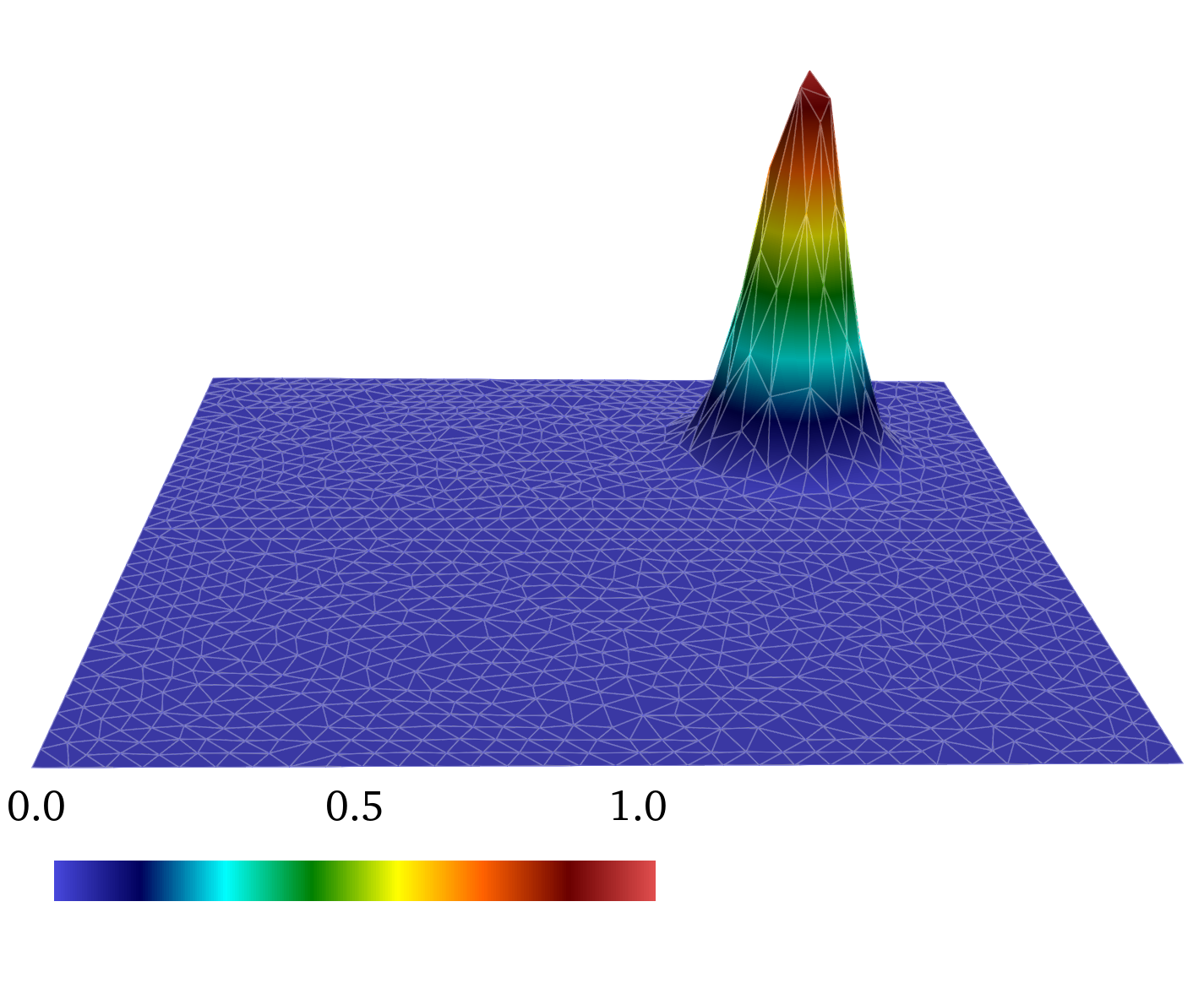}\\[2ex]
    \includegraphics[width=\fwidth]{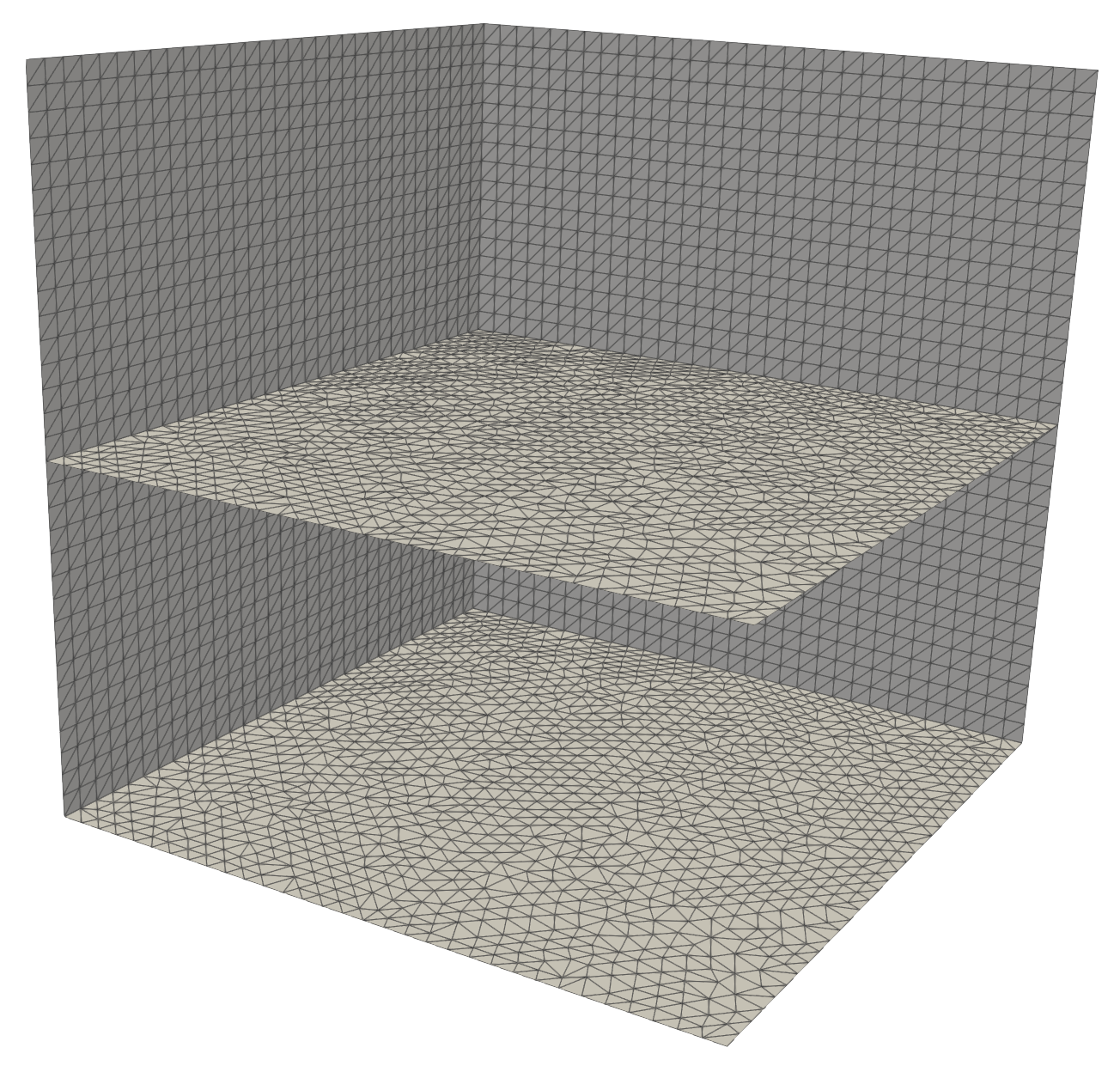} & \includegraphics[width=.92\fwidth]{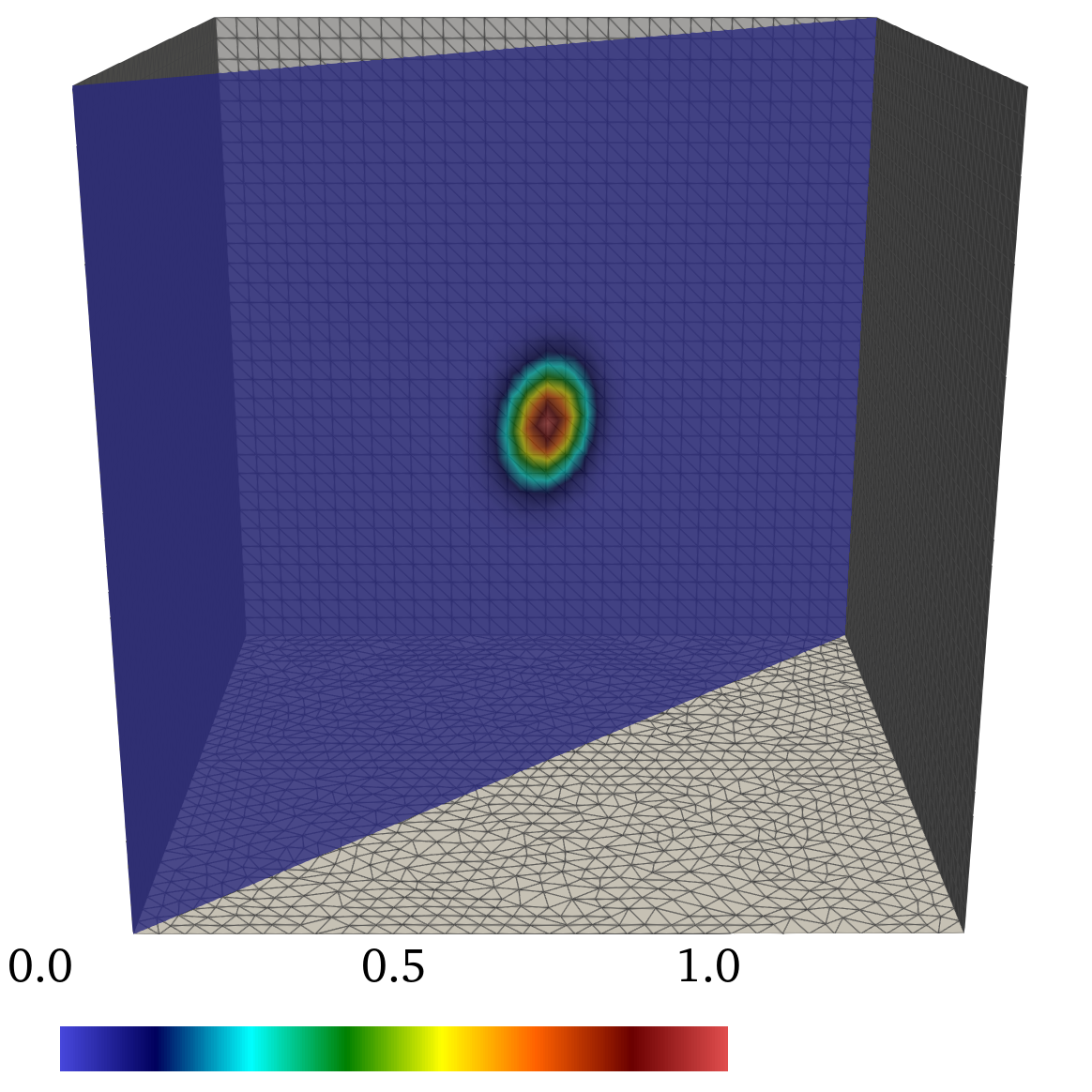} & \includegraphics[width=.92\fwidth]{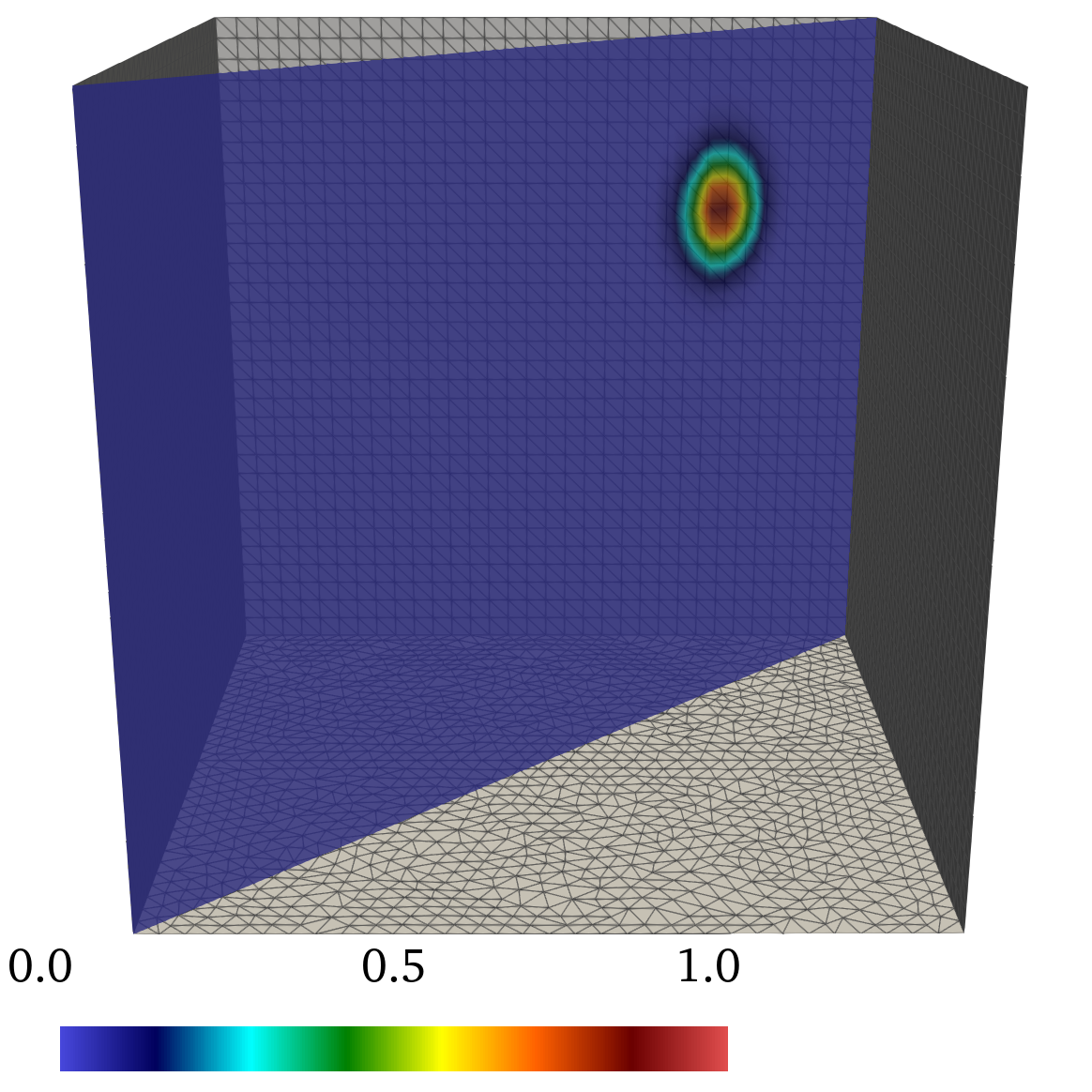} \\
\end{tabular}
\caption{Employed meshes (left column) and numerical solutions of Example~\ref{ex:1} for $d=2$ (top row) and $d=3$ (bottom row). Initial conditions (middle column) and numerical solutions at the final time $T=0.5$ (right column) computed by scheme~\eqref{scheme} are shown.}\label{fig:numres}
\end{figure}
\begin{table}[!htbp]
\caption{Relative errors and EOCs for $\Delta t=4h$ in 1D~$(d=1)$.}\label{tab:first}
\centering
\scriptsize
\begin{tabular}{rrrrrrrrr}
\toprule
$N$ &\multicolumn{1}{c}{$\Delta t$} &$E_{\ell ^\infty(L^2)}$ &EOC &$E_{\ell ^2(H^1_0)}$ & EOC &$E_{\ell^\infty(H^1_0)}$ &EOC &$E_{\text{mass}}$ \\
\midrule
$32$ & $2.50\times10^{-1}$ & $2.49\times 10^{-2}$ & --- & $4.05\times 10^{-2}$ & --- & $4.36\times 10^{-2}$ & --- & $1.39\times10^{-3}$ \\
$64$ & $1.25\times10^{-1}$ & $9.02\times 10^{-3}$ & $1.46$ & $1.51\times 10^{-2}$ & $1.43$ & $1.60\times 10^{-2}$ & $1.45$ & $5.06\times10^{-5}$ \\
$128$ & $6.25\times10^{-2}$ & $2.80\times 10^{-3}$ & $1.69$ & $4.65\times 10^{-3}$ & $1.69$ & $5.68\times 10^{-3}$ & $1.49$ & $1.93\times10^{-5}$ \\
$256$ & $3.12\times10^{-2}$ & $8.09\times 10^{-4}$ & $1.79$ & $1.31\times 10^{-3}$ & $1.83$ & $1.80\times 10^{-3}$ & $1.65$ & $1.39\times10^{-6}$ \\
$512$ & $1.56\times10^{-2}$ & $2.22\times 10^{-4}$ & $1.86$ & $3.47\times 10^{-4}$ & $1.91$ & $5.29\times 10^{-4}$ & $1.76$ & $1.53\times10^{-6}$ \\
$1{,}024$ & $7.81\times10^{-3}$ & $5.93\times 10^{-5}$ & $1.91$ & $9.16\times 10^{-5}$ & $1.92$ & $1.47\times 10^{-4}$ & $1.85$ & $8.20\times10^{-8}$ \\
$2{,}048$ & $3.91\times10^{-3}$ & $1.54\times 10^{-5}$ & $1.95$ & $2.52\times 10^{-5}$ & $1.86$ & $3.96\times 10^{-5}$ & $1.90$ & $7.99\times10^{-8}$ \\
$4{,}096$ & $1.95\times10^{-3}$ & $3.95\times 10^{-6}$ & $1.96$ & $8.04\times 10^{-6}$ & $1.64$ & $1.05\times 10^{-5}$ & $1.90$ & $8.89\times10^{-8}$ \\
$8{,}192$ & $9.77\times10^{-4}$ & $1.00\times 10^{-6}$ & $1.98$ & $3.25\times 10^{-6}$ & $1.31$ & $6.32\times 10^{-6}$ & $0.74$ & $9.19\times10^{-8}$\\
\bottomrule
\end{tabular}
\end{table}
\begin{table}[!htbp]
\caption{Relative errors and EOCs for $\Delta t=0.4\sqrt{h}$ in 1D~$(d=1)$.}
\centering
\scriptsize
\begin{tabular}{rrrrrrrrr}
\toprule
 $N$ &\multicolumn{1}{c}{$\Delta t$} &$E_{\ell ^\infty(L^2)}$ &EOC &$E_{\ell ^2(H^1_0)}$ &EOC &$E_{\ell^\infty(H^1_0)}$ &EOC &$E_{\text{mass}}$\\
\midrule
$32$ & $1.00\times10^{-1}$ & $9.51\times 10^{-3}$ & --- & $2.51\times 10^{-2}$ & --- & $3.36\times 10^{-2}$ & --- & $2.72\times10^{-4}$  \\
$64$ & $7.07\times10^{-2}$ & $3.00\times 10^{-3}$ & $3.33$ & $5.71\times 10^{-3}$ & $4.27$ & $6.25\times 10^{-3}$ & $4.85$ & $4.87\times10^{-5}$ \\
$128$ & $5.00\times10^{-2}$ & $1.82\times 10^{-3}$ & $1.43$ & $2.96\times 10^{-3}$ & $1.89$ & $3.79\times 10^{-3}$ & $1.45$ & $9.76\times10^{-7}$ \\
$256$ & $3.54\times10^{-2}$ & $1.02\times 10^{-3}$ & $1.67$ & $1.69\times 10^{-3}$ & $1.63$ & $2.26\times 10^{-3}$ & $1.50$ & $5.33\times10^{-8}$ \\
$512$ & $2.50\times10^{-2}$ & $5.49\times 10^{-4}$ & $1.79$ & $9.11\times 10^{-4}$ & $1.77$ & $1.25\times 10^{-3}$ & $1.70$ & $3.95\times10^{-8}$ \\
$1{,}024$ & $1.77\times10^{-2}$ & $2.87\times 10^{-4}$ & $1.88$ & $4.78\times 10^{-4}$ & $1.87$ & $6.81\times 10^{-4}$ & $1.76$ & $6.71\times10^{-8}$ \\
$2{,}048$ & $1.25\times10^{-2}$ & $1.49\times 10^{-4}$ & $1.89$ & $2.45\times 10^{-4}$ & $1.92$ & $3.61\times 10^{-4}$ & $1.82$ & $9.11\times10^{-8}$ \\
$4{,}096$ & $8.84\times10^{-3}$ & $7.64\times 10^{-5}$ & $1.92$ & $1.25\times 10^{-4}$ & $1.95$ & $1.88\times 10^{-4}$ & $1.88$ & $7.40\times10^{-8}$ \\
$8{,}192$ & $6.25\times10^{-3}$ & $3.90\times 10^{-5}$ & $1.94$ & $6.29\times 10^{-5}$ & $1.97$ & $9.72\times 10^{-5}$ & $1.91$ & $9.49\times10^{-8}$ \\
\bottomrule
\end{tabular}
\end{table}
\begin{table}[!htbp]
\caption{Relative errors and EOCs for $\Delta t=h^{2/3}$ in 1D~$(d=1)$.}
\centering
\scriptsize
\begin{tabular}{rrrrrrrrrr}
\toprule
$N$ &\multicolumn{1}{c}{$\Delta t$} &$E_{\ell ^\infty(L^2)}$ &EOC &$E_{\ell ^2(H^1_0)}$ &EOC &$E_{\ell^\infty(H^1_0)}$ &EOC &$E_{\text{mass}}$ \\
\midrule
$32$ & $1.57\times10^{-1}$ & $1.16\times 10^{-2}$ & --- & $2.19\times 10^{-2}$ & --- & $2.61\times 10^{-2}$ & --- & $5.43\times10^{-5}$ \\
$64$ & $9.92\times10^{-2}$ & $5.89\times 10^{-3}$ & $1.47$ & $9.83\times 10^{-3}$ & $1.74$ & $1.06\times 10^{-2}$ & $1.96$ & $1.22\times10^{-5}$ \\
$128$ & $6.25\times10^{-2}$ & $2.80\times 10^{-3}$ & $1.61$ & $4.65\times 10^{-3}$ & $1.62$ & $5.68\times 10^{-3}$ & $1.36$ & $1.93\times10^{-5}$  \\
$256$ & $3.94\times10^{-2}$ & $1.26\times 10^{-3}$ & $1.73$ & $2.12\times 10^{-3}$ & $1.71$ & $2.74\times 10^{-3}$ & $1.58$ & $6.67\times10^{-7}$  \\
$512$ & $2.48\times10^{-2}$ & $5.41\times 10^{-4}$ & $1.82$ & $8.98\times 10^{-4}$ & $1.85$ & $1.23\times 10^{-3}$ & $1.73$ & $5.75\times10^{-8}$ \\
$1{,}024$ & $1.56\times10^{-2}$ & $2.27\times 10^{-4}$ & $1.87$ & $3.73\times 10^{-4}$ & $1.90$ & $5.43\times 10^{-4}$ & $1.77$ & $2.47\times10^{-7}$  \\
$2{,}048$ & $9.84\times10^{-3}$ & $9.39\times 10^{-5}$ & $1.91$ & $1.53\times 10^{-4}$ & $1.94$ & $2.30\times 10^{-4}$ & $1.86$ & $6.44\times10^{-8}$  \\
$4{,}096$ & $6.20\times10^{-3}$ & $3.83\times 10^{-5}$ & $1.94$ & $6.16\times 10^{-5}$ & $1.97$ & $9.57\times 10^{-5}$ & $1.90$ & $7.90\times10^{-8}$  \\
$8{,}192$ & $3.91\times10^{-3}$ & $1.55\times 10^{-5}$ & $1.96$ & $2.47\times 10^{-5}$ & $1.98$ & $3.94\times 10^{-5}$ & $1.92$ & $9.47\times10^{-8}$  \\
\bottomrule
\end{tabular}
\end{table}
\begin{table}[!htbp]
\caption{Relative errors and EOCs for $\Delta t=4h$ in 2D~$(d=2)$.}
\centering
\scriptsize
\begin{tabular}{rrrrrrrrrr}
\toprule
$N$ &\multicolumn{1}{c}{$\Delta t$} &$E_{\ell ^\infty(L^2)}$ &EOC &$E_{\ell ^2(H^1_0)}$ &EOC &$E_{\ell^\infty(H^1_0)}$ &EOC &$E_{\text{mass}}$\\
\midrule
$32$ & $2.50\times10^{-1}$ & $4.27\times 10^{-2}$ & --- & $7.29\times 10^{-2}$ & --- & $7.62\times 10^{-2}$ & --- & $3.66\times10^{-3}$ \\
$64$ & $1.25\times10^{-1}$ & $1.42\times 10^{-2}$ & $1.59$ & $2.90\times 10^{-2}$ & $1.33$ & $3.10\times 10^{-2}$ & $1.30$ & $1.37\times10^{-3}$ \\
$128$ & $6.25\times10^{-2}$ & $4.42\times 10^{-3}$ & $1.69$ & $1.20\times 10^{-2}$ & $1.28$ & $1.36\times 10^{-2}$ & $1.19$ & $9.18\times10^{-5}$  \\
$256$ & $3.12\times10^{-2}$ & $1.28\times 10^{-3}$ & $1.78$ & $4.54\times 10^{-3}$ & $1.40$ & $5.31\times 10^{-3}$ & $1.36$ & $2.26\times10^{-5}$  \\
$512$ & $1.56\times10^{-2}$ & $3.63\times 10^{-4}$ & $1.82$ & $2.45\times 10^{-3}$ & $0.89$ & $2.92\times 10^{-3}$ & $0.86$ & $5.31\times10^{-6}$  \\
$1{,}024$ & $7.81\times10^{-3}$ & $9.78\times 10^{-5}$ & $1.89$ & $1.11\times 10^{-3}$ & $1.14$ & $1.41\times 10^{-3}$ & $1.05$ & $1.36\times10^{-6}$  \\
$2{,}048$ & $3.91\times10^{-3}$ & $2.57\times 10^{-5}$ & $1.93$ & $5.62\times 10^{-4}$ & $0.98$ & $7.04\times 10^{-4}$ & $1.01$ & $6.97\times10^{-7}$  \\
\bottomrule
\end{tabular}
\end{table}
\begin{table}[!htbp]
\caption{Relative errors and EOCs for $\Delta t=0.4\sqrt{h}$ in 2D~$(d=2)$.}
\centering
\scriptsize
\begin{tabular}{rrrrrrrrr}
\toprule
$N$ &\multicolumn{1}{c}{$\Delta t$} &$E_{\ell ^\infty(L^2)}$ &EOC &$E_{\ell ^2(H^1_0)}$ &EOC &$E_{\ell^\infty(H^1_0)}$ &EOC &$E_{\text{mass}}$ \\
\midrule
$32$ & $1.00\times10^{-1}$ & $2.06\times 10^{-2}$ & --- & $5.52\times 10^{-2}$ & --- & $7.34\times 10^{-2}$ & --- & $2.41\times10^{-3}$ \\
$64$ & $7.07\times10^{-2}$ & $5.57\times 10^{-3}$ & $3.77$ & $2.25\times 10^{-2}$ & $2.58$ & $2.61\times 10^{-2}$ & $2.99$ & $8.04\times10^{-4}$  \\
$128$ & $5.00\times10^{-2}$ & $3.00\times 10^{-3}$ & $1.79$ & $1.10\times 10^{-2}$ & $2.07$ & $1.32\times 10^{-2}$ & $1.96$ & $1.25\times10^{-4}$  \\
$256$ & $3.54\times10^{-2}$ & $1.62\times 10^{-3}$ & $1.78$ & $4.75\times 10^{-3}$ & $2.42$ & $5.43\times 10^{-3}$ & $2.57$ & $2.47\times10^{-5}$  \\
$512$ & $2.50\times10^{-2}$ & $8.80\times 10^{-4}$ & $1.76$ & $2.69\times 10^{-3}$ & $1.64$ & $3.05\times 10^{-3}$ & $1.67$ & $9.88\times10^{-6}$  \\
$1{,}024$ & $1.77\times10^{-2}$ & $4.66\times 10^{-4}$ & $1.84$ & $1.26\times 10^{-3}$ & $2.18$ & $1.49\times 10^{-3}$ & $2.07$ & $2.67\times10^{-6}$  \\
$2{,}048$ & $1.25\times10^{-2}$ & $2.43\times 10^{-4}$ & $1.87$ & $6.45\times 10^{-4}$ & $1.94$ & $7.43\times 10^{-4}$ & $2.00$ & $8.11\times10^{-7}$  \\
\bottomrule
\end{tabular}
\end{table}
\begin{table}[!htbp]
\caption{Relative errors and EOCs for $\Delta t=h^{2/3}$ in 2D~$(d=2)$.}
\centering
\scriptsize
\begin{tabular}{rrrrrrrrr}
\toprule
$N$ &\multicolumn{1}{c}{$\Delta t$} &$E_{\ell ^\infty(L^2)}$ &EOC &$E_{\ell ^2(H^1_0)}$ &EOC &$E_{\ell^\infty(H^1_0)}$ &EOC &$E_{\text{mass}}$\\
\midrule
$32$ & $1.57\times10^{-1}$ & $2.24\times 10^{-2}$ & --- & $5.09\times 10^{-2}$ & --- & $5.80\times 10^{-2}$ & --- & $1.34\times10^{-3}$ \\
$64$ & $9.92\times10^{-2}$ & $9.76\times 10^{-3}$ & $1.80$ & $2.46\times 10^{-2}$ & $1.58$ & $2.63\times 10^{-2}$ & $1.71$ & $1.23\times10^{-3}$  \\
$128$ & $6.25\times10^{-2}$ & $4.42\times 10^{-3}$ & $1.72$ & $1.19\times 10^{-2}$ & $1.57$ & $1.36\times 10^{-2}$ & $1.43$ & $3.62\times10^{-5}$  \\
$256$ & $3.94\times10^{-2}$ & $1.97\times 10^{-3}$ & $1.75$ & $5.04\times 10^{-3}$ & $1.86$ & $5.59\times 10^{-3}$ & $1.93$ & $1.91\times10^{-5}$  \\
$512$ & $2.48\times10^{-2}$ & $8.67\times 10^{-4}$ & $1.77$ & $2.68\times 10^{-3}$ & $1.37$ & $3.04\times 10^{-3}$ & $1.32$ & $9.60\times10^{-6}$  \\
$1{,}024$ & $1.56\times10^{-2}$ & $3.70\times 10^{-4}$ & $1.85$ & $1.21\times 10^{-3}$ & $1.73$ & $1.46\times 10^{-3}$ & $1.59$ & $2.91\times10^{-6}$  \\
$2{,}048$ & $9.84\times10^{-3}$ & $1.54\times 10^{-4}$ & $1.89$ & $5.95\times 10^{-4}$ & $1.53$ & $7.19\times 10^{-4}$ & $1.53$ & $7.92\times10^{-7}$  \\
\bottomrule
\end{tabular}
\end{table}
\begin{table}[!htbp]
\caption{Relative errors and EOCs for $\Delta t=2h$ in 3D~$(d=3)$.}
\centering
\scriptsize
\begin{tabular}{rrrrrrrrr}
\toprule
$N$ &\multicolumn{1}{c}{$\Delta t$} &$E_{\ell ^\infty(L^2)}$ &EOC &$E_{\ell ^2(H^1_0)}$ &EOC &$E_{\ell^\infty(H^1_0)}$ &EOC &$E_{\text{mass}}$\\
\midrule
32 & $1.25\times10^{-1}$ & $4.41\times 10^{-2}$ & --- & $8.15\times 10^{-2}$ & --- & $1.01\times 10^{-1}$ & --- & $2.70\times10^{-3}$ \\
64 & $6.25\times10^{-2}$ & $1.19\times 10^{-2}$ & 1.89 & $2.72\times 10^{-2}$ & 1.58 & $3.26\times 10^{-2}$ & 1.64 & $1.28\times10^{-3}$ \\
128 & $3.13\times10^{-2}$ & $3.04\times 10^{-3}$ & 1.97 & $1.07\times 10^{-2}$ & 1.34 & $1.29\times 10^{-2}$ & 1.33 & $1.05\times10^{-4}$ \\
256 & $1.56\times10^{-2}$ & $7.51\times 10^{-4}$ & 2.02 & $4.05\times 10^{-3}$ & 1.41 & $4.89\times 10^{-3}$ & 1.40 & $3.48\times10^{-5}$ \\
\bottomrule
\end{tabular}
\end{table}
\begin{table}[!htbp]
\caption{Relative errors and EOCs for $\Delta t=4h$ in 3D~$(d=3)$.}
\centering
\scriptsize
\begin{tabular}{rrrrrrrrr}
\toprule
$N$ &\multicolumn{1}{c}{$\Delta t$} &$E_{\ell ^\infty(L^2)}$ &EOC &$E_{\ell ^2(H^1_0)}$ &EOC &$E_{\ell^\infty(H^1_0)}$ &EOC &$E_{\text{mass}}$\\
\midrule
$32$ & $2.50\times10^{-1}$ & $6.28\times 10^{-2}$ & --- & $9.12\times 10^{-2}$ & --- & $9.79\times 10^{-2}$ & --- & $3.08\times10^{-3}$ \\
$64$ & $1.25\times10^{-1}$ & $1.92\times 10^{-2}$ & $1.71$ & $3.30\times 10^{-2}$ & $1.47$ & $3.40\times 10^{-2}$ & $1.53$ & $1.46\times10^{-3}$ \\
$128$ & $6.25\times10^{-2}$ & $5.81\times 10^{-3}$ & $1.73$ & $1.26\times 10^{-2}$ & $1.39$ & $1.35\times 10^{-2}$ & $1.33$ & $1.38\times10^{-4}$  \\
$256$ & $3.12\times10^{-2}$ & $1.78\times 10^{-3}$ & $1.70$ & $4.48\times 10^{-3}$ & $1.49$ & $5.01\times 10^{-3}$ & $1.43$ & $3.53\times10^{-5}$  \\
\bottomrule
\end{tabular}
\end{table}
\begin{table}[!htbp]
\caption{Relative errors and EOCs for $\Delta t=0.2\sqrt{h}$ in 3D~$(d=3)$.}
\centering
\scriptsize
\begin{tabular}{rrrrrrrrr}
\toprule
$N$ &\multicolumn{1}{c}{$\Delta t$} &$E_{\ell ^\infty(L^2)}$ &EOC &$E_{\ell ^2(H^1_0)}$ &EOC &$E_{\ell^\infty(H^1_0)}$ &EOC &$E_{\text{mass}}$ \\
\midrule
32 & $5.00\times10^{-2}$ & $5.33\times 10^{-2}$ & --- & $1.00\times 10^{-1}$ & --- & $1.24\times 10^{-1}$ & --- & $8.03\times10^{-3}$ \\
64 & $3.54\times10^{-2}$ & $1.25\times 10^{-2}$ & 4.19 & $2.86\times 10^{-2}$ & 3.61 & $3.46\times 10^{-2}$ & 3.69 & $1.99\times10^{-3}$ \\
128 & $2.50\times10^{-2}$ & $3.08\times 10^{-3}$ & 4.03 & $1.07\times 10^{-2}$ & 2.83 & $1.30\times 10^{-2}$ & 2.83 & $1.13\times10^{-4}$ \\
256 & $1.77\times10^{-2}$ & $8.44\times 10^{-4}$ & 3.74 & $4.08\times 10^{-3}$ & 2.79 & $4.90\times 10^{-3}$ & 2.81 & $3.15\times10^{-5}$ \\
\bottomrule
\end{tabular}
\end{table}
\begin{table}[!htbp]
\caption{Relative errors and EOCs for $\Delta t=0.4\sqrt{h}$ in 3D~$(d=3)$.}
\centering
\scriptsize
\begin{tabular}{rrrrrrrrr}
\toprule
$N$ &\multicolumn{1}{c}{$\Delta t$} &$E_{\ell ^\infty(L^2)}$ &EOC &$E_{\ell ^2(H^1_0)}$ &EOC &$E_{\ell^\infty(H^1_0)}$ &EOC &$E_{\text{mass}}$\\
\midrule
$32$ & $1.00\times10^{-1}$ & $4.60\times 10^{-2}$ & --- & $8.66\times 10^{-2}$ & --- & $1.10\times 10^{-1}$ & --- & $1.03\times10^{-3}$  \\
$64$ & $7.07\times10^{-2}$ & $1.25\times 10^{-2}$ & $3.75$ & $2.74\times 10^{-2}$ & $3.32$ & $3.30\times 10^{-2}$ & $3.46$ & $1.29\times10^{-3}$  \\
$128$ & $5.00\times10^{-2}$ & $3.91\times 10^{-3}$ & $3.36$ & $1.14\times 10^{-2}$ & $2.54$ & $1.32\times 10^{-2}$ & $2.64$ & $1.59\times10^{-4}$  \\
$256$ & $3.54\times10^{-2}$ & $2.24\times 10^{-3}$ & $1.61$ & $4.77\times 10^{-3}$ & $2.51$ & $5.13\times 10^{-3}$ & $2.73$ & $1.76\times10^{-5}$  \\
\bottomrule
\end{tabular}
\end{table}
\begin{table}[!htbp]
\caption{Relative errors and EOCs for $\Delta t=h^{2/3}$ in 3D~$(d=3)$.}\label{tab:last}
\centering
\scriptsize
\begin{tabular}{rrrrrrrrr}
\toprule
$N$ &\multicolumn{1}{c}{$\Delta t$} &$E_{\ell ^\infty(L^2)}$ &EOC &$E_{\ell ^2(H^1_0)}$ &EOC &$E_{\ell^\infty(H^1_0)}$ &EOC &$E_{\text{mass}}$ \\
\midrule
32 & $1.57\times10^{-1}$ & $4.61\times 10^{-2}$ & --- & $7.82\times 10^{-2}$ & --- & $9.86\times 10^{-2}$ & --- & $2.02\times10^{-4}$ \\
64 & $9.92\times10^{-2}$ & $1.48\times 10^{-2}$ & $2.46$ & $2.87\times 10^{-2}$ & $2.17$ & $3.19\times 10^{-2}$ & $2.44$ & $1.30\times10^{-3}$ \\
128 & $6.25\times10^{-2}$ & $5.81\times 10^{-3}$ & $2.02$ & $1.26\times 10^{-2}$ & $1.78$ & $1.35\times 10^{-2}$ & $1.86$ & $1.38\times10^{-4}$ \\
256 & $3.94\times10^{-2}$ & $2.72\times 10^{-3}$ & $1.64$ & $5.19\times 10^{-3}$ & $1.92$ & $5.81\times 10^{-3}$ & $1.82$ & $2.30\times10^{-5}$ \\
\bottomrule
\end{tabular}
\end{table}
%
%
%
%
\begin{table}[!htbp]
\cR
\caption{\cR Relative errors and EOCs for $N=256$ in 2D~$(d=2)$ for $\nu=10^{-2}$.}
\label{tab:N_fixed}
\centering
\scriptsize
\begin{tabular}{rrrrrrrrr}
\toprule
$N$ &\multicolumn{1}{c}{$\Delta t$} &$E_{\ell ^\infty(L^2)}$ &EOC &$E_{\ell ^2(H^1_0)}$ &EOC &$E_{\ell^\infty(H^1_0)}$ &EOC & $E_{\text{mass}}$ \\
\midrule
256 & $2.50\times10^{-1}$ & $4.56\times 10^{-2}$ & --- & $6.45\times 10^{-2}$ & --- & $6.63\times 10^{-2}$ & --- & $4.79\times10^{-5}$ \\
256 & $1.25\times10^{-1}$ & $1.49\times 10^{-2}$ & 1.62 & $2.27\times 10^{-2}$ & 1.51 & $2.53\times 10^{-2}$ & 1.39 & $2.41\times10^{-5}$ \\
256 & $6.25\times10^{-2}$ & $4.48\times 10^{-3}$ & 1.73 & $7.80\times 10^{-3}$ & 1.54 & $9.01\times 10^{-3}$ & 1.49 & $2.79\times10^{-5}$ \\
256 & $3.13\times10^{-2}$ & $1.28\times 10^{-3}$ & 1.80 & $4.54\times 10^{-3}$ & 0.78 & $5.31\times 10^{-3}$ & 0.76 & $2.26\times10^{-5}$ \\
256 & $1.56\times10^{-2}$ & $3.70\times 10^{-4}$ & 1.80 & $4.23\times 10^{-3}$ & 0.10 & $5.17\times 10^{-3}$ & 0.04 & $1.41\times10^{-4}$ \\
256 & $7.81\times10^{-3}$ & $6.88\times 10^{-4}$ & -0.90 & $4.23\times 10^{-3}$ & 0.00 & $5.14\times 10^{-3}$ & 0.01 & $6.38\times10^{-4}$ \\
256 & $3.91\times10^{-3}$ & $8.81\times 10^{-4}$ & -0.36 & $4.53\times 10^{-3}$ & -0.10 & $5.41\times 10^{-3}$ & -0.08 & $8.33\times10^{-5}$ \\
\bottomrule
\end{tabular}
\end{table}
\begin{table}[!htbp]
\cR
\caption{\cR Relative errors and EOCs in~$h$ (denoted by EOC$_h$ in the table) for $\Delta t=0.01$ in~2D~$(d=2)$ for $\nu=10^{-2}$.}
\label{tab:dt_fixed}
\centering
\scriptsize
\begin{tabular}{rrrrrrrrr}
\toprule
$N$ &\multicolumn{1}{c}{$\Delta t$} &$E_{\ell ^\infty(L^2)}$ & EOC$_h$ &$E_{\ell^2(H^1_0)}$ &  EOC$_h$ &$E_{\ell^\infty(H^1_0)}$ & EOC$_h$ & $E_{\text{mass}}$ \\
\midrule
32 & $0.01$ & $4.79\times 10^{-2}$ & --- & $1.07\times 10^{-1}$ & --- & $1.34\times 10^{-1}$ & --- & $5.25\times10^{-4}$ \\
64 & $0.01$ & $9.48\times 10^{-3}$ & $2.34$ & $2.98\times 10^{-2}$ & $1.84$ & $3.61\times 10^{-2}$ &  & $1.04\times10^{-3}$ \\
128 & $0.01$ & $1.79\times 10^{-3}$ & $2.40$ & $1.08\times 10^{-2}$ & $1.46$ & $1.35\times 10^{-2}$ & $1.42$ & $1.99\times10^{-5}$ \\
256 & $0.01$ & $3.14\times 10^{-4}$ & $2.51$ & $4.21\times 10^{-3}$ & $1.36$ & $5.16\times 10^{-3}$ & $1.39$ & $1.86\times10^{-4}$ \\
512 & $0.01$ & $1.53\times 10^{-4}$ & $1.04$ & $2.41\times 10^{-3}$ & $0.80$ & $2.91\times 10^{-3}$ & $0.83$ & $9.85\times10^{-6}$ \\
1024 & $0.01$ & $1.58\times 10^{-4}$ & $-0.05$ & $1.12\times 10^{-3}$ & $1.11$ & $1.42\times 10^{-3}$ & $1.04$ & $3.04\times10^{-6}$ \\
2048 & $0.01$ & $1.59\times 10^{-4}$ & $-0.01$ & $5.98\times 10^{-4}$ & $0.91$ & $7.20\times 10^{-4}$ & $0.98$ & $7.99\times10^{-7}$ \\
\bottomrule
\end{tabular}
\end{table}
\begin{table}[!htbp]
\cR
\caption{\cR Relative errors of mass for $\Delta t=4h$ in 2D~$(d=2)$.}\label{tab:2d_mass}
\centering
\scriptsize
\begin{tabular}{rrrrr}
\toprule
$N$ & \multicolumn{1}{c}{$\Delta t$} & \multicolumn{1}{c}{$E_{\text{mass}}$} & \multicolumn{1}{c}{$E_{\text{mass}}^\prime$} & \multicolumn{1}{c}{$E_{\text{mass}}^{\prime\prime}$} \\
\midrule
32 & $2.50\times10^{-1}$ & $3.66\times10^{-3}$ & $1.36\times10^{-3}$ & $3.83\times10^{-3}$\\
64 & $1.25\times10^{-1}$ & $1.37\times10^{-3}$ & $9.23\times10^{-5}$ & $1.08\times10^{-3}$\\
128 & $6.25\times10^{-2}$ & $9.18\times10^{-5}$ & $4.11\times10^{-5}$ & $1.50\times10^{-4}$\\
256 & $3.13\times10^{-2}$ & $2.26\times10^{-5}$ & $2.30\times10^{-6}$ & $2.56\times10^{-5}$\\
512 & $1.56\times10^{-2}$ & $5.31\times10^{-6}$ & $3.43\times10^{-6}$ & $6.29\times10^{-6}$\\
1024 & $7.81\times10^{-3}$ & $1.36\times10^{-6}$ & $6.49\times10^{-7}$ & $1.23\times10^{-6}$\\
2048 & $3.91\times10^{-3}$ & $6.97\times10^{-7}$ & $5.03\times10^{-7}$ & $3.45\times10^{-7}$\\
\bottomrule
\end{tabular}
\end{table}
\begin{table}[!htbp]
\cR
\caption{\cR Relative errors and EOCs for $\Delta t=4h$ in 2D~$(d=2)$ for $\nu=10^{-3}$.}
\label{tab:nu_10m3}
\centering
\scriptsize
\begin{tabular}{rrrrrrrrr}
\toprule
$N$ &\multicolumn{1}{c}{$\Delta t$} &$E_{\ell ^\infty(L^2)}$ &EOC &$E_{\ell ^2(H^1_0)}$ &EOC &$E_{\ell^\infty(H^1_0)}$ &EOC &$E_{\text{mass}}$ \\
\midrule
32 & $2.50\times10^{-1}$ & $1.91\times 10^{-1}$ & --- & $2.46\times 10^{-1}$ & --- & $2.65\times 10^{-1}$ & ---
 & $4.59\times10^{-2}$ \\
64 & $1.25\times10^{-1}$ & $4.50\times 10^{-2}$ & 2.09 & $1.02\times 10^{-1}$ & 1.26 & $1.33\times 10^{-1}$ & 0.99 & $6.15\times10^{-3}$ \\
128 & $6.25\times10^{-2}$ & $1.35\times 10^{-2}$ & 1.74 & $3.71\times 10^{-2}$ & 1.47 & $4.92\times 10^{-2}$ & 1.43 & $3.72\times10^{-3}$ \\
256 & $3.13\times10^{-2}$ & $3.25\times 10^{-3}$ & 2.05 & $1.47\times 10^{-2}$ & 1.33 & $1.98\times 10^{-2}$ & 1.31 & $1.33\times10^{-3}$ \\
512 & $1.56\times10^{-2}$ & $8.46\times 10^{-4}$ & 1.94 & $7.66\times 10^{-3}$ & 0.94 & $9.75\times 10^{-3}$ & 1.02 & $5.51\times10^{-4}$ \\
1024 & $7.81\times10^{-3}$ & $3.31\times 10^{-4}$ & 1.35 & $3.40\times 10^{-3}$ & 1.17 & $4.82\times 10^{-3}$ & 1.02 & $3.03\times10^{-4}$ \\
2048 & $3.91\times10^{-3}$ & $1.18\times 10^{-4}$ & 1.49 & $1.74\times 10^{-3}$ & 0.97 & $2.38\times 10^{-3}$ & 1.02 & $1.13\times10^{-4}$ \\
\bottomrule
\end{tabular}
\end{table}
\begin{table}[!htbp]
\cR
\caption{\cR Relative errors and EOCs for $\Delta t=4h$ in 2D~$(d=2)$ for $\nu=10^{-4}$.}
\label{tab:nu_10m4}
\centering
\scriptsize
\begin{tabular}{rrrrrrrrr}
\toprule
$N$ &\multicolumn{1}{c}{$\Delta t$} &$E_{\ell ^\infty(L^2)}$ &EOC &$E_{\ell ^2(H^1_0)}$ &EOC &$E_{\ell^\infty(H^1_0)}$ &EOC & $E_{\text{mass}}$ \\
\midrule
32 & $2.50\times10^{-1}$ & $3.27\times 10^{+1}$ & --- & $4.63\times 10^{+1}$ & --- & $3.41\times 10^{+1}$ & --- & $1.08\times10^{+2}$ \\
64 & $1.25\times10^{-1}$ & $9.51\times 10^{-1}$ & 5.11 & $1.03\times 10^{-0}$ & 5.50 & $1.06\times 10^{0}$ & 5.01 & $5.70\times10^{-1}$ \\
128 & $6.25\times10^{-2}$ & $1.84\times 10^{-1}$ & 2.37 & $3.41\times 10^{-1}$ & 1.59 & $4.11\times 10^{-1}$ & 1.37 & $3.47\times10^{-2}$ \\
256 & $3.13\times10^{-2}$ & $5.44\times 10^{-2}$ & 1.76 & $1.07\times 10^{-1}$ & 1.67 & $1.50\times 10^{-1}$ & 1.45 & $5.82\times10^{-3}$ \\
512 & $1.56\times10^{-2}$ & $1.04\times 10^{-2}$ & 2.38 & $3.51\times 10^{-2}$ & 1.61 & $4.48\times 10^{-2}$ & 1.75 & $2.61\times10^{-3}$ \\
1024 & $7.81\times10^{-3}$ & $2.69\times 10^{-3}$ & 1.95 & $1.31\times 10^{-2}$ & 1.42 & $1.88\times 10^{-2}$ & 1.25 & $1.08\times10^{-3}$ \\
2048 & $3.91\times10^{-3}$ & $9.54\times 10^{-4}$ & 1.50 & $5.98\times 10^{-3}$ & 1.13 & $8.24\times 10^{-3}$ & 1.19 & $1.69\times10^{-4}$ \\
\bottomrule
\end{tabular}
\end{table}
%
%
%
%
%
%
%
%
%
%
%
\section{Conclusions}\label{sec:conclusions}
%
We have presented a mass-{\cR preserving two-step} Lagrange--Galerkin scheme of second order in time for convection-diffusion problems.
Its mass-{\cR preserving} property is achieved by the Jacobian multiplication technique, and its accuracy of second order in time is obtained based on the idea of the multistep Galerkin method along characteristics.
For the  first time step, we have proposed to employ a mass-{\cR preserving} scheme of first order in time. This construction is efficient and does not decrease the convergence orders in the $\ell^\infty(L^2)$- and $\ell^2(H^1_0)$-norms.
\par
Both main advantages of Lagrange--Galerkin methods, the CFL-free robustness for convection-dominated problems and the  symmetric and positive coefficient matrix of the resulting system of linear equations, are kept in our scheme.
Additionally, our scheme has a mass-{\cR preserving} property as proved in Theorem~\ref{thm:mass-conservation}.
We have proved unconditional stability without any stabilization parameter in Theorem~\ref{thm:stability}, and error estimates of second order in time in Theorem~\ref{thm:error_estimates}. For the error estimates two key lemmas on the truncation error analysis of the material derivative in conservative form, cf. Lemma~\ref{lem:truncation}, and a discrete Gronwall inequality for multistep methods, cf. Lemma~\ref{lem:gronwall}, have been prepared.
\par
We summarize the shown convergence orders as follows.
The order in the $\ell^\infty(L^2)\cap \ell^2(H^1_0)$-norm is $O(\Delta t^2 + h^k)$, and the order in the $\ell^\infty(L^2)$-norm is $O(\Delta t^2 + h^{k+1})$ if the duality argument can be employed.
We have also proved the convergence order $O(\Delta t^{3/2}+h^k)$ in the discrete $\ell^\infty(H^1_0)$- and  $H^1(L^2)$-norm, which will be useful when we apply the scheme to, e.g., the Navier--Stokes equations.
We have presented numerical results in one-, two- and three-dimensions, which have supported the theoretical convergence orders.
%
%
%
%
\section*{Acknowledgements}
This work was supported by JSPS KAKENHI Grant Numbers JP18H01135, JP19F19701, JP20H01823, JP20KK0058, and~JP21H04431, JST CREST Grant Number JPMJCR2014, and JST PRESTO Grant Number JPMJPR16EA. NK was supported by the JSPS Postdoctoral Fellowships for Research in Japan (Standard).
%
%
%
%
%
\appendix
\renewcommand{\thesection}{A}
\setcounter{Lem}{0}
\renewcommand{\theLem}{\thesection.\arabic{Lem}}
\setcounter{Rmk}{0}
\renewcommand{\theRmk}{\thesection.\arabic{Rmk}}
\setcounter{Cor}{0}
\renewcommand{\theCor}{\thesection.\arabic{Cor}}
\setcounter{figure}{0}
\renewcommand{\thefigure}{\thesection.\arabic{figure}}
\setcounter{equation}{0}
\makeatletter
    \renewcommand{\theequation}{%
    \thesection.\arabic{equation}}
    \@addtoreset{equation}{section}
  \makeatother
%
%
\section*{Appendix}
\subsection{Proof of Lemma~\ref{lem:gronwall}}
\label{subsec:proof_lem:gronwall}
From the assumption~\eqref{ieq:gronwall_hyp}, there exists a non-negative sequence~$\{\tilde{z}_n\}_{n\ge 2}$ such that
\[
\fz{1}{\Delta t}\Bigl( \fz{3}{2} x_n - 2 x_{n-1} + \fz{1}{2} x_{n-2} + y_n - y_{n-1} \Bigr) + \tilde{z}_n = a_0 x_n + a_1 x_{n-1} + a_2 x_{n-2} + b_n,\quad \forall n\ge 2,
\]
where $\tilde{z}_n$ satisfies
\[
z_n \le \tilde{z}_n, \quad \forall n\ge 2.
\]
Let $p$ and $q$, $p < q$, be the roots of quadratic equation $f(x) \defeq (3/2-a_0\Delta t) x^2 - (2+a_1\Delta t) x + (1/2 -a_2\Delta t) = 0$, and let $\lambda \defeq 2/3$ and $D \defeq (2+a_1\Delta t)^2 - (3-2a_0\Delta t)(1-2a_2\Delta t)$ $(= 1+2(a_0+2a_1+3a_2)\Delta t + (a_1^2-4a_0a_2)\Delta t^2)$.
The numbers $p$ and $q$ have the properties
\begin{align}
|p| & < 1 \le q, &
2\lambda & \le p+q, &
pq & \le \lambda, &
q - p 
& \ge \lambda, &
q^n - p^n & 
\le \exp(2 a_\ast n\Delta t) +1,
\label{ieq:p_q_properties}
\end{align}
which are obtained from
$f(1) = - (a_0+a_1+a_2)\Delta t \le 0$,
$f(-1) = 4 + (- a_0 + a_1 - a_2) \Delta t \ge {\cR 13/4} > 0$,
$a_0\Delta t \le 3/4$,
$1 \le D \le [1+(a_0+2a_1+3a_2)\Delta t]^2$,
and
\begin{align*}
q^n - p^n & \le q^n + |p|^n \le q^n + 1 
= \left( \fz{2+a_1\Delta t+\sqrt{D} }{3-2a_0\Delta t} \right)^n + 1 
\le \left( 1 + \fz{3 a_\ast \Delta t}{3-2a_0\Delta t} \right)^n + 1 \\
& \le ( 1 + 2 a_\ast \Delta t )^n {\cB +1}
\le \exp(2 a_\ast n\Delta t) {\cB + 1.}
\end{align*}
\par
Let $n\ge 2$ be fixed arbitrarily.
Then, we have
\begin{align*}
x_n - p x_{n-1} + \lambda (y_n - y_{n-1}) + \lambda\Delta t\tilde{z}_n & = q (x_{n-1}-p x_{n-2}) + \lambda\Delta t b_n, \\
x_n - q x_{n-1} + \lambda (y_n - y_{n-1}) + \lambda\Delta t\tilde{z}_n & = p (x_{n-1}-q x_{n-2}) + \lambda\Delta t b_n,
\end{align*}
which imply
\begin{subequations}
\begin{align}
x_n - p x_{n-1} + \lambda \Biggl[ \sum_{i=2}^n q^{n-i} y_i - \sum_{i=1}^{n-1} q^{n-1-i} y_i \Biggr] + \lambda\Delta t \sum_{i=2}^n q^{n-i} \tilde{z}_i & = q^{n-1} (x_1-p x_0) + \lambda\Delta t \sum_{i=2}^n q^{n-i} b_i, 
\label{eq:proof_gronwall_1} \\
x_n - q x_{n-1} + \lambda \Biggl[ \sum_{i=2}^n p^{n-i} y_i - \sum_{i=1}^{n-1} p^{n-1-i}y_i \Biggr] + \lambda\Delta t \sum_{i=2}^n p^{n-i} \tilde{z}_i & = p^{n-1} (x_1-q x_0) + \lambda\Delta t \sum_{i=2}^n p^{n-i} b_i.
\label{eq:proof_gronwall_2}
\end{align}
\end{subequations}
Multiplying~\eqref{eq:proof_gronwall_1} by~$q$ and~\eqref{eq:proof_gronwall_2} by~$p$ and subtracting the second equation from the first, we get
\begin{align}
(q-p) x_n & + \lambda \Biggl[ \sum_{i=2}^n (q^{n+1-i}-p^{n+1-i}) y_i - \sum_{i=1}^{n-1} (q^{n-i} - p^{n-i}) y_i \Biggr] + \lambda\Delta t \sum_{i=2}^n (q^{n+1-i}-p^{n+1-i}) \tilde{z}_i \notag \\
& = (q^n-p^n) x_1 - pq (q^{n-1}-p^{n-1}) x_0 + \lambda\Delta t \sum_{i=2}^n (q^{n+1-i} - p^{n+1-i}) b_i.
\label{eq:proof_gronwall_3}
\end{align}
It is noted here that
\begin{align}
\sum_{i=2}^n (q^{n+1-i} &- p^{n+1-i}) y_i - \sum_{i=1}^{n-1} (q^{n-i} - p^{n-i}) y_i \notag\\
& = (q-p) y_n + \sum_{i=2}^{n-1} \Bigl[ (q^{n+1-i}-p^{n+1-i}) - (q^{n-i} - p^{n-i}) \Bigr] y_i - (q^{n-1}-p^{n-1})y_1 \notag\\
& \ge (q-p) y_n - (q^{n-1}-p^{n-1})y_1,
\label{ieq:y-term}
\end{align}
where the following inequality has been employed:
\begin{align*}
q^{k+1}-p^{k+1} \ge q^k - p^k, \qquad \forall k\in \N\cup\{0\}.
\end{align*}
This inequality holds obviously from the first property in~\eqref{ieq:p_q_properties} for $p\ge 0$ or for $p < 0$ and an even number~$k$.
For $p < 0$ and an odd number~$k$, it is proved by induction, and the key inequality in the induction is
\begin{align*}
q^{k+2}-p^{k+2} = (q^{k+1}-p^{k+1})(p+q) -pq(q^k-p^k) \ge (q^{k+1}-p^{k+1})(p+q) \ge q^{k+1}-p^{k+1}.
\end{align*}
Combining~\eqref{ieq:p_q_properties} and~\eqref{ieq:y-term} with~\eqref{eq:proof_gronwall_3} {\cR and noting that $0 \le -pq/(q-p) \le q/(q-p) \le 1$ for $p \in (-1,0)$ and $-pq/(q-p) \le 0 < 1$ for $p \in [0,1)$,} we obtain
\begin{align*}
x_n + \lambda y_n + \lambda \Delta t \sum_{i=2}^n \tilde{z}_i 
& \le \fz{q^n-p^n}{q-p} \Biggl[ x_1 - pq x_0 + \lambda y_1 + \lambda\Delta t \sum_{i=2}^n b_i \Biggr] \notag\\
& \le \Bigl( \exp( 2 a_\ast n\Delta t ) + 1 \Bigr) \Biggl( x_0 + \fz{3}{2} x_1 + y_1+\Delta t\sum_{i=2}^n b_i \Biggr),
\end{align*}
which completes the proof.
\subsection{Proof of Lemma~\ref{lem:R}}
\label{subsection:lemma_R}
We prove~\textit{(i)}.
\cR
For the estimate~\eqref{ieq:eta_t_L2}, from the next calculations,
\begin{align*}
\eta(\cdot,t) 
& = \fz{1}{\Delta t} \int_{t^{n-1}}^{t^n} \eta(\cdot,t)\, ds 
= \fz{1}{\Delta t} \int_{t^{n-1}}^{t^n} \Bigl( \bigl[ \eta(\cdot,s_1) \bigr]_{s_1 = s}^t \Bigr) ds
+ \fz{1}{\Delta t} \int_{t^{n-1}}^{t^n} \eta(\cdot,s)\, ds \\
& = \fz{1}{\Delta t} \int_{t^{n-1}}^{t^n}ds \int_s^{t} \prz{\eta}{t}(\cdot,s_1)  \, ds_1 + \fz{1}{\Delta t} \Bigl[ \int_{t^{n-1}}^{t^n} \eta(\cdot,s)^2 ds \Bigr]^{1/2} \Bigl[ \int_{t^{n-1}}^{t^n} 1^2 ds \Bigr]^{1/2} \\
& \le \fz{1}{\Delta t} \int_{t^{n-1}}^{t^n} ds \Bigl[ \int_s^t \prz{\eta}{t} (\cdot,s_1)^2 ds_1 \Bigr]^{1/2} \Bigl[ \int_s^t 1^2 ds_1 \Bigr]^{1/2} + \fz{1}{\sqrt{\Delta t}} \Bigl[ \int_{t^{n-1}}^{t^n} \eta(\cdot,s)^2 ds \Bigr]^{1/2} \\
& \le \fz{1}{\Delta t} \int_{t^{n-1}}^{t^n} ds \Bigl[ \int_{t^{n-1}}^{t^n} \prz{\eta}{t} (\cdot,s_1)^2 ds_1 \Bigr]^{1/2} \Bigl[ \int_{t^{n-1}}^{t^n} 1^2 ds_1 \Bigr]^{1/2} + \fz{1}{\sqrt{\Delta t}} \Bigl[ \int_{t^{n-1}}^{t^n} \eta(\cdot,s)^2 ds \Bigr]^{1/2} \\
& = \sqrt{\Delta t} \Bigl[ \int_{t^{n-1}}^{t^n} \prz{\eta}{t} (\cdot,s)^2 ds \Bigr]^{1/2} + \fz{1}{\sqrt{\Delta t}} \Bigl[ \int_{t^{n-1}}^{t^n} \eta(\cdot,s)^2 ds \Bigr]^{1/2} \\
& \le \sqrt{\fz{2}{\Delta t}}\, \Bigl[ \int_{t^{n-1}}^{t^n} \Bigl( \eta(\cdot,s)^2 + \prz{\eta}{t} (\cdot,s)^2 \Bigr) ds \Bigr]^{1/2},
\end{align*}
and Lemma~\ref{lem:projection}-\textit{(i)}, we obtain the inequalities as
\begin{align*}
\| \eta(\cdot,t) \|
& \le \| \eta(\cdot,t) \|_{H^1(\Omega)}
\le \sqrt{2} \, \Delta t^{-1/2} \| \eta \|_{H^1(t^{n-1},t^n;H^1)} 
\le c h^k \Delta t^{-1/2} \| \phi \|_{H^1(t^{n-1},t^n;H^{k+1})} \\
& \le c^\prime h^k \| \phi \|_{H^2(0,T; H^{k+1})} \quad \mbox{(by the Sobolev embedding theorem with respect to time).}
\end{align*}
\cK
For the estimate~\eqref{ieq:D_dt_eta}, noting that
\begin{align}
\bigl\| \bar{D}_{\Delta t}^{(1)} \eta^n \bigr\|
& = \fz{1}{\Delta t} \|\eta^n - \eta^{n-1}\|
\le \fz{1}{\Delta t} \, \Bigl\| \bigl[ \eta(\cdot, s) \bigr]_{s=t^{n-1}}^{t^n} \Bigr\|
= \fz{1}{\Delta t} \Bigl\| \int_{t^{n-1}}^{t^n} \prz{\eta}{t} (\cdot, s) \, ds \Bigr\| \notag\\
& \le \fz{1}{\Delta t} \int_{t^{n-1}}^{t^n} \Bigl\| \prz{\eta}{t} (\cdot, s) \Bigr\| \, ds 
\le \fz{1}{\sqrt{\Delta t}} \, \Bigl\| \prz{\eta}{t} \Bigr\|_{L^2(t^{n-1},t^n; L^2(\Omega))} \notag\\
& \le \Delta t^{-1/2} \|\eta\|_{H^1(t^{n-1},t^n; L^2(\Omega))}
\le c h^k \Delta t^{-1/2} \|\phi\|_{H^1(t^{n-1},t^n; H^{k+1}(\Omega))},
\label{ieq:D_dt_1_eta_n}
\end{align}
for $n \ge 1$, we have
\begin{align*}
\bigl\| \bar{D}_{\Delta t}^{(2)} \eta^n \bigr\|
& = \Bigl\| \fz{3}{2}\bar{D}_{\Delta t}^{(1)} \eta^n - \fz{1}{2}\bar{D}_{\Delta t}^{(1)} \eta^{n-1} \Bigr\|
\le \fz{3}{2} \bigl\| \bar{D}_{\Delta t}^{(1)} \eta^n \bigr\| + \fz{1}{2} \bigl\| \bar{D}_{\Delta t}^{(1)} \eta^{n-1} \bigr\| \\
& \le c h^k \Delta t^{-1/2} \bigl( \|\phi\|_{H^1(t^{n-1},t^n; H^{k+1}(\Omega))} + \|\phi\|_{H^1(t^{n-2},t^{n-1}; H^{k+1}(\Omega))} \bigr) \\
& \le c^\prime h^k \Delta t^{-1/2} \|\phi\|_{H^1(t^{n-2},t^n; H^{k+1}(\Omega))}.
\end{align*}
Thus, we obtain~\eqref{ieq:D_dt_eta}.
From Lemma~\ref{lem:truncation} and Remark~\ref{rmk:truncation_first_order}, the estimate of~\eqref{ieq:R_h1_n} follows.
Since we have
\begin{align}
\fz{1}{\Delta t} \bigl\| \eta^0 - \eta^0 \circ X_1^1 \gamma^1 \bigr\|_{\Psi_h^\prime}
& \le \fz{1}{\Delta t} \bigl\| \eta^0 - \eta^0 \circ X_1^1 \bigr\|_{\Psi_h^\prime} + \fz{1}{\Delta t} \bigl\| \eta^0 \circ X_1^1 (1-\gamma^1) \bigr\|_{\Psi_h^\prime} \notag\\
& \le c_1 \bigl( \| \eta^0 \| + \| \eta^0 \circ X_1^1 \| \bigr) \quad \mbox{(by Lem.~\ref{lem:comp_funcs}-\eqref{ieq:v-vX_2} and $\|1-\gamma^1\|_{L^\infty(\Omega)} \le c_1\Delta t$)} \notag\\
& \le c_1^\prime \| \eta^0 \|
\quad \mbox{(by Lem.~\ref{lem:comp_funcs}-\eqref{ieq:vX_1})},
\label{ieq:eta0_minus_eta0_X_gamma}
\end{align}
the estimate~\eqref{ieq:R_h2_n} is obtained as
\begin{align}
\|R_{h2}^1\|_{\Psi_h^\prime}
& = \fz{1}{\Delta t} \bigl\| \eta^1 - \eta^0 \circ X_1^1 \gamma^1 \bigr\|_{\Psi_h^\prime} 
\le \bigl\| \bar{D}_{\Delta t}^{(1)} \eta^1 \bigr\|_{\Psi_h^\prime} + \fz{1}{\Delta t} \bigl\| \eta^0 - \eta^0 \circ X_1^1 \gamma^1 \bigr\|_{\Psi_h^\prime} \notag\\
& \le \bigl\| \bar{D}_{\Delta t}^{(1)} \eta^1 \bigr\| + c_1 \| \eta^0 \|
\le c_1^\prime h^k \Delta t^{-1/2} \|\phi\|_{H^1(t^0,t^1;H^{k+1})} \quad \mbox{(by~\eqref{ieq:D_dt_1_eta_n} and~\eqref{ieq:eta_t_L2})}, 
\label{ieq:R_h2_1_Psi_h_prime} \\
\|R_{h2}^n\|_{\Psi_h^\prime}
& = \fz{1}{2\Delta t} \bigl\| 3\eta^n-4\eta^{n-1}\circ X_1^n \gamma^n + \eta^{n-2}\circ \tilde{X}_1^n \tilde{\gamma}^n \bigr\|_{\Psi_h^\prime} \quad \mbox{(for $n \ge 2$)} \notag\\
& = \Bigl\| \fz{3}{2} \bar{D}_{\Delta t}^{(1)} \eta^n - \fz{1}{2} \bar{D}_{\Delta t}^{(1)} \eta^{n-1} + \fz{2}{\Delta t} \bigl( \eta^{n-1} - \eta^{n-1}\circ X_1^n \gamma^n \bigr) - \fz{1}{2\Delta t} \bigl( \eta^{n-2} - \eta^{n-2} \circ \tilde{X}_1^n \tilde{\gamma}^n \bigr) \Bigr\|_{\Psi_h^\prime} \notag\\
& \le \fz{3}{2} \bigl\| \bar{D}_{\Delta t}^{(1)} \eta^n \bigr\| + \fz{1}{2} \bigl\| \bar{D}_{\Delta t}^{(1)} \eta^{n-1} \bigr\| + \fz{2}{\Delta t} \bigl\| \eta^{n-1} - \eta^{n-1}\circ X_1^n \gamma^n \bigr\|_{\Psi_h^\prime} \notag\\
& \quad + \fz{1}{2\Delta t} \bigl\| \eta^{n-2} - \eta^{n-2} \circ \tilde{X}_1^n \tilde{\gamma}^n \bigr\|_{\Psi_h^\prime} \notag\\
& \le c_1 \bigl( \bigl\| \bar{D}_{\Delta t}^{(1)} \eta^n \bigr\| + \bigl\| \bar{D}_{\Delta t}^{(1)} \eta^{n-1} \bigr\| + \| \eta^{n-1} \| + \| \eta^{n-2} \| \bigr) \quad \mbox{(cf.~\eqref{ieq:eta0_minus_eta0_X_gamma})} \notag\\
& \le c_1^\prime h^k \Delta t^{-1/2} \|\phi\|_{H^1(t^{n-2},t^n; H^{k+1})} \quad \mbox{(by~\eqref{ieq:D_dt_1_eta_n} and~\eqref{ieq:eta_t_L2})}.
\label{ieq:R_h2_n_Psi_h_prime}
\end{align}
The estimate~\eqref{ieq:R_h3_n} is obvious from~\eqref{ieq:eta_t_L2}.
Using a similar evaluation to~\eqref{ieq:R_h2_n} with some modifications, we get~\eqref{ieq:R_h2_n_L2} by
\begin{align}
\|R_{h2}^1\|
& = \fz{1}{\Delta t} \bigl\| \eta^1 - \eta^0 \circ X_1^1 \gamma^1 \bigr\|
\le \bigl\| \bar{D}_{\Delta t}^{(1)} \eta^1 \bigr\| + \fz{1}{\Delta t} \bigl\| \eta^0 - \eta^0 \circ X_1^1 \bigr\| + \fz{1}{\Delta t} \bigl\| \eta^0 \circ X_1^1 (1-\gamma^1) \bigr\| \notag\\
& \le \bigl\| \bar{D}_{\Delta t}^{(1)} \eta^1 \bigr\| + c_1 \| \eta^0 \|_{H^1(\Omega)} \quad \mbox{(by Lem.~\ref{lem:comp_funcs}-\eqref{ieq:v-vX_1} and $\|1-\gamma^1\|_{L^\infty(\Omega)} \le c_1\Delta t$)} \notag\\
& \le c_1^\prime h^k \Delta t^{-1/2} \|\phi\|_{H^1(t^0,t^1;H^{k+1})},
\label{ieq:estimate_R_h2_1_L2} \\
\|R_{h2}^n\|
& \le \fz{3}{2} \bigl\| \bar{D}_{\Delta t}^{(1)} \eta^n \bigr\| + \fz{1}{2} \bigl\| \bar{D}_{\Delta t}^{(1)} \eta^{n-1} \bigr\| + \fz{2}{\Delta t} \bigl\| \eta^{n-1} - \eta^{n-1} \circ X_1^n \bigr\| + \fz{2}{\Delta t} \bigl\| \eta^{n-1} \circ X_1^n (1-\gamma^n) \bigr\| \notag\\
& \quad + \fz{1}{2\Delta t} \bigl\| \eta^{n-2} - \eta^{n-2} \circ \tilde{X}_1^n \bigr\| + \fz{1}{2\Delta t} \bigl\| \eta^{n-2} \circ \tilde{X}_1^n (1-\tilde{\gamma}^n) \bigr\| \quad \mbox{(for $n \ge 2$)} \notag\\
& \le c_1 \bigl( \bigl\| \bar{D}_{\Delta t}^{(1)} \eta^n \bigr\| + \bigl\| \bar{D}_{\Delta t}^{(1)} \eta^{n-1} \bigr\| + \| \eta^{n-1} \|_{H^1(\Omega)} + \| \eta^{n-2} \|_{H^1(\Omega)} \bigr) \quad \mbox{(cf.~\eqref{ieq:estimate_R_h2_1_L2})} \notag\\
& \le c_1^\prime h^k \Delta t^{-1/2} \|\phi\|_{H^1(t^{n-2},t^n; H^{k+1})}. \notag
\end{align}
\par
We prove~\textit{(ii)}.
The estimate~\eqref{ieq:eta_t_L2_with_hyp4} is obvious from Lemma~\ref{lem:projection}-\textit{(ii)}.
We evaluate~$\|R_{h2}^n\|_{\Psi_h^\prime}$.
Recalling the calculation of~$\|R_{h2}^n\|_{\Psi_h^\prime}$, cf.~\eqref{ieq:R_h2_1_Psi_h_prime} and~\eqref{ieq:R_h2_n_Psi_h_prime}, in the proof of~\textit{(i)}, and noting that
\begin{align*}
\bigl\| \bar{D}_{\Delta t}^{(1)} \eta^n \bigr\|
& \le \Delta t^{-1/2} \|\eta\|_{H^1(t^{n-1},t^n; L^2(\Omega))} 
\quad \mbox{(cf.~\eqref{ieq:D_dt_1_eta_n})} \notag\\
& \le c h^{k+1} \Delta t^{-1/2} \|\phi\|_{H^1(t^{n-1},t^n; H^{k+1}(\Omega))} \quad \mbox{(by Lem.~\ref{lem:projection}-\textit{(ii)})}
\end{align*}
for $n \ge 1$, we have the following estimates,
\begin{align}
\|R_{h2}^1\|_{\Psi_h^\prime}
& \le \bigl\| \bar{D}_{\Delta t}^{(1)} \eta^1 \bigr\| + c_1 \| \eta^0 \|
\le c_1^\prime h^{k+1} \Delta t^{-1/2} \|\phi\|_{H^1(t^0,t^1;H^{k+1})} \quad \mbox{(by~\eqref{ieq:eta_t_L2_with_hyp4})}, \notag\\
\|R_{h2}^n\|_{\Psi_h^\prime}
& \le c_1 \bigl( \bigl\| \bar{D}_{\Delta t}^{(1)} \eta^n \bigr\| + \bigl\| \bar{D}_{\Delta t}^{(1)} \eta^{n-1} \bigr\| + \| \eta^{n-1} \| + \| \eta^{n-2} \| \bigr) \quad \mbox{(for $n \ge 2$)} \notag\\
& \le c_1^\prime h^{k+1} \Delta t^{-1/2} \|\phi\|_{H^1(t^{n-2},t^n; H^{k+1})} \quad \mbox{(by~\eqref{ieq:eta_t_L2_with_hyp4})}, \notag
\end{align}
which complete the proof of~\eqref{ieq:R_h2_n_with_hyp4}.
The estimate of~\eqref{ieq:R_h3_n_with_hyp4} is obvious under Hypothesis~\ref{hyp:A-N} from Lemma~\ref{lem:projection}-\textit{(ii)}.
\subsection{Proof of Lemma~\ref{lem:eh1}}
\label{subsection:lemma_eh1}
We prove~\textit{(i)}.
From Lemma~\ref{lem:R}-\textit{(i)}, it holds that
\begin{align}
\|R_h^1\| 
& \le \Delta t \sum_{i=1}^3 \|R_{hi}^1\| 
\le c_1 \bigl( \Delta t^{1/2} \|\phi\|_{Z^2(t^0,t^1)} + h^k \Delta t^{-1/2} \|\phi\|_{H^1(t^0,t^1;H^{k+1})} \bigr) \notag\\
& \le c_1^\prime \bigl( \Delta t \|\phi\|_{Z^3} + h^k \|\phi\|_{H^2(H^{k+1})} \bigr)
\le c_1^{\prime\prime} (\Delta t + h^k) \|\phi\|_{Z^3\cap H^2(H^{k+1})}.
\label{ieq:R_h_first_step_L2}
\end{align}
The equation~\eqref{eq:error} with~$n=1$ is rewritten as
\begin{align}
\bigl( \bar{D}_{\Delta t}^{(1)} e_h^1,  \psi_h \bigr) + a_0( e_h^1, \psi_h ) =  \lA R_h^1, \psi_h \rA, \quad
\forall \psi_h \in \Psi_h
\label{eq:error_with_n_eq_1}
\end{align}
from $e_h^0 = 0$ and, therefore, $\frac{1}{\Delta t} (e_h^1 - e_h^0 \circ X_1^n\gamma^n) = \bar{D}_{\Delta t}^{(1)}e_h^1$.
Substituting $e_h^1$ into $\psi_h$ in~\eqref{eq:error_with_n_eq_1}, dropping the positive term $a_0( e_h^1, e_h^1 )$, and using $e_h^0 = 0$ and $\lA R_h^1, e_h^1 \rA \le \|R_h^1\| \|e_h^1\|$, we have
\begin{align}
\|e_h^1\|
& \le \Delta t \|R_h^1\| 
\le \Delta t \bigl[ c_1 (\Delta t + h^k) \|\phi\|_{Z^3\cap H^2(H^{k+1})} \bigr] \quad \mbox{(by~\eqref{ieq:R_h_first_step_L2})} \notag\\
& \le c_1^\prime (\Delta t^2 + h^{k+1}) \, \|\phi\|_{Z^3\cap H^2(H^{k+1})},
\label{ieq:e_h_first_step_L2}
\end{align}
where for the last inequality we have employed 
\[
\Delta t h^k \le \fz{1}{2} (\Delta t^2 + h^{2k}) \le \fz{1}{2} (\Delta t^2 + h^{k+1}),\quad k \ge 1, \ h \in {\cR (0,1).}
\]
Again, substituting $e_h^1$ into $\psi_h$ in~\eqref{eq:error_with_n_eq_1}, and using $e_h^0 = 0$ and $\lA R_h^1, e_h^1 \rA \le \|R_h^1\| \|e_h^1\|$, we have
\begin{align}
\|e_h^1\|^2 + \nu \Delta t \|\nabla e_h^1\|^2
& \le \Delta t \|R_h^1\| \|e_h^1\| \notag\\
& \le c_1 \Delta t (\Delta t + h^k) (\Delta t^2 + h^{k+1}) \|\phi\|_{Z^3\cap H^2(H^{k+1})}^2 \quad \mbox{(by~\eqref{ieq:R_h_first_step_L2} and~\eqref{ieq:e_h_first_step_L2})} \notag\\
& \le c_1^\prime \bigl[ (\Delta t^2 + h^{k+1}) \|\phi\|_{Z^3\cap H^2(H^{k+1})} \bigr]^2,
\label{ieq:proof_convergence_L2_eh1}
\end{align}
which implies~\eqref{ieq:eh1_L2}.
\par
Substituting $\bar{D}_{\Delta t}^{(1)}e_h^1$ into $\psi_h$ in~\eqref{eq:error_with_n_eq_1} and using the estimates,
\begin{align*}
\bigl( \bar{D}_{\Delta t}^{(1)}e_h^1, \bar{D}_{\Delta t}^{(1)}e_h^1 \bigr) & = \bigl\| \bar{D}_{\Delta t}^{(1)}e_h^1 \bigr\|^2, 
\notag\\
a_0 \bigl( e_h^1, \bar{D}_{\Delta t}^{(1)}e_h^1 \bigr) 
& \ge \fz{1}{\Delta t} \biggl( \fz{\nu}{2} \|\nabla e_h^1\|^2 - \fz{\nu}{2} \|\nabla e_h^0\|^2 \biggr)
= \fz{1}{\Delta t} \Bigl( \fz{\nu}{2} \|\nabla e_h^1\|^2 \Bigr),
\notag\\
\bigl\lA R_h^1, \bar{D}_{\Delta t}^{(1)}e_h^1 \bigr\rA & \le \bigl\|R_h^1\bigr\| \, \bigl\| \bar{D}_{\Delta t}^{(1)}e_h^1 \bigr\|
\le \fz{1}{2}\bigl\|R_h^1\bigr\|^2 + \fz{1}{2} \bigl\| \bar{D}_{\Delta t}^{(1)}e_h^1 \bigr\|
\end{align*}
we get
\begin{align}
\nu \|\nabla e_h^1\|^2 + \Delta t \bigl\| \bar{D}_{\Delta t}^{(1)}e_h^1 \bigr\|^2 
& \le \Delta t\| R_h^1 \|^2 
\le \Delta t \bigl[ c_1 (\Delta t + h^k) \|\phi\|_{Z^3\cap H^2(H^{k+1})} \bigr]^2 \quad \mbox{(by~\eqref{ieq:R_h_first_step_L2})} \notag\\
& \le c_1^\prime \bigl[ (\Delta t^{3/2} + h^k) \|\phi\|_{Z^3\cap H^2(H^{k+1})} \bigr]^2, \notag
\end{align}
which implies~\eqref{ieq:eh1_H1_D_dt_eh1_L2}.
%
%
%
%
%
\bibliographystyle{spmpsci}      
\bibliography{./ref}

\end{document}